%% file: thesis.tex
\documentclass[12pt,letterpaper]{report}

\input{extra_packages}

\usepackage{bookmark}
\newcommand{\thesisdedication}{To my mother.}

\usepackage{setspace}

\input{definitions}

\begin{document}
%% Produces a test "layout" page, for "debugging" purposes only.
%% Comment out for final version.
%\layout % requires package layout (see above, on this same file)

%%%%%% Title page %%%%%%%%%%%
%% Sets page numbering to "roman style" i, ii, iii, iv, etc:
\pagenumbering{roman}
%
%% No numbering in the title page:
\thispagestyle{empty}
\begin{center}
  {\large\textbf{\thesistitle}}
  \vspace{.7in}

  by
  \vspace{.7in}

  \thesisauthor
  \vfill

\begin{doublespace}
  A dissertation submitted in partial fulfillment\\
  of the requirements for the degree of\\
  Doctor of Philosophy\\
  Department of Mathematics\\
  New York University\\
  \graddate
\end{doublespace}
\end{center}
\vfill

% advisor, has to be in alphabetical order
\noindent\makebox[\textwidth]{\hfill\makebox[2.5in]{\hrulefill}}\\
\makebox[\textwidth]{\hfill\makebox[2.5in]{\hfill\thesiscoadvisor\hfill}}\\
% advisor, co-advisor
\vspace{20pt}\\
\noindent\makebox[\textwidth]{\hfill\makebox[2.5in]{\hrulefill}}\\
\makebox[\textwidth]{\hfill\makebox[2.5in]{\hfill\thesisadvisor\hfill}}
\newpage
%%%%%%%%%%%% Copyright page %%%%%%%%%%%%%%%%%
\thispagestyle{empty}
\begin{center}
  \copyright ~Arjun Krishnan. \\
  All rights reserved, 2014.
\end{center}
\newpage
%%%%%%%%%%%%% Blank page %%%%%%%%%%%%%%%%%%
\thispagestyle{empty}
\vspace*{0in}
\newpage

%%%%%%%%%%%%%% Dedication %%%%%%%%%%%%%%%%%
%% Comment out the following lines if you do not want to dedicate
%% this to anyone...
\section*{Dedication}\addcontentsline{toc}{section}{Dedication}
\vspace*{\fill}
\begin{center}
  \thesisdedication
\end{center}
\vfill
\newpage
%%%%%%%%%%%%%% Acknowledgements %%%%%%%%%%%%
%% Comment out the following lines if you do not want to acknowledge
%% anyone's help...
\section*{Acknowledgements}\addcontentsline{toc}{section}{Acknowledgements}
\input{acknowledge}
\newpage
%%%% Abstract %%%%%%%%%%%%%%%%%%
\section*{Abstract}\addcontentsline{toc}{section}{Abstract}
\input{abstract}
\newpage
%%%% Table of Contents %%%%%%%%%%%%
\tableofcontents

%%%%% List of Figures %%%%%%%%%%%%%
%% Comment out the following two lines if your thesis does not
%% contain any figures. The list of figures contains only
%% those figures included withing the "figure" environment.
\listoffigures\addcontentsline{toc}{section}{List of Figures}
\newpage

%%%%% List of Tables %%%%%%%%%%%%%
%% Comment out the following two lines if your thesis does not
%% contain any tables. The list of tables contains only
%% those tables included withing the "table" environment.
%\listoftables\addcontentsline{toc}{section}{List of Tables}
%\newpage

%%%%% Body of thesis starts %%%%%%%%%%%%
\pagenumbering{arabic} % switches page numbering to arabic: 1, 2, 3, etc.
%% Introduction. If your thesis has no introduction, or chapter 1 is
%% meant to be the introduction, then comment out the lines below.
%\section*{Introduction}\addcontentsline{toc}{section}{Introduction}
%\input{intro}
%% If your thesis has different "Parts", use commands such as the following:
\part{Homogenization theorem and variational formula\label{part:one}}%
\input{intro} %introduction
\input{preliminaries} % further chapters -- change file names to meaningful things...
\input{outline_of_proof}

%\part{Proof of homogenization theorem and variational formula\label{part:two}}%
%\input{proof_continuum_homogenization}
\input{proof_discrete_homogenization}

\input{proof_variational_formula}

\part{Applications and Discussion\label{part:two}}
%\addcontentsline{toc}{chapter}{I'm proof reading this and will be done in a few days}

%I'm proof reading this and will be done in a few days
\input{recap} %this will have to be removed once the references have been fixed.
\input{abstract_algorithm}
\input{comparing_two_distributions}
\bookmarksetup{startatroot}
\addtocontents{toc}{\smallskip} %this adds a small jump to the toc.
\input{future_work}
%%%%% Appendices start %%%%%%%%%%%%%%%%
%% Comment out the following line if your thesis has no appendix
\appendix
\input{miscellaneous_proofs}

%% Note: If your thesis has more than one appendix, NYU requires a "list of
%% appendices" page before the body of the thesis. I don't provide the tools
%% to create that here, so you're on your own for that one... Sorry.
%\input{app2}
%%%% Input bibliography file %%%%%%%%%%%%%%%
%\input{biblio}
%%fakesection
% this is the natbib command
% include so vim latexsuite can find the bib file and resources
\bibliographystyle{chicago}

\begin{onehalfspace}
  \bibliography{My_Zotero_Library.bib}
\end{onehalfspace}
\addcontentsline{toc}{chapter}{Bibliography}

\end{document}

%% file: extra_packages.tex
\usepackage[final]{graphicx}
\usepackage{indentfirst}

\usepackage{amsthm} % facilities for theorem-like environments
\usepackage[tbtags]{amsmath} % a lot of good stuff!

\usepackage{amsfonts}
\usepackage{amssymb}
\usepackage{amscd} % inputs the commutative diagram command, CD

\usepackage{todonotes}
\usepackage{tikz}
\usepackage{syntonly}
\usepackage[numbers,sort&compress]{natbib}

\usepackage[urlcolor=black,colorlinks=true,linkcolor=black,citecolor=black]{hyperref}

%% file: definitions.tex
\renewcommand{\emptyset}{\ensuremath{\varnothing}}

\newcommand{\Eqref}[1]{\eqref{#1}}
\newcommand{\Figref}[1]{Fig.~\ref{#1}}

\newcommand{\Secref}[1]{Section~\ref{#1}}
\newcommand{\Chapref}[1]{Chapter~\ref{#1}}
\newcommand{\Thmref}[1]{Theorem~\ref{#1}}
\newcommand{\Corref}[1]{Corollary~\ref{#1}}
\newcommand{\Lemref}[1]{Lemma~\ref{#1}}
\newcommand{\Propref}[1]{Prop.~\ref{#1}}
\newcommand{\Defref}[1]{Definition~\ref{#1}}
\newcommand{\Remref}[1]{Remark~\ref{#1}}
\newcommand{\Claimref}[1]{Claim~\ref{#1}}

\newcommand{\R}{\mathbb{R}} %real numbers
\newcommand{\Z}{\mathbb{Z}} %integers
\newcommand{\Q}{\mathbb{Q}} %rational numbers
\newcommand{\Prob}{\mathbb{P}} %probability 

\newcommand{\forevery}{\forall} %Because I keep getting it wrong
\newcommand{\e}{\epsilon}
\newcommand{\w}{\omega}
\newcommand{\W}{\Omega}
\newcommand{\Norm}[2]{\left\lVert{#1}\right\rVert_{#2}} %norm
\newcommand{\tr}{\textrm{tr}} %trace
\newcommand{\Leb}{\textrm{Leb}} % lebesgue measure
\DeclareMathOperator*{\esssup}{ess\,sup} %essential suprememum
\DeclareMathOperator*{\essinf}{ess\,inf} %essential infimum 
\newcommand{\AND}{\textrm{ and }} 
\newcommand{\OR}{\textrm{ or }}

\newcommand{\almostsurely}{\textrm{a.s}}
\newcommand{\sigmaalgebra}{$\sigma$-algebra}
\theoremstyle{plain}
\newtheorem{theorem}{Theorem}[chapter]
\newtheorem*{theorem*}{Theorem}
\newtheorem{lemma}[theorem]{Lemma}
\newtheorem*{lemma*}{Lemma}
\newtheorem{cor}[theorem]{Corollary}

\newtheorem{prop}[theorem]{Proposition}
\newtheorem*{prop*}{Proposition}

\theoremstyle{definition}
\newtheorem{define}[theorem]{Definition}
\newtheorem*{example*}{Example}
\newtheorem{claim}[theorem]{Claim}

\theoremstyle{definition}
\newtheorem{remark}[theorem]{Remark}
\newtheorem*{remark*}{Remark}

\makeatletter
\def\@wraptoccontribs#1#2{}
\makeatother

\newcommand{\thesistitle}{Variational formula for the time-constant of first-passage percolation}
\newcommand{\thesisauthor}{Arjun Krishnan}
\newcommand{\thesisadvisor}{Professor S.R.S Varadhan}
\newcommand{\thesiscoadvisor}{Professor Sourav Chatterjee}
\newcommand{\graddate}{May 2014}

\newcommand{\FPP}{first-passage percolation}
\newcommand{\FPT}{first-passage time}
\newcommand{\deltaapprox}{$\delta$-approximation}

%% file: acknowledge.tex
%% Write your acknowledgements in this file. If you do not want to acknowledge anyone,
%% you can delete this file and comment out the corresponding part in the "thesis.tex"
%% file.
%
I would first like to acknowledge the support I received from my advisors, Sourav Chatterjee and Raghu Varadhan. Sourav Chatterjee introduced me to first-passage percolation in the 2nd year of my PhD. He supported my work even when the ideas were extremely vague and nebulous, and has been very encouraging throughout. I received enormous amounts of help from Raghu Varadhan. He found several gaps in my arguments, made uncountably many helpful suggestions, and mentored me from the beginning of my PhD. 

Second, I'd like to acknowledge Bob Kohn's encouragement and input throughout my PhD. If he hadn't reassured me quite early on that my sketchy PDE based intuition about the problem might actually work, I wouldn't have pursued it. 

Third, I'd like to thank my friends Matan Harel and Behzad Mehrdad both for their valuable mathematical input, and for being a part of my PhDs Anonymous support group.

My family has been very supportive throughout my PhD. Special thanks to my aunt and uncle Harini and Sri, my aunt Jayashree, my cousin Mira, and my uncles Srinivasan and Venugopal. 

Last but not least, I deeply appreciate the love and support from my girlfriend Shirley Zhao. Her inexplicable tolerance for my particular brand of weirdness has been a great source of strength.

%% file: abstract.tex
%% Write your abstract here.
%
We consider first-passage percolation with positive, stationary-ergodic weights on the square lattice $\Z^d$. Let $T(x)$ be the first-passage time from the origin to a point $x$ in $\Z^d$. The convergence of the scaled first-passage time $T([nx])/n$ to the time-constant as $n$ tends to infinity can be viewed as a problem of homogenization for a discrete Hamilton-Jacobi-Bellman (HJB) equation. By borrowing several tools from the continuum theory of stochastic homogenization for HJB equations, we derive an exact variational formula for the time-constant. We then construct an explicit iteration that produces the minimizer of the variational formula (under a symmetry assumption), thereby computing the time-constant. The variational formula may also be seen as a duality principle, and we discuss some aspects of this duality.

%% file: intro.tex
\chapter{Introduction}
\label{sec:introduction}
\section{Overview}
\label{sec:intro-fpp}
First-passage percolation is a growth model in a random medium introduced by~\citet{hammersley_first-passage_1965}. Consider the nearest-neighbor directed graph on the cubic lattice $\Z^d$. We will define the model when the random medium consists of positive edge-weights attached to the edges of this graph. For the purposes of this paper, \FPP~is better thought of as an optimal-control problem. Define the set of \emph{control directions}
   \begin{equation}
     A := \{\pm e_1,\ldots,\pm e_d\},
     \label{eq:control-directions}
   \end{equation}
   where $e_i$ are the canonical unit basis vectors for the lattice $\Z^d$. Let $(\W,\mathcal{F},\Prob)$ be a probability space\footnote{We will frequently drop reference to the probability space when it plays no role in the argument.}. The weights will be given by a function $\tau\colon \Z^d \times A \times \W \to \R$, where $\tau(x,\alpha,\w)$ refers to the weight on the edge from $x$ to $x+\alpha$. We assume that the function $\tau(x,\alpha,\w)$ is stationary-ergodic (see~\Secref{sec:results-from-stochastic-homogenization}) under translation by $\Z^d$. 
   
  A path connecting $x$ to $y$ is a finite ordered set of nearest-neighbor vertices:
\begin{equation}
  \gamma_{x,y} = \{x=v_1,\ldots,v_n=y\}.
  \label{eq:generic-path-gamma-as-set-of-vertices}
\end{equation}
The weight or total time of the path is  
\[ 
  \mathcal{W}(\gamma_{x,y}) := \sum_{i=1}^{n-1} \tau(v_i,v_{i+1}-v_i,\w).
\]
The first-passage time from $x$ to $y$ is an infimum of the total time of the path taken over all paths from $x$ to $y$: 
\begin{equation}
  \mathcal{T}(x,y) := \inf_{\gamma_{x,y}} \mathcal{W}(\gamma_{x,y}).
  \label{eq:definition-of-first-passage-time}
\end{equation}
We will use $\mathcal{T}(x)$ to mean $\mathcal{T}(x,0)$ unless otherwise specified. We're interested in the first-order asymptotics of $\mathcal{T}(x)$ as $|x| \to \infty$.

For any $x \in \R^d$, define the scaled~\FPT
\begin{equation}
  \mathcal{T}_n(x) := \frac{\mathcal{T}([nx])}{n}, 
  \label{eq:n-scaling-first-passage-time}
\end{equation}
where $[nx]$ represents the closest lattice point to $nx$ (with some uniform way to break ties).
The law of large numbers for $\mathcal{T}(x)$ has been the subject of a lot of research over the last $50$ years and involves the existence of the so-called \emph{time-constant} $m(x)$ given by
\begin{equation}
  m(x) := \lim_{n \to \infty} \mathcal{T}_n(x).
\label{eq:time-constant}
\end{equation}
The limit certainly exists in $d=1$, since it is simply the usual law of large numbers. For general $d \geq 1$, Kingman's classical subadditive ergodic theorem~\citep{kingman_ergodic_1968} along with some simple estimates is enough to show the existence of $m(x)$ for all $x \in \R^d$. However, the theorem is merely an existence theorem; i.e., it does not give us any quantitative information about the limit, unlike the usual law of large numbers. Proving something substantial about the time-constant has been an open problem for the last several decades. 

We prove that the time-constant satisfies a Hamilton-Jacobi-Bellmann (HJB) partial differential equation (PDE), and thus derive a new variational formula for the time-constant in part I. In part II, we first present a new explicit algorithm to produce a minimizer of the formula. Then, we discuss some aspects of the formula as a duality principle.
 
\section{First-passage percolation as a homogenization problem}
\label{sec:fpp-as-homo-problem}
Since the \FPT~$\mathcal{T}(x)$ is an optimal-control problem (see~\Chapref{chap:preliminaries}), it has a dynamic programming principle (DPP) which says that
  \[ \mathcal{T}(x) = \inf_{\alpha \in A} \{ \mathcal{T}(x+\alpha) + \tau(x,\alpha) \}. \]
  We can rewrite the DPP as a difference equation in the so-called \emph{metric} form of the HJB equation. Assuming $\tau(x,\alpha)$ are positive, 
\begin{equation}
  \sup_{\alpha \in A} \left\{ -\frac{(\mathcal{T}(x+\alpha) - \mathcal{T}(x))}{\tau(x,\alpha)} \right\} = 1. 
  \label{eq:discrete-HJB}
\end{equation}
Let's imagine that we were somehow able to extend $\mathcal{T}(x)$ as a smooth function on $\R^d$. Taylor expand $\mathcal{T}(x)$ at $[nx]$ to get
\begin{equation}
  \sup_{\alpha \in A} \left\{ -\frac{D\mathcal{T}([nx])\cdot \alpha + 1/2 (\alpha,D^2\mathcal{T}(\xi)\alpha)}{\tau([nx],\alpha)} \right\} = 1, 
  \label{eq:discrete-HJB-manipulation-into-homo-problem}
\end{equation}
where $(\cdot,\cdot)$ is the usual inner product on $\R^d$, and $\xi$ is a point in $\R^d$. Introduce the scaled \FPT~$\mathcal{T}_n(x)$ into~\Eqref{eq:discrete-HJB-manipulation-into-homo-problem} to get
\begin{equation}
  \sup_{\alpha \in A} \left\{ -\frac{D\mathcal{T}_n(x)\cdot \alpha}{\tau([nx],\alpha)} \right\} + O(n^{-1})= 1. 
  \label{eq:discrete-stochastic-homo-problem-from-intro}
\end{equation}
Equation~\eqref{eq:discrete-stochastic-homo-problem-from-intro} is reminiscent of a stochastic homogenization problem for a metric HJB equation in $\R^d$.   

By considering the lattice to be embedded in $\R^d$, we can view the path $\gamma_{x,y}$ in~\Eqref{eq:generic-path-gamma-as-set-of-vertices} as a continuous curve moving along the edges of the lattice from $x$ to $y$. Let $g_{x,y}(s)$ be a parametrization of this path satisfying
\[ 
  g_{x,y}'(s) = \frac{1}{\tau(z,\alpha)} \alpha,
\]
when $g_{x,y}(s) = z + \lambda \alpha$ for $z \in \Z^d,~0 < \lambda < 1, \AND \alpha \in A$. It's clear that if $g_{x,y}(0) = x$, 
%The slope is discontinuous at lattice points, but this will not matter for us. 
\[ 
  g_{x,y} \left( \sum_{i=1}^{n-1} \tau(v_i,v_{i+1})\right) = y. 
\]

Motivated by this interpretation, we can formulate a continuous version of \FPP~in $\R^d$. Let $t\colon\R^d \times A \to \R$ be a Lipschitz function in $\R^d$ (uniformly in $A$) satisfying for some constants $a, b > 0$
\[
  a \leq t(x,\alpha) \leq b.
\]
Let the set of allowable paths be 
% I could define the C^{0,1} here, but then I would have to state that the condition holds only where the g is differentiable
\[ 
  \mathcal{A} := \left\{ g \in C^{0,1}([0,\infty),\R^d) :
    g'(s) = \frac{\alpha}{t(g(s),\alpha)}~a.e.~s \in [0,\infty), ~\alpha \in A \right\}.
\]
Define the continuous version of the~\FPT~as
\begin{equation}
  T(x) := \inf_{g \in \mathcal{A}} \{s : g(0) = x,~ g(s)=0\}.
  \label{eq:continuousFPT-definition}
\end{equation}
Define the Hamiltonian for this continuous \FPP~to be
\begin{equation}
  H(p,x) := \sup_{\alpha \in A} \frac{p\cdot \alpha}{t(x,\alpha)}.
  \label{eq:H-for-general-continuous-FPP}
\end{equation} 

It's a classical fact in optimal-control theory that $T(x)$ is the (unique) viscosity solution of the metric HJB equation~\citep{bardi_optimal_1997} 
\begin{equation}
  \begin{split}
    H(DT(x),x) & = 1, \\
    T(0)  & = 0.
  \end{split}
  \label{eq:continuousFPT-HJB-eqn-for}
\end{equation}
Let $T_n(x) = T(nx)/n$ be the scaled continuous \FPT. For each $n$, $T_n(x)$ solves 
\begin{equation}
  \begin{split}
    H(DT_n(x),nx) & = 1,\\
    T_n(0) & = 0.
  \end{split}
  \label{eq:general-homogenization-problem-description}
\end{equation}
The set of equations in~\Eqref{eq:general-homogenization-problem-description} constitute a homogenization problem for the Hamiltonian in~\Eqref{eq:H-for-general-continuous-FPP}. The theory of stochastic homogenization states that $T_n(x) \to m(x)$ locally uniformly, and further, that there is a deterministic Hamiltonian $\overline{H}(p)$ such that $m(x)$ is the viscosity solution of
\begin{equation}
  \begin{split}
    \overline{H}(Dm) & = 1,\\
    m(0) & = 0.
  \end{split}
  \label{eq:limiting-HJB}
\end{equation}
Importantly, one can characterize $\overline{H}(p)$ using a variational formula. 
%Under appropriate assumptions on $t(x,\alpha)$,

We will first prove that the time-constant of discrete \FPP~satisfies a HJB equation of the form~\Eqref{eq:limiting-HJB}. Proving that a continuous, but possibly non-smooth function like the time-constant is a solution of a HJB equation is most easily done using viscosity solution theory~\citep{crandall_users_1992}. However, this is a continuum theory, and \FPP~is on the lattice. Constructing a continuous version of \FPP~allows us to embed the discrete problem in $\R^d$, and borrow the tools we need from the continuum theory. 

%%% stochastic homogenization
\section{Stochastic homogenization on $\R^d$}
\label{sec:results-from-stochastic-homogenization}
Fairly general stochastic homogenization theorems about HJB equations have been proved in recent years. We will state a special case of the theorem from~\citet{lions_homogenization_2005} that is relevant to our problem, although the later paper by~\citet{armstrong_stochastic_2012} would have been just as appropriate. 

For a group $G$, let 
\begin{equation}
  \mathcal{G} := \{V_{g}:\Omega \to \Omega\}_{g \in G} 
  \label{eq:translation-group-of-operators}
\end{equation}
be a family of invertible measure-preserving maps satisfying
\[ V_{gh} = V_g \circ V_h \quad \forall~g,h \in G. \]
That is, $V_{\cdot}$ is a homomorphism from $G$ to the group of all measure-preserving transformations on $(\W,\mathcal{F},\Prob)$. In our case, $G$ will either be $\Z^d \OR \R^d$. Let $X = \R^d \OR \Z^d$, and suppose $G \subset X$. A random function $f\colon X \times \Omega \to \R$ is said to be stationary with respect to $G$ if it satisfies
\begin{equation}
  f(x + g,\omega) = f(x,V_g\omega) \quad \forall~ x \in X,~g \in G.
  \label{eq:stationary-process-definition}
\end{equation}

We say $B \in \mathcal{F}$ is an invariant set if it satisfies $V_g B = B$ for any $g \in G\backslash \{e\}$ where $e$ is the identity element of $G$. The family of maps $\mathcal{G}$ is called (strongly) ergodic if invariant sets are either null or have full measure. A process $f(x,\w)$ is called stationary-ergodic if it's stationary with respect to a group $G$, and $\mathcal{G}$ is ergodic.
%\begin{remark}
%
%  
%  Incorporating this assumption is just a matter of using a stronger version of the subadditive ergodic theorem like~\citet{akcoglu_ergodic_1981}. We require the subadditive ergodic theorem along all non-trivial subgroups of $\mathcal{G}$ (every direction), and hence we would assume that every subgroup is ergodic. I initially cited~\citet{kingman_ergodic_1968}, and will systematically change this when reviewing the paper.
%\end{remark}

Let $G=\R^d$, and suppose the Hamiltonian $H\colon\R^d \times \R^d \times \W \to \R$ is
\begin{enumerate}
  \item stationary-ergodic,
  \item convex in $p$ for each $x$ and $\w$,
  \item coercive in $p$; i.e., uniformly in $x$ and $\w$, 
    \[ \lim_{|p| \to \infty} H(p,x,\w)  = +\infty, \]
  \item and regular; i.e., for each $\w$,
 \[ 
   H(\cdot,\cdot,\w) \in C^{0,1}_{loc}(\R^d \times \R^d) \cap C^{0,1}(\overline{B(0,R)} \times \R^d).
 \]    
\end{enumerate}
Consider the homogenization problem in~\Eqref{eq:general-homogenization-problem-description} for $T_n(x)$. The following is a special case of Theorem $3.1$,~\citet{lions_homogenization_2005}: 
\begin{theorem}
  There exists a deterministic, convex, Lipschitz $\overline{H}(p)$ with viscosity solution $m(x)$ of ~\Eqref{eq:limiting-HJB}, such that $T_{n}(x) \to m(x)$ locally uniformly in $\R^d$.
%bounded uniformly continuous 
  \label{thm:lions-homogenization-theorem}
\end{theorem}

There is also a variational characterization of $\overline{H}(p)$. Define the set of functions with stationary and mean-zero gradients:
\begin{equation}
  S:= \left\{ 
  f(\cdot,\w) \in C^{0,1}(\R^d) ~\left|~ 
  \begin{split}  
      & Df(x + z,\w) = Df(x,V_z\w), ~\forall x,z \in \R^d \\
      & E[Df(x,\w)] = 0 ~\forall x \in \R^d 
    \end{split} \right. \right\} .
  \label{eq:lions-S-set-definition-translation-group-R}
\end{equation}
Proposition $3.2$ from~\citet{lions_homogenization_2005} states that 
\begin{prop}
  For each $p \in \R^d$, 
  \begin{equation}
    \overline{H}(p) = \inf_{f \in S} \esssup_{\w} \sup_{x \in \R^d} H(Df + p,x,\w).
    \label{eq:variational-characterization-of-H-bar-discrete-translation-group}
  \end{equation}
  \label{prop:lions-variational-characterization-of-h-bar}
\end{prop}
To apply~\Thmref{thm:lions-homogenization-theorem} and~\Propref{prop:lions-variational-characterization-of-h-bar} to \FPP, we show that these results apply when $G = \Z^d$ by making the necessary minor modifications (see~\Secref{sec:continuum-homogenization-with-G-equals-Zd} and~\Secref{sec:proof-of-homogenization}). 
%(or to any other lattice)

%Theorem~\ref{thm:lions-homogenization-theorem} needs no modification other than citing the appropriate ergodic theorem and using a continuity estimate in~\citet{lions_homogenization_2005}. Although we will have no direct use for it, we also show the minor modification necessary to prove~\Propref{prop:lions-variational-characterization-of-h-bar} when $G = \Z^d$.

%% Main results
\section{Main results}
\label{sec:main-results}
%Let $\mathcal{G} := \{V_g\}_{g \in G}$ be a measure preserving group of functions on $\W$.
Our first result is the homogenization theorem for the time-constant of discrete \FPP. Let the edge-weights $\tau\colon\Z^d \times A \times \W \to \R$ be 
\begin{enumerate}
  \item (essentially) bounded above and below; i.e.,
    \begin{equation}
      \begin{split}
        0 <~& a = \essinf_{x,\alpha,\w} \tau(x,\alpha,\w), \\
          & b = \esssup_{x,\alpha,\w} \tau(x,\alpha,\w) < \infty, 
      \label{eq:basic-assumption-on-edge-weights}
    \end{split}
    \end{equation}
    and
  \item stationary-ergodic with $G=\Z^d$.
\end{enumerate}
%Then,
\begin{theorem}
  The time-constant $m(x)$ solves a Hamilton-Jacobi equation
  \begin{equation} 
    \begin{split}
      \overline{H}(Dm(x)) & = 1,\\
      m(0)     & = 0.
    \end{split} 
    \label{eq:time-constant-satisfies-a-PDE}
  \end{equation}
    \label{thm:-equation-for-mu}
\end{theorem}

The next result is a discrete variational formula for $\overline{H}(p)$. 
\begin{define}[Discrete derivative]
  For a function $\phi\colon\Z^d \to \R$, let
  \[ \mathcal{D}\phi(x,{\alpha}) = \phi(x+\alpha) - \phi(x) \]
    be its discrete derivative at $x \in \Z^d$ in the direction $\alpha \in A$. 
\end{define}

Let the discrete Hamiltonian for \FPP~ be
\begin{equation}
  \mathcal{H}(\phi,p,x,\w) = \sup_{\alpha \in A} \left\{ \frac{-\mathcal{D}\phi(x,\alpha,\w) - p\cdot \alpha}{\tau(x,\alpha,\w)} \right\}, 
  \label{eq:discrete-hamiltonian-1}
\end{equation}
and define the discrete counterpart of the set~\Eqref{eq:lions-S-set-definition-translation-group-R} 
\begin{equation}
  S:= \left\{ 
    \phi:\Z^d \times \Omega \to \R ~\left|~ 
  \begin{split}  
      & \mathcal{D}\phi(x + z,\w) = \mathcal{D}\phi(x,V_z\w), ~\forall x,z \in \Z^d \\
      & E[\mathcal{D}\phi(x,\alpha)] = 0 ~\forall x \in \Z^d \AND \alpha \in A ,
\end{split} 
  \label{eq:discrete-set-S-for-variationalf-formula}
\right\} \right.
\end{equation}
where $A$ is defined in~\Eqref{eq:control-directions}. Then,
\begin{theorem}
 the limiting Hamiltonian $\overline{H}(p)$ is given by
    \begin{equation}
    %  \begin{split}
    \overline{H}(p) = \inf_{\phi \in S} \esssup_{\w \in \W} \sup_{x \in \Z^d} \mathcal{H}(\phi,p,x,\w) . 
    \label{eq:original-statement-of-discrete-variational-formula-in-intro}
    %  \end{split}
    \end{equation}
    \label{thm:discrete-variational-formula-for-H-bar}
\end{theorem}
The variational formula tells us that $\overline{H}(p)$ is positive $1$-homogeneous, convex, and $\overline{H}(p) = 0$ iff $p = 0$. This means that it is a norm on $\R^d$, and indeed, the same is true of $m(x)$. By an elementary Hopf-Lax formula for the PDE in~\Eqref{eq:time-constant-satisfies-a-PDE}, we find that
\begin{cor}
  $\overline{H}(p)$ is the dual norm of $m(x)$ on $\R^d$, defined as usual by
  \[ 
    \overline{H}(p) = \sup_{m(x) = 1} p\cdot x .
  \]
  \label{cor:H-is-the-dual-norm-of-time-constant}
\end{cor}
In part II we present a new explicit algorithm that produces a minimizer of the variational formula under a symmetry assumption. It also contains a discussion of the formula as a duality principle. 

%% background
\section{Background on the time-constant}
We give a brief overview of results about the time-constant in \FPP. In this section, unless otherwise specified, we will assume that the nearest-neighbor graph on $\Z^d$ is undirected and that the edge-weights are i.i.d. \citet{cox_limit_1981} proved a celebrated result about the relationship between the time-constant and the so-called limit-shape of~\FPP. Let 
\begin{equation}
  R_t := \{x \in \R^2 : \mathcal{T}([x]) \leq t \}
  \label{eq:percolation-shape-defn}
\end{equation}
be the \emph{reachable set}. It is a fattened version of the sites reached by the percolation before time $t$. We're interested in the limiting behavior of the set $t^{-1} R_t$ as $t \to \infty$. 

Let $F(t)$ be the cumulative distribution of the edge-weights. Define the distribution $G$ by $(1-G(t)) = (1-F(t))^4$. The following theorem holds iff the second moment of $G$ is finite.   
\begin{theorem*}[\citet{cox_limit_1981}]
  Fix any $\e > 0$. If $m(x) > 0$ for all $x \in \R^2$,
  \begin{equation}
    \{x:m(x) \leq 1 - \e \} \subset \frac{R_t}{t} \subset \{x : m(x) \leq 1 + \e\} \textrm{ as } t \to \infty  ~\almostsurely.
    \label{eq:limit-shape-result-of-cox}
  \end{equation}
  Otherwise $m(x)$ is identically $0$, and for every compact $K \subset \R^2$,  
    \[ 
    K \subset \frac{R_t}{t} \textrm{ as } t \to \infty ~\almostsurely. 
    \]
\end{theorem*}
Under the conditions of the above theorem, the sublevel sets of the time-constant 
\[
  B_0 = \{x : m(x) \leq 1 \}
\]
can be thought of as the limit-shape. The extension of the~\citet{cox_limit_1981} theorem to $\Z^d$ was shown by~\citet{kesten_aspects_1986}. \citet{boivin_first_1990} proved the result for stationary-ergodic media. Despite these strong existence results on the time-constant and limit-shape, surprisingly little else is known in sufficient generality~\citep{van_den_berg_inequalities_1993}. 

The following is a selection of facts that are known about the time-constant. It's known that $m(e_1) = 0$ iff $F(0) \geq p_T$, where $p_T$ is closely related to the critical probability for bond percolation on $\Z^d$~\citep{kesten_aspects_1986}\footnote{It's the largest $p$ such that the expected size of the cluster containing the origin is finite.}. \citet{durrett_shape_1981} described an interesting class of examples where $B_0$ has flat-spots. \citet{marchand_strict_2002} and subsequently,~\citet{auffinger_differentiability_2013} have recently explored several aspects of this class of examples in great detail. It's also known that if $F$ is an exponential distribution, $B_0$ is not a Euclidean ball in high-enough dimensions~\citep{kesten_aspects_1986}. Exact results for the limit-shape are only available for ``up-and-right'' directed percolation with special edge weights~\citep{seppalainen_exact_1998,johansson_shape_2000}. In fact,~\citet{johansson_shape_2000} not only obtains the limit-shape, but also shows
\[
  \mathcal{T}(x) \sim m(x) + |x|^{1/3} \xi,
\]
where $\xi$ is distributed according to the (GUE) Tracy-Widom distribution. Hence, \FPP~is thought to be in the KPZ universality class.

%Some bounds are known about the time-constant: for example, $m(e_1) \leq E[\tau]$ follows easily from subadditivity~\citep{hammersley_first-passage_1965}. \citet{hammersley_first-passage_1966} obtained an upper bound for the diameter of the limit-shape in terms of the connective constant of the graph, (a lower bound for the time-constant). \citet{janson_upper_1981} obtained lower bounds for the time-constant for $\Z^2$ in terms of the exponential moment generating function of $F$ and the connective constant of the graph. Numerical results and improvements can be found in~\citet{alm_lower_2002}. Some comparison results and inequalities are known~\citep{van_den_berg_inequalities_1993,marchand_strict_2002}.  

%We will assume throughout this paper that $0 < a \leq \tau_e \leq b$. Restricting $\tau_e$ to $[a,b]$ can be viewed as a trunctation. If, say, we obtain bounds on $m$ that are independent of $a$ and $b$, they can be transferred over to general distributions by invoking the continuity of $m$ with respect to the weak topology on $F$ (at least for i.i.d percolation)~\citep{kesten_aspects_1986}.

Several theorems can be proved assuming properties of the limit-shape. For example, results about the fluctuations of $\mathcal{T}(x)$ can be obtained if it's known that the limit-shape has a ``curvature'' that's uniformly bounded~\citep{auffinger_differentiability_2013,newman_surface_1995}.~\citet{chatterjee_central_2013} prove Gaussian fluctuations for \FPP~in thin-cylinders under the hypothesis that the limit-shape is strictly convex in the $e_1$ direction. Properties like strict convexity, regularity or the curvature of the limit-shape have not been proved, and are of great interest.

%Experts believe that if the edge-weight distribution does not have atoms, the limit-shape is strictly convex. 
We suggest the lecture notes of~\citet{kesten_aspects_1986}, and the review papers by~\citet{grimmett_percolation_2012}, and~\citet{blair-stahn_first_2010} for a more exhaustive survey of the many aspects of first-passage percolation. 

\section{Background on stochastic homogenization}
Stochastic homogenization has been an active field of research in recent years, and there have been several results and methods of proof. Periodic homogenization of HJB equations was studied first by~\citet{lions_homogenization_1987}. The first results on stochastic HJB equations were obtained by~\citet{souganidis_stochastic_1999}, and~\citet{rezakhanlou_homogenization_2000}. These and other results~\citep{schwab_stochastic_2009} were about the ``non-viscous'' problem, and require super-linear growth of the Hamiltonian. That is, for positive constants $C_1,C_2,C_3,C_4 > 0 \AND \alpha_1,\alpha_2 > 1$,
\begin{equation}
  C_1 |p|^{\alpha_1} - C_2 \leq H(p,x,\omega) \leq C_3 |p|^{\alpha_2} + C_4.
  \label{eq:H-assumption-superlinear-growth}
\end{equation}
for all $p,x \in \R^d$ and $\almostsurely~\omega$. These results do not apply directly to our situation since we have exactly linear growth in~\Eqref{eq:H-for-general-continuous-FPP}. 

The ``viscous'' version of the problem includes a second-order term:
\begin{equation}
  \begin{split}
   -\e~\tr A(\e^{-1}x,\w)D^2u^{\e}(x,t,\omega) + H(Du^{\e},\e^{-1}x,\omega) = 1,\\
 \end{split}
  \label{eq:HJ-viscous-eqn-time-dependent}
\end{equation}
where $A(x,\omega)$ is a symmetric matrix. This problem is considered in~\citet{kosygina_stochastic_2006},~\citet{caffarelli_homogenization_2005},~\citet{lions_homogenization_2005}, and~\citet{lions_stochastic_2010}.~\citet{caffarelli_homogenization_2005} and~\citet{kosygina_stochastic_2006} require uniform ellipticity of the matrix $A$; i.e., they assume $\exists~ \lambda_1,~\lambda_2 >0$ such that for all $x \AND \w$,
\begin{equation}
  \lambda_1 |\xi|^2 \leq (A(x,\omega)\xi,\xi) \leq \lambda_2 |\xi|^2.
  \label{eq:A-assumption-uniform-ellipticity}
\end{equation}

\citet{lions_homogenization_2005} allow for $A=0$ (degenerate-ellipticity), and only linear growth of the Hamiltonian. Their method relies heavily on the optimal-control interpretation, and we were therefore able to borrow several ideas from them. The variational formula for the time-constant is a discrete version of theirs. We must also mention the work of~\citet{armstrong_stochastic_2012} that focuses specifically on metric Hamiltonians like the one for \FPP. In fact, it was brought to our attention that~\citet{armstrong_error_2012} made the following observation: since $\mathcal{T}(x,y)$ induces a random metric on the lattice, it's reasonable to believe that there ought to be some relation to metric HJB equations. This is exactly what we prove.

\section{Other variational formulas}
Once we posted our preprint~\citep{krishnan_variational_2013-1} on the arXiv, the concurrent but independent work of~\citet{georgiou_variational_2013} appeared. They prove discrete variational formulas for the directed polymer model at zero and finite temperature, and for the closely related last-passage percolation model. Their ideas originate in the works of~\citet{rosenbluth_quenched_2008},~\citet{rassoul-agha_quenched_2014} and~\citet{rassoul-agha_quenched_2013} for quenched large-deviation principles for random-walk in random environment. 

It is interesting to note that quite coincidentally, our results and those of~\citet{georgiou_variational_2013} almost exactly parallel the development of stochastic homogenization results in the continuum. \citet{lions_homogenization_2005} published their viscous homogenization results in 2005, using the classical cell-problem idea and the viscosity solution framework. Concurrently and independently in 2006,~\citet{kosygina_stochastic_2006} published their viscous stochastic homogenization result. In contrast to~\citet{lions_homogenization_2005}, their proof technique has the flavor of a duality principle and has a minimax theorem at its core. Both our results and those of~\citet{georgiou_variational_2013} are discrete adaptations of~\citet{lions_homogenization_2005} and~\citet{kosygina_stochastic_2006} respectively. 
% optimal-control theory

%% file: preliminaries.tex
\chapter{Setup and notation}
\label{chap:preliminaries}
\section{Continuum optimal-control problems}
\label{sec:optimal-control-theory}
We introduce the classical optimal-control framework here, since it plays a major-role in our proofs. We follow~\citet{bardi_optimal_1997} and~\citet{evans_partial_1998} for the setup. The evolution of the state of a control system $g(s)$ is governed by a system of ordinary differential equations 
\begin{equation}
  \begin{split}
    & g'(s) = f(g(s),a(s)),\\
    & g(0) = x,
  \end{split}
  \label{eq:control-ode}
\end{equation}
where $a(s)$ is known as the control. $A$ is typically a compact subset of a topological space like the one in~\Eqref{eq:control-directions}, and the space of allowable controls consists of all measurable functions 
\begin{equation}
  \mathcal{A} := \{ a\colon\R^+ \to  A \}.
  \label{eq:space-of-all-measurable-controls}
\end{equation}
The function $f(y,\alpha)$ is assumed to be bounded and Lipschitz in $y$ (uniformly in $\alpha$). Hence for fixed $a \in \mathcal{A}$,~\Eqref{eq:control-ode} has a unique (global) solution $g_{a,x}(s)$. Define the total cost to be
\begin{equation}
  I_{x,t}(a) := \int_0^t l(g_{a,x}(s),a(s))ds ,
  \label{eq:cost-functional}
\end{equation}
where $l(x,\alpha)$ is called the \emph{running cost} and satisfies for some $C > 0$, and all $x,y \in \R^d \AND \alpha \in A$:
\begin{equation}
  \begin{split}
    & - C \leq l(x,\alpha) \leq C,\\
    & |l(x,\alpha) - l(y,\alpha)| \leq C|x-y|.
  \end{split}
\end{equation}
Let $u_0\colon\R^d \to \R$ be the \emph{terminal cost}. The finite time-horizon problem is defined to be
  \begin{equation}
    u(x,t) = \inf_{a \in \mathcal{A}} \left\{ I_{x,t}(a) + u_0(g(t)) \right\}.
    \label{eq:finite-time-horizon-variational-problem}
  \end{equation}
  We will usually assume that $u_0$ is globally Lipschitz continuous.
%  $\Norm{u_0}{\text{Lip}} < \infty$.

There is a dynamic programming principle (DPP) for $u(x,t)$ and consequently, it is the viscosity solution of a HJB equation 
  \begin{equation}
    \begin{split}
      u_t(x,t) & + H(Du,x) = 0,\\
      u(x,0) & = u_0(x),
    \end{split}
    \label{eq:HJ-eqn-nonviscous-time-dependent}
  \end{equation}
where
  \begin{equation}
    H(p,x) = \sup_{a \in A} ( - p \cdot f(x,a) - l(x,a) ).
    \label{eq:continuous-hamiltonian-definition}
  \end{equation}

We will have use for another type of optimal-control problem called the infinite-horizon or stationary problem. For $\e> 0$, let
  \begin{equation}
    v(x) = \inf_{a \in \mathcal{A}} 
    \int_0^{\infty} e^{-\e s} l(g_{a,x}(s),a(s))~ds .
    \label{eq:continuous-stationary-problem}
  \end{equation}
$v(x)$ also has a DPP and is the unique viscosity solution of
  \begin{equation}
    \e v(x) + H(Dv,x) = 0.
    \label{eq:stationary-problem-pde}
  \end{equation}
  with the same Hamiltonian defined in~\Eqref{eq:continuous-hamiltonian-definition}. The functions $v(x)$ and $u(x,t)$ defined above are usually called \emph{value functions}. Randomness is usually introduced into the problem by requiring $l(x,a,\w)$ and $f(x,a,w)$ to be stationary-ergodic processes. 
  
  For the optimal-control problems that are of interest to us, $l$ and $f$ take the particular forms (for fixed $p \in \R^d$):
\begin{align}
  f(x,\alpha,\w) & = \frac{\alpha}{t(x,\alpha,\w)},  \\
  l(x,\alpha,\w) & = \frac{p \cdot \alpha}{t(x,\alpha,\w)},
  \label{eq:special-form-of-l-and-f}
\end{align}
where $t(x,\alpha,\w)$ is the continuous edge-weight function discussed in~\Secref{sec:fpp-as-homo-problem}.

\section{Discrete optimal-control problems}
Next, we define the discrete counterparts to continuum optimal-control problems we defined in the previous section. Let the state $\gamma_{\alpha,x}\colon\Z^+ \to \Z^d$ satisfy the difference equation 
\begin{equation}
  \begin{split}
    \gamma_{\alpha,x}(j+1) & = \gamma_{\alpha}(j) + \alpha(j) \qquad \forall~j \geq 0,\\
    \gamma_{\alpha,x}(0) & = x,
  \end{split}
  \label{eq:ODE-discrete-optimal-control-problem}
\end{equation}
The controls lie in the set
\[  
  \mathcal{A} := \{ \alpha\colon\Z^+ \to  A \}, 
\]
where $A$ is defined in~\Eqref{eq:control-directions}.

Suppose we have edge-weights $\tau(x,\alpha)$ as in \FPP, and discrete running costs $\lambda(x,\alpha)$ satisfying
\begin{equation}
 |\lambda(x,\alpha)| \leq C. 
  \label{eq:assumptions-on-running-cost}
\end{equation}
Assume that $\tau(x,\alpha)$ is positive and bounded as in~\Eqref{eq:basic-assumption-on-edge-weights}. For a control $\alpha \in \mathcal{A}$, let
\[ \mathcal{W}_{x,k}(\alpha) = \sum_{i=1}^k \tau(\gamma_{\alpha,x}(i),\alpha(i)) \]
be the total time for $k$ steps of the path $\gamma$. For any $\mu_0\colon\Z^d \to \R$ the finite time-horizon problem is 
\begin{equation}
  \mu(x,t) = \inf_{\alpha \in \mathcal{A}} \inf_{k \in \Z^+} \left\{ \sum_{i=0}^{k} \lambda(\gamma_{\alpha,x}(i),\alpha(i)) + \mu_0(\gamma_{\alpha,x}(k)) : \mathcal{W}_{x,k}(\alpha) \leq t \right\}.
  \label{eq:discrete-finite-time-horizon-problem}
\end{equation}
Again, we will assume that the discrete Lipschitz norm $\Norm{\mu_0}{\text{Lip}}$ is finite (see~\Eqref{eq:discrete-lipschitz-norm}).

The stationary problem is defined to be
\begin{equation}
  \nu(x) = \inf_{\alpha \in \mathcal{A}} \left( \sum_{i=0}^{\infty} e^{-\e \mathcal{W}_{x,i}(\alpha)} \lambda(\gamma_{\alpha,x}(i),\alpha(i)) \right).
  \label{eq:discrete-stationary-control-problem}
\end{equation}

\section{Generalization of our setup}
\label{sec:generalization}
We've formulated the problem so that it applies to \FPP~on the directed nearest-neighbor graph of $\Z^d$. It covers the following situations:
\begin{itemize}
  \item Regular first-passage percolation on the undirected nearest-neighbor graph of $\Z^d$ if the edge-weights satisfy
    \[
      \tau(x,\alpha) = \tau(x+\alpha,-\alpha) \quad \forall x \in \Z^d \AND \alpha \in A
    \]
  \item Site first-passage percolation (weights are on the vertices of $\Z^d$) if
    \[
      \tau(x, \alpha) = \tau(x) \quad \forall x \in \Z^d \AND \alpha \in A.
    \]
\end{itemize}
These are by no means the most general problems that comes under the optimal-control framework. Specializing to nearest-neighbor first-passage percolation has mostly been a matter of convenience and taste.

For example, the $e_i$ in the definition of $A$~\Eqref{eq:control-directions} could be any basis for $\R^d$ ---i.e., any lattice--- and our main theorems would hold with little modification. If $A = \{e_1,\ldots,e_d\}$ and we consider $\mathcal{T}(0,x)$, we get directed first-passage percolation; i.e., paths are only allowed to go up or right at any point. Versions of the theorems in~\Secref{sec:main-results} do indeed hold for such $A$, but the \FPT~$\mathcal{T}(x)$ is only defined for $x$ in the convex cone of $A$. If $A$ is enlarged to allow for long-range jumps ---and very large jumps are appropriately penalized--- we obtain long-range percolation. We avoid handling such subtleties here. 

The $d+1$ dimensional directed random polymer assigns a random cost to randomly chosen paths in $\Z^d$. At zero-temperature, this too can be seen as an optimal-control problem. However, as mentioned earlier, variational formulas for directed last-passage percolation and zero-temperature polymer models have been proved in considerable generality by~\citet{georgiou_variational_2013}.

\section{Notation}
%\begin{remark}[Notation]
  We will frequently need to compare discrete and continuous optimal-control problems. So we've tried to keep our notation as consistent as possible. Discrete objects ---functions with at least one input taking values in $\Z^d$--- will be either a Greek or a calligraphic version of a Latin letter. Objects that are not discrete will mostly use the Latin letters. For example, the function $t(x,\alpha)$ in~\Eqref{eq:special-form-of-l-and-f} will be built out of the edge-weights $\tau(x,\alpha)$, and the running costs $l(x,\alpha)$ will be built out of $\lambda(x,\alpha)$. 
  
  Stated as a general rule of thumb: if it's a squiggly variable it's usually discrete and if it's Latin it's usually continuous. Discrete objects and their continuous counterparts are summarized below: 
%  \Tabref{table:notation}.
%\begin{table}
%  \label{table:notation}
\begin{center}
  \begin{tabular}{lcc}
    \textbf{Description} & \textbf{Discrete} & \textbf{Continuous} \\
    Edge-weight function & $\tau(x,\alpha,\w)$ & $t(x,\alpha,\w)$ \\
    Running costs & $\lambda(x,\alpha)$ & $l(x,\alpha,\w)$\\
    Paths & $\gamma(i)$ & $g(t)$ \\
    Weight of a path & $\mathcal{W}(\gamma)$ & $W(g)$\\
    First-passage time & $\mathcal{T}(x,y)$ & $T(x,y)$\\
    Total cost of a path & $\mathcal{I}$ & $I$ \\
    Time-constant & & $m(x)$\\
    Finite time-horizon problem & $\mu(x,t)$ & $u(x,t)$\\
    Stationary problem & $\nu(x)$ & $v(x)$\\
    Hamiltonian & $\mathcal{H}(f,p,x,\w)$ & $H(Df,p,x,\w)$\\
    Homogenized Hamiltonian & & $\overline{H}(p)$ \\
    Derivative  & $\mathcal{D}$ & $D$
  \end{tabular}
\end{center}
%\end{table}

Other notations and conventions are summarized below. 

$\R^+$ and $\Z^+$ refer to the nonnegative real numbers and integers respectively. $\Leb[a,b]$ represents Lebesgue measure on the interval $[a,b]$. $B_R(x)$ is the Euclidean ball on $\R^d$ that has radius $R$ and is centered at $x$.

Integrals with respect to the probability measure will be written as $E[X]$, as $\int X d\Prob$, or as $\int X \Prob(d\w)$. 

$|\cdot|_p$ will refer to usual the $l^p$ on $\R^d$. $|\cdot|$ without a subscript will either mean the $l^2$ norm or the absolute value of a number, depending on the context. $(\cdot,\cdot)$ is the usual dot product on $\R^d$. $L^p$ refers to the space of functions over a measure space with the usual $\Norm{\cdot}{p}$ norm.

The Lipschitz norm of a function $f\colon X \to \R$ on a metric space $(X,\rho)$ is defined as
  \begin{equation}
    \Norm{\cdot}{Lip}(f) := \inf \left\{ C : |f(x) - f(y)| \leq C \rho(x,y) ~\forall~ x,y \in X \right\}.
    \label{eq:discrete-lipschitz-norm}
  \end{equation}
For us, $(X,\rho)$ will be either $(\R^d,|\cdot|)$ or $(\Z^d,|\cdot|_1)$.

The symbol $\emptyset$ refers to the empty set.

The initialism DPP refers to the dynamic programming principle, and HJB stands for Hamilton-Jacobi-Bellmann.
%%% outline

%% file: outline_of_proof.tex
\chapter{Outline of Proof}
\label{sec:outline-of-paper}
\section{Continuum homogenization with $G=\Z^d$}
\label{sec:continuum-homogenization-with-G-equals-Zd}
As described in the introduction, we will construct a function $t(x,\alpha,\w):\R^d \times A \to \R$ using the edge-weights $\tau(x,\alpha,\w)$. Hence, $t(x,\alpha,\w)$ will only inherit the stationarity of the edge-weights on the lattice; i.e.,  
\[ 
  t(x + z,\cdot,\w) = t(x,\cdot,V_z\w)  \quad \forall~z \in \Z^d.
\]
Therefore, we first observe that
\begin{prop}
  The homogenization theorem (Theorem $3.1$) from~\citet{lions_homogenization_2005} holds with $G=\Z^d$.
  \label{prop:homogenization-theorem-for-discrete-translation-group}
\end{prop}
The proof of~\Propref{prop:homogenization-theorem-for-discrete-translation-group} can be summarized as follows: as the functions $T(x)$ are scaled by $n$, it's as if the lattice is scaled to have size $1/n$. The optimal-control interpretation gives us a uniform in $n$ Lipschitz continuity estimate, and hence the scaled functions $T_n(x)$ do not fluctuate too much on the scaled lattice. Therefore, the discrete subadditive ergodic theorem is enough to prove the homogenization theorem. To flesh out some of the details, we'll identify where exactly the subadditive ergodic theorem is used in~\citet{lions_homogenization_2005}. 

The classical approach to proving homogenization is to find a corrector to the cell-problem; i.e., to find a function $v(y)$ satisfying
\begin{equation}
  \begin{split}
    H(p+Dv(y),y) & =  \overline{H}(p), \\
    \lim_{|y|\to\infty} \frac{v(y)}{|y|} & =  0.
  \end{split}
    \label{eq:cell-problem-with-du}
\end{equation}
However, correctors with sublinear growth at infinity do not exist in general~\citep{lions_correctors_2003}. To get around this problem,~\citet{lions_homogenization_2005} consider the equation:
    \begin{equation}
        u_t(x,t) + H(p + Du,x) = 0,  \quad u(x,0) = 0.
        \label{eq:finite-time-horizon-problem-PDE-for-cell-problem}
    \end{equation}
A (sub)solution of this equation can be written in terms of a subadditive quantity, and it follows from the subadditive ergodic theorem that
\begin{prop}
  for any $R > 0$ and $\e > 0$, there is a (random) $t_0$ large enough so that for all $t \geq t_0$
  \[ 
    \left|\frac{u(x,t)}{t} + \overline{H}(p)\right| \leq \e \quad \forall x \in  B_{Rt}(0) ~\almostsurely.
  \]
  \label{prop:convergence-of-u-to-H-p}
\end{prop}
The main ingredient needed to prove~\Thmref{thm:lions-homogenization-theorem}, are approximate corrector-like functions. One way to construct them is through the stationary equation for the cell-problem:
\[
  \e v_{\e}(x) + H(p+Dv_{\e}(x)) = 0.
\]
\citet{lions_homogenization_2005} quote a Abelian-Tauberian theorem for the variational problems in~\Eqref{eq:finite-time-horizon-variational-problem} and~\Eqref{eq:continuous-stationary-problem}, which says that 
\begin{equation}
  \lim_{t \to \infty} \frac{u(x,t)}{t} = \lim_{\e \to 0} \e v_{\e}(x) = - \overline{H}(p). 
  \label{eq:abelian-tauberian-theorem-for-u-and-v}
\end{equation}
Hence, all we need to do is to prove~\Propref{prop:convergence-of-u-to-H-p} when $G = \Z^d$, and this will give the approximate corrector-like functions.

% I'd originally thought it would make sense to prove a uniqueness result, like Theorem $1.5$, ~\cite{lions_correctors_2003}. But it doesn't really matter so much. We can make do with a simpler Lipschitz estimate in the following.

The variational formula for the limiting Hamiltonian (\Propref{prop:lions-variational-characterization-of-h-bar}) also holds with $G = \Z^d$ in~\Eqref{eq:lions-S-set-definition-translation-group-R}. This requires a little work, rather than merely being an observation like the homogenization theorem. We don't need the continuum variational formula in this paper. We present it here because we'll take an analogous route to prove the discrete variational formula. The proof closely follows~\citet{lions_homogenization_2005}. 

The sketch proof for~\Propref{prop:convergence-of-u-to-H-p} (and hence~\Propref{prop:homogenization-theorem-for-discrete-translation-group}), and the proof of the continuum variational formula are in Appendix~\ref{sec:proof-of-homogenization}.

\section{Embedding the discrete problem in the continuum}
\label{sec:taking-the-discrete-problem-into-the-continuum}
Now that we have the appropriate version of~\Thmref{thm:lions-homogenization-theorem} on $\R^d$ with $G = \Z^d$, we have to take discrete \FPP~into the continuum. In addition to defining a continuum version of~\FPP~---as we've sketched in the introduction--- we also need a discrete cell-problem. We will first ``reverse engineer'' the discrete cell-problem from the continuum Hamiltonian.

For fixed $p \in \R^d$, the shifted Hamiltonian for \FPP~in~\Eqref{eq:H-for-general-continuous-FPP} can be written as
\[
H(p+q,x) = \sup_{\alpha \in A} \frac{-p\cdot \alpha - q\cdot\alpha}{t(x,\alpha)}.
\]
Comparing this formula with the definition of the continuum Hamiltonian in the optimal-control formulation~\Eqref{eq:continuous-hamiltonian-definition}, it follows that we must take the continuous running-cost to be 
\begin{equation}
  l(x,\alpha) = \frac{p\cdot\alpha}{t(x,\alpha)}. 
  \label{eq:running-costs-for-continuous-cell-problem}
\end{equation}
Since we'll eventually require the continuous \FPP~to mimic discrete \FPP, we'll define $t(x,\alpha) = \tau(x,\alpha)$ along the edges (see~\Secref{chap:proofs-related-to-penalization-theorems}). For such a $t(x,\alpha)$ and $l(x,\alpha)$, let $u(x,t)$ be the finite time-horizon variational problem defined in~\Eqref{eq:finite-time-horizon-variational-problem}. Suppose a path $g_{x,\alpha}$ traverses an edge $(x,x+\alpha)$. Along this edge, we accumulate cost
\[
\int_0^{\tau(x,\alpha)} \frac{p\cdot \alpha}{\tau(x,\alpha)} = p\cdot \alpha. 
\]
This indicates that we must consider discrete running costs of the form 
  \begin{equation}
    \lambda(x,\alpha) = p \cdot \alpha .
    \label{eq:running-costs-for-discrete-cell-problem}
  \end{equation}

  For the edge-weights $\tau(x,\alpha)$ and the cell-problem running-cost $\lambda(x,\alpha) = p\cdot\alpha$, we'll consider three optimal-control problems: first-passage percolation $\mathcal{T}(x)$, the finite time cell-problem $\mu(x,t)$, and the stationary cell-problem $\nu(x)$. The latter two are defined in~\Eqref{eq:discrete-finite-time-horizon-problem} and~\Eqref{eq:discrete-stationary-control-problem}. 
  
  Now that we've reverse engineered the cell-problem running costs~\Eqref{eq:running-costs-for-discrete-cell-problem}, we turn to constructing continuous approximations of our three discrete optimal-control problems. Using the edge-weights $\tau(x,\alpha)$ and the running-costs $\lambda(x,\alpha)$, we will define (precisely in~\Secref{chap:proofs-related-to-penalization-theorems}) families of functions $t^{\delta}\colon\R^d \times A \to \R$ and $l^{\delta}\colon\R^d \times A \to \R$ parametrized by $\delta$. Let $T^{\delta}(x)$, $u^{\delta}(x,t)$ and $v^{\delta}(x)$ be defined by~\Eqref{eq:continuousFPT-definition},~\Eqref{eq:finite-time-horizon-variational-problem} and~\Eqref{eq:continuous-stationary-problem} (with $\e = 1$) respectively. As $\delta \to 0$, $T^{\delta},~u^{\delta} \AND v^{\delta}$ will approach $\mathcal{T},~\mu \AND \nu$. Hence, we will frequently refer to the continuum problems as ``\deltaapprox s'' to \FPP.
  
  Next, we define the scaling for the three functions $T^{\delta}(x),u^{\delta}(x,t) \AND v^{\delta}(x)$ and their discrete counterparts $\mathcal{T}(x),~\mu(x) \AND \nu(x)$. The function $\mathcal{T}_n(x)$ has already been defined in~\Eqref{eq:n-scaling-first-passage-time}; $T^{\delta}_n(x)$ is similarly defined in terms of $T^{\delta}(x)$. For the finite-time horizon problems we similarly define
\begin{equation}
\begin{split}
  u^{\delta}_n(x,t) = \frac{u^{\delta}([nx],nt)}{n} ~\AND~ \mu_n(x,t) = \frac{\mu([nx],nt)}{n}.
  %    v^{\delta}_n(x) = \frac{v{\delta}([nx])
  \end{split}
  \label{eq:scaling-for-finite-time-horizon-problem}
\end{equation}
 %We extend it from $\Z^d$ to $\R^d$ by defining $u(x,t) = u([x],t)$. 
% $v_n^{\delta}(x)$ is defined to be the solution of the PDE in~\Eqref{eq:continuous-stationary-problem} with $n = e^{-1}$:
For the stationary problem, the scaled versions $v_n^{\delta}(x)$ and $\nu_n(x)$ are obtained by setting $\e = 1/n$ in the variational problems in~\Eqref{eq:continuous-stationary-problem} and~\Eqref{eq:discrete-stationary-control-problem} respectively.
%\begin{equation}
%  \frac{1}{n} v^{\delta}_{n}(y) + H(p+Dv^{\delta}_{n}(y),y)=0 \quad \forall~ y \in \R^d.
%  \label{eq:approximate-problem-to-obtain-corrector}
%\end{equation}

The following interchange of limits (or commutation diagram) is the main ingredient in the proofs of the discrete homogenization theorem and variational formula (\Thmref{thm:-equation-for-mu} and~\Thmref{thm:discrete-variational-formula-for-H-bar}). The theorem compares each of the three sequences of continuum functions 
\begin{equation}
  b^{\delta}_n = T_n^{\delta}(x),~u_n^{\delta}(x,t),\OR \nu_n^{\delta}(x)
  \label{eq:b-sequence-of-continuous-problems}
\end{equation}
with the corresponding discrete versions 
\begin{equation}
  \beta_n = \mathcal{T}_n(x), \mu_n(x,t),\OR \nu_n(x).
  \label{eq:beta-sequence-of-continuous-problems}
\end{equation}
%, where we've used $g$ as short-hand for each of the functions $T(x),u(x,t) \AND v(x)$. 
\begin{theorem}
  For each fixed $x \in \R^d$ and $t \in \R^+$, every pair $(b^{\delta}_n,\beta_n)$ in~\Eqref{eq:b-sequence-of-continuous-problems} and~\Eqref{eq:beta-sequence-of-continuous-problems} homogenizes. That is, $b^{\delta}_n \to \overline{b}^{\delta}$ and $\beta_n \to \overline{b}$. Further,
%  For all $t > b|x|_1$ in case~\ref{item:first-passage-percolation-edge-weight} and all $t > 0$ in case~\ref{item:cell-problem-edge-weight} of edge-weights, the following interchange of limits is justified:
%  \begin{align*}
    \begin{equation*}
      \lim_{\delta \to 0} \lim_{n \to \infty} b^{\delta}_n = \lim_{n \to \infty} \lim_{\delta \to 0} b^{\delta}_n = \lim_{n \to \infty} \beta_n = \overline{b}. 
    \end{equation*}
%      \lim_{\delta \to 0} \lim_{n \to \infty} T^{\delta}_n = \lim_{n \to \infty} \lim_{\delta \to 0} T^{\delta}_n \\
%    \end{align*}
Stated as a commutative diagram, this is
%  \begin{equation}
%    \begin{array}{ccc}
%      g^{\delta}_n	&  \xrightarrow{n}  &   \overline{g}^{\delta} \\
%      \downarrow \delta	        &		    &   \downarrow \delta \\
%      g		&   \xrightarrow{n} & \overline{g}
%    \end{array}.
%  \end{equation}
\begin{equation*}
  \begin{CD}
    b^{\delta}_n @>n>> \overline{b}^{\delta} \\
    @V{\delta}VV  @VV{\delta}V \\
    \beta_n          @>>n> \overline{b}
  \end{CD}
\end{equation*}
  \label{thm:commutation-theorem-for-u-and-T}
\end{theorem}
%  where $g = u \OR T$, and $\overline{g} = -\overline{H}(p)$ and $\mu(x)$ respectively.
We will prove theorem~\Thmref{thm:commutation-theorem-for-u-and-T} when $(b,\beta) = (u,\mu)$, and when $(b,\beta) = (T,\mathcal{T})$. The proof for $(b,\beta) = (v,\nu)$ is nearly identical, and to avoid repetition of ideas, we will omit it.

Proposition~\ref{prop:convergence-of-u-to-H-p} says that 
\[
  \lim_{t \to \infty} u_n^{\delta}(x,t) = -\overline{H}^{\delta}(p).
\]
Let $m^{\delta}$ and $m$ be the time-constants of $T$ and $\mathcal{T}$. The continuum homogenization theorem says that $m^{\delta}$ is the viscosity solution of
\[
  \overline{H}^{\delta}(Dm^{\delta}(x)) = 0, ~m^{\delta}(0) = 0.
\]
\Thmref{thm:commutation-theorem-for-u-and-T} applied to $(u,\mu)$ says that there exists a function $\overline{H}(p)$ such that
\[ 
  \overline{H}^{\delta}(p)  \to \overline{H}(p).
\]
\Thmref{thm:commutation-theorem-for-u-and-T} applied to $(T,\mathcal{T})$ to says that for each $x$, 
\[ 
  m^{\delta}(x) \to m(x).
\]

We would like to show that $m(x)$ solves the PDE corresponding to $\overline{H}(p)$. For this, a uniform in $\delta$ Lipschitz estimate for both $\overline{H}^{\delta}(p)$ and $m^{\delta}(x)$ is sufficient. We state this as two separate propositions below. 
\begin{prop}
  $\overline{H}^{\delta}(p)$ is Lipschitz continuous in $p$ with constant bounded above by $1/a$, where $a$ is defined in~\Eqref{eq:basic-assumption-on-edge-weights}. 
  \label{prop:lipschitz-continuity-of-limit-hamiltonian}
\end{prop}

\begin{prop}
  $m^{\delta}(x)$ is Lipschitz in $x$ with constant bounded above by $b$, where $b$ is defined in~\Eqref{eq:basic-assumption-on-edge-weights}.
  \label{prop:lipschitz-continuity-of-time-constant}
\end{prop}

% Then it follows immediately from~\Thmref{thm:commutation-theorem-for-u-and-T} and~\Propref{prop:lipschitz-continuity-of-limit-hamiltonian},
%\begin{lemma}
%  locally uniformly in $p$ and $x$,
%  \[ \overline{H}^{\delta}(p)  \to \overline{H}(p), \]
%  and  
%  \[ \mu^{\delta}(x) \to \mu(x).\]
%  \label{lem:local-uniform-convergence-of-mu-and-H}
%\end{lemma}
Since $\overline{H}^{\delta}$ and $m^{\delta}$ converge locally uniformly to $\overline{H}$ and $m$ respectively, the standard stability theorem for viscosity solutions~\citep{crandall_users_1992} implies~\Thmref{thm:-equation-for-mu}. 
 
The proofs of the results in this section are in~\Chapref{chap:proofs-related-to-penalization-theorems}.

\section{Discrete variational formula and solution of the limiting PDE}
\label{sec:outline-discrete-variational-formula}
The commutation theorem also transfers the Abelian-Tauberian theorem over from the continuum relating the limits of $u^{\delta}_n(x,t)$ and $v^{\delta}_{\e}(x)$~\Eqref{eq:abelian-tauberian-theorem-for-u-and-v}. That is, both $\mu_n(x,t)$ and $\nu_{\e}(x)$ converge to $\overline{H}(p)$ almost surely as $n \to \infty$ and $\e \to 0$. This is very useful in establishing the discrete version of the variational formula. We have
\begin{prop} for each $R > 0$ and $\eta > 0$, there is a small enough (random) $\e_0$ such that for all $\e \leq \e_0$,
  \[ 
    \left|\e \nu_{\e}(x) + \overline{H}(p)\right| < \eta \quad \forall x \in B_{\e^{-1}R}(0).
  \]
%  uniformly for $x \in B_{\e^{-1}R}(0)\cap\Z^d$. 
  \label{prop:statement-of-tauberian-theorem}
\end{prop}

The stationary cell-problem has the discrete (DPP)
\begin{equation}
  \nu_{\e}(x) = \inf_{\alpha \in A} \left( \alpha \cdot p + e^{-\e \tau(x,\alpha)} \nu_{\e}(x + \alpha) \right).
  \label{eq:discrete-dpp-stationary-problem}
\end{equation}
With a little manipulation of the DPP, we can obtain a discrete version of the stationary PDE in~\Eqref{eq:stationary-problem-pde}.
\begin{prop}
  For all $x \in \Z^d$, there is a constant $C > 0$ uniform in $\e$ and $\w$ such that
  \[
  -C \e \leq \e \nu_{\e}(x) + \mathcal{H}(\nu_{\e},p,x) \leq C \e,
  \]
  \label{prop:approximate-discrete-HJB-for-stationary-problem}
  where $\mathcal{H}$ is the discrete Hamiltonian~\Eqref{eq:discrete-hamiltonian-1}.
\end{prop}

The final ingredient for the variational formula is what is usually called a comparison principle for HJB equations. In the continuum, the comparison principle is stated for sub- and supersolutions of the PDE. The discrete comparison principle we prove is a less general version that suffices for our purposes. Consider the discrete problem in~\Eqref{eq:discrete-finite-time-horizon-problem} for any $\phi\colon\Z^d \to \R$ such that $\Norm{\phi}{Lip} < \infty$. Then,
\begin{prop}
  \[ 
    \mu(x,t) \geq \phi(x) - t \sup_{x \in \Z^d} \mathcal{H}(\phi,p,x) \quad \forall x \in \Z^d \AND t \in \R^+.
  \]
  \label{prop:comparison-principle}
\end{prop}
Using these facts, the discrete variational formula in~\Thmref{thm:discrete-variational-formula-for-H-bar} is easy to prove.

Now that we have a formula for the limiting Hamiltonian, we can solve the PDE in~\Eqref{eq:limiting-HJB} to obtain the time-constant $m(x)$. It follows directly from the variational formula that 
\begin{prop}
  $\overline{H}(p)$ is a norm on $\R^d$.
  \label{prop:limiting-hamiltonian-is-a-norm}
\end{prop}

Let $L(x)$ be the dual norm of $\overline{H}(p)$ on $\R^d$. Consider the set of paths
\[ 
  \mathcal{A} := \left\{g \in C^1([0,\infty),\R^d) : L(g'(s)) = 1 ~\forall~ s \in [0,t)\right\}. 
\]
  Let $T(x)$ be the minimum time function defined by~\Eqref{eq:continuousFPT-definition}. From standard optimal-control theory~\citep{bardi_optimal_1997}, it follows that $T(x)$ is the unique viscosity solution of the metric HJB equation~\Eqref{eq:time-constant-satisfies-a-PDE}. A standard Hopf-Lax formula gives 
\begin{prop}
  $T(x) = L(x)$.
  \label{prop:Hopf-Lax-formula-for-u}
\end{prop}
The fact that $\overline{H}(p)$ is the dual norm of $m(x)$ follows immediately, and \Corref{cor:H-is-the-dual-norm-of-time-constant} is proved.

\begin{remark}
  It is natural to question the necessity of taking the discrete problem into the continuum; the PDE is unnecessary once the Hamiltonian has been identified as the dual norm of the time-constant\footnote{observation due to S.R.S Varadhan.}. 

  There is a direct proof based on the continuum homogenization theorem of~\citet{kosygina_stochastic_2006}. Such a route has been taken to prove variational formulas for the large deviations of random walks in random environments by~\citet{rosenbluth_quenched_2008}. This work was considerably generalized by~\citet{rassoul-agha_quenched_2013} and \citet{rassoul-agha_quenched_2014}. Subsequently,~\citet{georgiou_variational_2013} extended these ideas to prove variational formulas for the directed random polymer, and for last-passage percolation.
%  There ought to be a direct proof of this fact that does not depend on the translation group. This is frequently done in the theory of large deviations for random walks in random environments
\end{remark}

\begin{remark}
  We chose the metric-form of the HJB equation so that the limiting Hamiltonian $\overline{H}(p)$ could be interpreted as a norm. This allowed us to solve the PDE for the time-constant. We can make the assumption on the edge-weights $\tau(x,\alpha,\w)$ in~\Eqref{eq:basic-assumption-on-edge-weights} less restrictive if the Hamiltonian is written in the more standard form
  \[ 
    \mathcal{H}(p,x) = \sup_{\alpha \in A} \left\{-p\cdot\alpha - \tau(x,\alpha) \right\}. 
  \]
  Here, $\tau(x,\alpha,\w)$ can take the value $0$ without making the Hamiltonian blow-up. However, even though the homogenization theorem and variational formula are still valid, we do not know the Hopf-Lax formula for the limiting PDE.
\end{remark}

%Section~\ref{sec:proof-of-homogenization} is about the continuum homogenization theorem when $G=\Z^d$, and can be skipped on first reading. It's essentially an outline of the method of~\citet{lions_homogenization_2005}. 
%The meat of the argument is in Chapters~\ref{chap:proofs-related-to-penalization-theorems} and~\ref{chap:proofs-related-to-discrete-variational-formula}. The proof of the continuum homogenization theorem when $G=\Z^d$ is an elementary extension of~\citet{lions_homogenization_2005}, and hence is banished to appendix~\ref{sec:proof-of-homogenization}.
The proofs of the results in this section can be found in~\Chapref{chap:proofs-related-to-discrete-variational-formula}.
%% proof of homogenization

%% file: proof_discrete_homogenization.tex
\chapter{Proofs related to the discrete homogenization theorem} 
\label{chap:proofs-related-to-penalization-theorems}
\section{Fattening the unit-cell}
In this section, all the estimates will hold almost surely, and hence we'll not explicitly refer to $\w \OR \W$. It will be useful for the reader to visualize the lattice as being embedded in $\R^d$. Discrete paths on the lattice will now be allowed to wander away from the edges of the graph on $\Z^d$ and into $\R^d$. This will let us define continuum variational problems ---what we've called \deltaapprox s in~\Secref{sec:outline-of-paper}--- that will approximate discrete \FPP~and its associated cell-problem.

Consider a unit-cell of the lattice embedded in $\R^d$. We will ``fatten'' the edges and vertices of the lattice into tubes and corners. The remaining space in the unit-cell contains its center point, and so we will call this region a center (see~\Figref{fig:sketch-of-fattened-unit-cell}). Let $0 < \delta < 1/2$ be a parameter describing the size of the tubes and corners. Define \begin{enumerate}
  \item The \emph{tube} at $x \in \Z^d$ in the $\alpha \in A$ direction as:  
    \[ 
      TU^{\delta}_{x,\alpha} := \big\{ x + \lambda \alpha + y:\delta < \lambda < 1 - \delta,~y \in \{\alpha\}^{\perp},~|y|_{\infty} \leq \delta\big\}.
    \]
    The tubes have width $2\delta$ and length $1-2\delta$. 
  \item The \emph{corner} around a vertex $x$ as:
\[ 
  CO^{\delta}_{x} := \{ y:|y-x|_{\infty} \leq \delta \}.
\]
  \item The \emph{center} of the cell as:
\[ 
  CE^{\delta}_x := \bigg\{ x + \sum_{i=1}^d \lambda_i e_i : \delta < \lambda_i < 1 - \delta, ~\forevery~ i \bigg\}.
\]
\end{enumerate}
The three regions are disjoint. That is, for any $x,y,z \in \R^d$ and $\alpha \in A$, 
\[
  TU_{x,\alpha}^{\delta} \cap CO_y^{\delta} \cap CE^{\delta}_z = \emptyset .
\]
% figure of unit cell
\begin{figure}[t]
  \begin{center}
    \includegraphics[width=0.65\textwidth]{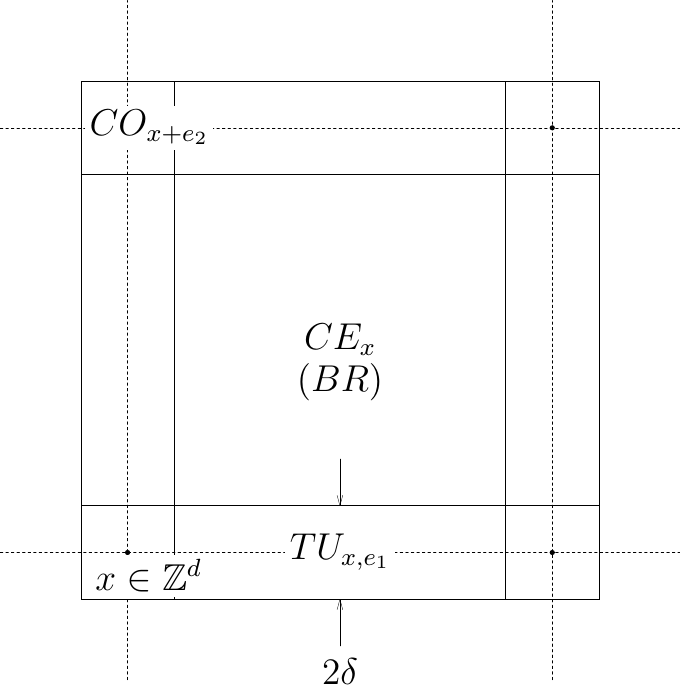}
  \end{center}
    \caption{Sketch of fattened unit-cell. The dotted lines represent the edges of the lattice. The solid lines show the boundaries of the corners, tubes and centers. When $\delta$ is expanded to $\eta$, the center becomes a bad region.}
    \label{fig:sketch-of-fattened-unit-cell}
\end{figure}

Next, we need to define the edge-weight function $t^{\delta}(x,\alpha)$ and the cell-problem running cost $l^{\delta}(x,\alpha)$ in the \deltaapprox . We'd like the value functions in the \deltaapprox~to be close to the discrete problem, and so we would like paths to avoid the centers of cells and stick close to the edges. By penalizing paths that cross into centers of cells with additional cost, we'll ensure that they stay inside the tubes and corners. For each $\alpha \in A$, let
\begin{equation}
  t^{\delta}_c(x,\alpha) := \left\{ 
  \begin{array}{cl}
    \tau(x,\alpha) & \text{if } x \in TU^{\delta}_{z,\alpha} \text{ for some } z \in \Z^d \\
    a   & \text{if } x \in CO_z \text{ for some } z \in \Z^d \\
    \delta^{-1}   & \text{otherwise}
  \end{array} \right. .
  \label{eq:time-edge-weights-tau}
\end{equation}
%\begin{remark}
%  The tube $TU^{\delta}_{x,\alpha}$ is considered distinct from $TU_{x+\alpha,\alpha}^{\delta}$ to account for the possibility that $\tau(x,\alpha) \neq \tau(x+\alpha,-\alpha)$. 
%\end{remark}

Fix $p \in \R^d$, and let $\lambda(x,\alpha) = p\cdot\alpha$ be the discrete cell-problem weight  defined in~\Eqref{eq:running-costs-for-discrete-cell-problem}. We define the running costs as
\begin{equation}
  l^{\delta}_c(x,\alpha) = \left\{ 
  \begin{array}{cl}
    \lambda(x,\alpha) & \text{if } x \in TU_{z,\alpha} \text{ for some } z \in \Z^d\\
    -C & \text{if } x \in CO_z \text{ for some } z \in \Z^d \\
    \delta^{-1} & \text{otherwise}
  \end{array} \right. ,
  \label{eq:running-cost-per-unit-time-l-extension}
\end{equation}
  where $C = -|p|_{\infty} $ is the lower bound on $\lambda(x,\alpha)$. The functions $t^{\delta}_c$ and $l^{\delta}_c$ represent piecewise extensions of the discrete edge-weights and  running costs. Define the mollified functions
\[ 
  l^{\delta}(x,\alpha) = \eta_{\delta/2} * l^{\delta}_c(x,\alpha) \quad \AND \quad t^{\delta}(x,\alpha) = \eta_{\delta/2} * t^{\delta}_c(x,\alpha),
\]
where $\eta_{\delta}$ is the standard mollifier with support $\delta$. The Hamiltonian $H^{\delta}$ obtained from $l^{\delta}$ and $t^{\delta}$ using~\Eqref{eq:continuous-hamiltonian-definition} satisfies the hypotheses of the Lions-Souganidis continuum homogenization theorem (\Thmref{thm:lions-homogenization-theorem}). 

For the finite time-horizon cell-problem, there is a final cost given by $\mu_0\colon\Z^d \to \R$. We can extend the final cost smoothly to $\R^d$ by defining 
\[
  u_0^{\delta}(x) = \eta_{\delta/2} * \mu_0([x]).
\]
The continuous variational problems in~\Secref{sec:optimal-control-theory} can now be defined using the smooth functions $l^{\delta}$, $t^{\delta}$ and $u_0^{\delta}$.

The main focus of this section is~\Thmref{thm:commutation-theorem-for-u-and-T}. We will first prove it for when $\beta = \mu(x,s) \AND b^{\delta} = u^{\delta}(x,s)$, the finite time-horizon cell-problems. 
\begin{remark}
  Although trivial, we remark that 
  \[ u^{\delta}(x,s) \leq \mu(x,s),\]
    since we can always take paths going along edges in the \deltaapprox. 
    \label{rem:L-delta-smaller-than-L}
\end{remark}
Then, the proof of~\Thmref{thm:commutation-theorem-for-u-and-T} is easy given that 
\begin{lemma}
  for $\delta$ small enough, we have the estimate 
  \[ 
    \frac{\mu([nx],ns)}{n} - \frac{u^{\delta}([nx],ns)}{n} \leq C\left( \sqrt{\delta}s + \frac{1}{n}\right) .
  \]
  \label{lem:uniform-in-n-delta-estimate}
\end{lemma}
We will complete the proof of the main theorem before proving~\Lemref{lem:uniform-in-n-delta-estimate}. 
\begin{proof}[Proof of~\Thmref{thm:commutation-theorem-for-u-and-T}]
  While the homogenization theorem applies directly to $u^{\delta}(x,s)$, our scaling in~\Eqref{eq:scaling-for-finite-time-horizon-problem} is slightly different. We scaled it differently so that it is enough to compare the discrete problem $\mu(z,s)$ to the continuum problem $u^{\delta}(z,s)$ on lattice points $z \in \Z^d$. So we account for this first. Since $l^{\delta}$ is bounded above by $\delta^{-1}$, we have
  \[ 
    \left| \frac{u^{\delta}([nx],ns)}{n} - \frac{u^{\delta}(nx,ns)}{n} \right| \leq \frac{Cs}{\delta n}. 
  \]
    It follows from~\Propref{prop:convergence-of-u-to-H-p} (or the continuum homogenization theorem) that $u^{\delta}(nx,ns)/n$ has a limit, and hence
  \[ 
    \lim_{n \to \infty} \frac{u^{\delta}([nx],ns)}{n} = \lim_{n \to \infty} \frac{u^{\delta}(nx,ns)}{n} =: \overline{u}^{\delta}(x,s). 
  \]
  % I constantly debate whether I have to change this to the scaled and or write everything out explicitly. But I'm going to leave $\mu_n(x,s)$ and $u_n^{\delta}(x,s)$ here because of their definitions. 
  From~\Lemref{lem:uniform-in-n-delta-estimate} and~\Remref{rem:simple-observation-that-will-help-us-work-with-piecewise-functions}, we have the following inequality for the scaled functions~\Eqref{eq:scaling-for-finite-time-horizon-problem}
  \begin{equation}
    u^{\delta}_n(x,s) \leq \mu_n(x,s) \leq u^{\delta}_n(x,s) + C\left( \sqrt{\delta}s  + \frac{1}{n} \right) \quad \forall x \in \R^d.
    \label{eq:squeeze-mu-between-u-and-c-times-delta}
  \end{equation}
%  This implies that for each $n$ and all $x \in \R^d$,
%  \[
%    \lim_{\delta \to 0} u^{\delta}_n(x,s) = \mu_n(x,s) + .
%  \]
%  Inequality~\Eqref{eq:squeeze-mu-between-u-and-c-times-delta} also gives
%  \[
%  \begin{split}
%    \varlimsup_{n \to \infty} u^{\delta}_n(x,s) \leq & \varlimsup_{n \to \infty} \mu_n(x,s) + C\sqrt{\delta} \\
%        & \varliminf_{n \to \infty} \mu_n(x,s) \leq \varliminf_{n \to \infty} u^{\delta}_n(x,s) + C\sqrt{\delta} .
%  \end{split}
%  \]
  Taking a limit in $n$ first, and then in $\delta$ (limsups and liminfs as appropriate), we get
  \[
  \begin{split}
    \overline{u}^{\delta}(x,s) \leq & \varliminf_{n \to \infty} \mu_n(x,s),\\
    & \varlimsup_{n \to \infty} \mu_n(x,s) \leq \overline{u}^{\delta}(x,s) + C\sqrt{\delta} s.
  \end{split}
  \]
  Since $\delta$ is arbitrary, it follows that $\mu_n(x,s) \to \overline{u}(x,s)$ as $n \to \infty$. Taking limits in the reverse order and using the fact that $\mu_n(x,s)$ has a limit completes the proof. 
%  As $\delta \to 0$, $l^{\delta}(x,\alpha)$ increases monotonely for fixed $x$ and $\alpha$. Hence, the same is true of $u^{\delta}_n(x,s)$ for each fixed $n,~x \AND t$, and therefore, of $\overline{u}^{\delta}(x,s)$. 
%  Hence,
%  \[ \lim_{\delta \to 0} \overline{u}^{\delta}(x,s) = \overline{u}(x,s). \]
%  This proves the result.
\end{proof}

\section{Integral formulation of variational problem}
It's easier to work with an integral formulation of \FPP~and its cell-problem that will allow us to drop reference to the control $a(s)$ in~\Eqref{eq:cost-functional}. This is easily done by extending $l^{\delta}(x,\alpha)$ and $t^{\delta}(x,\alpha)$ one-homogeneously from $\R^d \times A$ to $\R^d \times \R^d$. For $x,r \in \R^d$ we redefine
\begin{equation}
  \begin{split}
    l^{\delta}(x,r) & = l^{\delta}\left( x, \frac{r}{|r|} \right)|r|, \\
    t^{\delta}(x,r) & = t^{\delta}\left( x, \frac{r}{|r|} \right)|r|.
  \end{split}
  \label{eq:running-cost-and-time-one-homo-extension}
\end{equation}

Hence, we may write the total cost of a path parametrized by $g\colon[0,s]\to\R^d$ (see~\Eqref{eq:cost-functional}) as
\begin{equation}
  I^{\delta}(g) = \int_0^s l^{\delta}(g(r),g'(r))dr.
  \label{eq:cost-functional-definition}
\end{equation}
We have to prove that dropping reference to the control in~\Eqref{eq:cost-functional-definition} will not affect our problem. That is we've to show that if we take a path $g$ realized by a control and reparametrize it, its total cost will be unaffected. Let $r = h(q)$ be a smooth change of time parametrization, where $h$ is an increasing function. Let $y(r) = g(h^{-1}(r))$ be the path. Then, a simple change of variable gives
\begin{prop}[Cost is independent of path parametrization]
  \[ I^{\delta}(g) = \int_0^s l^{\delta}(g(q),g'(q))dq = \int_0^{h(s)} l^{\delta}(y(r),y'(r))dr.\]   
  \label{prop:cost-is-independent-of-time-parametrization}
\end{prop}
%Although one can prove that it does not matter, we will restrict to paths $g(s)$ that can be ``realized'' with a control in the following. 

Since we can dispense with the controls and talk directly about paths, define
\[ 
  U := \left\{ g(s) \in C^{0,1}([0,1],\R^d): \frac{g'(s)}{|g'(s)|} \in A ~\text{a.e. } \Leb[0,1] \right\}.
\]
This restricts us to paths that can be ``realized'' with controls. Let $U_x$ be the subset of paths that start at $x$, and let $U_{x,y}$ be the subset of paths that go from $x$ to $y$. Let $\mathcal{U}$, $\mathcal{U}_x$ and $\mathcal{U}_{x,y}$ be the corresponding subsets of $U$ where paths are only allowed to go on the edges of the lattice $\Z^d$.

%We can set $\lambda(x,\alpha)= t^{\delta}(x,\alpha) |\alpha|_1$ to obtain the time of the path as the total cost. 
The total-time or weight of a path $g \in U$ is
\begin{equation}
  W^{\delta}(g) = \int_0^1 t^{\delta}\left( g(r),\frac{g'(r)}{|g'(r)|} \right)|g'(r)|~dr. 
  \label{eq:time-functional-definition}
\end{equation}
The cost of $g$ is given by~\Eqref{eq:cost-functional-definition} with $s=1$. Let the $L^1$ length of $g$ be 
\begin{equation}
  d(g) = \int_0^1 |g'(r)|~dr. 
  \label{eq:L1-length-of-a-continuous-path}
\end{equation}

This allows us the following reformulation of the discrete and continuous cell-problems: 
\begin{prop}
  For $x \in \R^d,~z \in \Z^d$  and $s \in \R^+$, we have
  \[ 
    \begin{split}
      u^{\delta}(x,s) & = \inf_{g \in U_x} \left\{ I^{\delta}(g) : W^{\delta}(g) \leq s \right\}, \\
      \mu(z,s) & = \inf_{\gamma \in \mathcal{U}_z} \left\{ I^{\delta}(\gamma) : W^{\delta}(\gamma) \leq s \right\}.
    \end{split}
  \]
  \label{prop:reformulation-of-control-problem-as-integral-variational-problem}
\end{prop}

\begin{proof}
  The proof follows from~\Propref{prop:cost-is-independent-of-time-parametrization} and the definitions of $u$ and $\mu$ in~\Eqref{eq:finite-time-horizon-variational-problem} and~\Eqref{eq:discrete-finite-time-horizon-problem}.
\end{proof}

\begin{remark}
  There are two things to notice about~\Propref{prop:reformulation-of-control-problem-as-integral-variational-problem}. First, notice that $\mu(x,s)$ is ``embedded'' in the continuum problem for every $\delta$, and further, $u^{\delta}(x,s) \leq \mu(x,s) ~\forall~ x \in \Z^d$. Second, we no longer need $t^{\delta}(x,\alpha)$ and $l^{\delta}(x,\alpha)$ to be smooth in $x$ for the variational problem to be well-defined. This allows us to work with the piecewise versions $l^{\delta}_c(x,\alpha)$ and $t^{\delta}_c(x,\alpha)$, since these are easier to compare to the discrete problems. 
\end{remark}

\begin{remark}
  One could argue that we needn't have introduced the optimal-control framework since our costs depend only on the graph of the path and not its parametrization. However, DPPs and Hamiltonians are conventionally represented using the control interpretation.  Several other models can be naturally formulated using the optimal-control language, and our approach ought to work for these (see~\Secref{sec:generalization}).
\end{remark}

%Before we move on to the proof of~\Lemref{lem:uniform-in-n-delta-estimate}, we summarize the notation we'll use: \\ \par
%\begin{tabular}{lp{0.7\textwidth}}
%  Paths & We'll denote edge-paths with the subscript $e$; i.e., as $\gamma \in \mathcal{U} \subset U$. \\
%
%  Functions & We'll attach a superscript $\delta$ to quantities associated with the \deltaapprox.  Subscript $p$ will mean piecewise constant over $x$. Functions sans subscript or superscript will refer to the discrete problem. \\
%
%  Functionals & The functionals $I:U \to \R$ and $W:U \to \R$ will carry a superscript $\delta$, a subscript $p$, or both to indicate which functions appear in the integrals in~\Eqref{eq:cost-functional-definition} and~\Eqref{eq:time-functional-definition}. 
%\end{tabular}
\section{Proof of~\Lemref{lem:uniform-in-n-delta-estimate}}
We state a simple comparison result first. Suppose $t_1(x,\alpha) \leq t_2(x,\alpha)$ and $l_1(x,\alpha) \leq l_2(x,\alpha)$. Let $u_1(x,s)$ and $u_2(x,s)$ be the corresponding variational problems defined by~\Propref{prop:reformulation-of-control-problem-as-integral-variational-problem}. Then, it follows quite easily that
\begin{prop}[Comparison of variational problems]
  for all $x \in \R^d$ and $s > 0$,
  \[ 
    u_1(x,s) \leq u_2(x,s).
  \]
  \label{prop:comparison-of-variational-problems}
\end{prop}

%Next, we remark that
\begin{remark}
  Let $w^{\delta}(x,s) \leq u^{\delta}(x,s)$ for all $x \in \R^d$ and $s > 0$. To prove~\Lemref{lem:uniform-in-n-delta-estimate}, it's enough to show that for any $s > 0$,
\[ 
  \mu(x,s) - w^{\delta}(x,s)  \leq  C (\sqrt{\delta} s + 1) \quad \forall x \in \Z^d. 
\] 

  \label{rem:simple-observation-that-will-help-us-work-with-piecewise-functions}
\end{remark}

\begin{proof}[Proof of~\Lemref{lem:uniform-in-n-delta-estimate}]
%We wish to show  
%\[ |u^{\delta}(x,s) - \mu(x,s)| \leq C (\sqrt{\delta} t + 1) \]
%In the light of~\Remref{rem:L-delta-smaller-than-L}, it's enough to show that
%\[ \mu(x,s) - u^{\delta}(x,s) \leq C (\sqrt{\delta} t + 1) .\]
We will use the observation made in~\Remref{rem:simple-observation-that-will-help-us-work-with-piecewise-functions} repeatedly in the proof. We will also refer to the bounds on $\tau(x,\alpha)$ assumed in~\Eqref{eq:basic-assumption-on-edge-weights}. The proof takes several steps.

%define a new counter
\newcounter{proofsteps} \setcounter{proofsteps}{0}
\stepcounter{proofsteps}\noindent\theproofsteps .  Let us first reduce to the case where the cost and time functions are piecewise constant. For this, we will show that
    \begin{equation}
      \begin{split}
        l^{2\delta}_c(x,\alpha) & \leq l^{\delta}(x,\alpha), \\
        t^{2\delta}_c(x,\alpha) & \leq t^{\delta}(x,\alpha). 
      \end{split}
      \label{eq:reducing-to-piecewise-case}
    \end{equation}
    Then, for a fixed path $g$, it would follow from~\Propref{prop:comparison-of-variational-problems} that 
    \[ 
      I^{2\delta}_c(g) \leq I^{\delta}(g) \AND W^{2\delta}_c(g) \leq W^{\delta}(g),
    \]
    where the subscript $c$ in $I$ and $W$ represent the cost and time corresponding to the piecewise versions of $l$ and $\tau$. Then,~\Propref{prop:comparison-of-variational-problems} and~\Remref{rem:simple-observation-that-will-help-us-work-with-piecewise-functions} imply that we can just work with $t_c^{2\delta}$ and $l_c^{2\delta}$. 

    To show the two inequalities in~\Eqref{eq:reducing-to-piecewise-case}, it's enough to show that $t^{2\delta}_c(x,\alpha) \leq t^{\delta}_c(y)$ and $l^{2\delta}_c(x) \leq l^{\delta}_c(y)$ for all $y$ in the ball $B_{\delta}(x)$. This is clear by drawing a picture. Then, multiply with the standard mollifier $\eta_{\delta/2}(y-x)$ and integrate over $y$.

  \stepcounter{proofsteps}\noindent\theproofsteps .
  Create boxes of side-length $1-2(\delta+\sqrt{\delta})$ called \emph{bad regions} (BRs) contained inside centers (see~\Figref{fig:sketch-of-fattened-unit-cell}). We show that if a path goes through a BR, it will do so badly that it will make more sense to stick to the tubes and corners. If $g$ visits a BR, it will take at least time 
    \[ \sqrt{\delta}t^{2\delta}_c(x,\alpha) = \frac{1}{\sqrt{\delta}}.  \]
    Hence it accumulates cost
    \[ \sqrt{\delta}t^{\delta}_c(x,\alpha) \cdot \frac{\delta^{-1}}{t^{\delta}_c(x,\alpha)} = \frac{1}{\sqrt{\delta}} .\]
    A path $g$ that visits a BR must leave a tube at some point $A$ to enter the BR. Once it's done being bad, it must re-enter another tube or corner in the same cell at a point $B$. We will form a new path $g^*$ that connects $A$ and $B$ by a path that only goes through tubes and corners. Making crude estimates, we see that the new path $g^*$ has distance at most $3$ to travel. Since the edge-weight function is bounded, it takes time at most $3b$, and costs at most $3C$. Since $g^*$ does not take more time than $g$ and costs less, we might as well assume that paths do not enter BRs. Henceforth, we will assume that $U$ contains only such good paths.
%    This depends on the value of $\tau_{ce}$. The length of $g^*$ between $A$ and $B$ is at most $C$ times the length of $g$ (draw a picture). So we might as well set $\tau_{ce} = Cb$, where $b$ is the worst value of $\tau$ in the tubes. 
    
\stepcounter{proofsteps}\noindent\theproofsteps .
 Expand the tubes $TU^{2\delta}$ and corners $CO^{2\delta}$ to have thickness $\eta$, where
  \begin{equation}
    \eta = 2 \delta + \sqrt{\delta}.
    \label{eq:eta-tube-size-definition}
  \end{equation}
That is, we've expanded the tubes $TU^{2\delta}$ and corners $CO^{2\delta}$ so that the centers shrink to the BRs. Consider the problem $t^{\eta}_c$, $l^{\eta}_c$ defined by~\Eqref{eq:running-cost-per-unit-time-l-extension} and~\Eqref{eq:time-edge-weights-tau}. Then for $\delta$ small enough,
\[ u^{\eta}(x,s) \leq u^{2\delta}(x,s). \]

\stepcounter{proofsteps}\noindent\theproofsteps .
 For a path $g \in U$, we need to construct an edge-path $\gamma \in \mathcal{U}$ that has a similar cost and time. We will consider each tube and corner that $g$ passes through (since it avoids BRs), and construct $\gamma$ in each region so that it follows $g$ around. Fix such a tube $TU$ or corner $CO$, and continue to write $g$ for just the section of the path going through it. Then,
\begin{claim}
  We can find $\gamma \in \mathcal{U}$ such that
  \[ 
    I^{\eta}_c(\gamma) \leq I^{\eta}_c(g) + C\eta, 
  \]
  \[ 
    W^{\eta}_c(\gamma) \leq W^{\eta}_c(g) + C\eta, 
  \]
  inside each tube or a corner that $g$ goes through.
  \label{claim:comparing-paths-deltapprox-to-edge-paths}
\end{claim}
We will prove Claim~\ref{claim:comparing-paths-deltapprox-to-edge-paths} after completing the proof of this lemma.

\stepcounter{proofsteps}\noindent\theproofsteps .
 A path $g \in U$ can travel a distance at most $s/a$ in time $s$. So the number of tubes and corners the path can visit is at most $2s/a$. Using Claim~\ref{claim:comparing-paths-deltapprox-to-edge-paths}, we will approximate $g$ by an edge-path $\gamma$ in each corner and tube that it goes through, except for possibly $C\eta s$ tubes and corners (since $\gamma$ is typically slower through the corners). Hence, the most cost that $\gamma$ could have missed out on is $C\eta s$, since $l^{\eta}_c(x,\alpha)$ is bounded below. 
 
 % We must also take the final cost into account because its needed in the variational formula.
 We must also take the final cost due to $u_0$ into account. Recall that we've assumed that $\Norm{u_0}{Lip} < \infty$. The final locations of $g$ and $\gamma$ cannot differ by more than $C\eta s$ and hence, neither can the final cost. Finally, $\gamma$ ought to end on a lattice point, whereas $g$ need not. Accounting for all this, we get that
\[ 
  I^{\eta}_c(\gamma) \leq I^{\eta}_c(g) + C \eta s + C.
\]

Since $I^{\eta}_c(\gamma)$ is exactly the cost accumulated by a path in the discrete problem, scaling by $n$ completes the proof.
\end{proof}

\begin{proof}[Proof of claim~\ref{claim:comparing-paths-deltapprox-to-edge-paths}]
  Corners are of size $\eta$, and the inequalities in Claim~\ref{claim:comparing-paths-deltapprox-to-edge-paths} are immediate. Hence, we only need to prove it for a tube, which wlog, we can assume to be $TU^{\eta}_{0,e_1}$. The end-caps of the tube of size $\eta$ at the origin in the $e_1$ direction are $\{ \eta e_1 + y : y \in \{e_1\}^{\perp}, |y|_{\infty} \leq \eta\}$, and $\{ (1-\eta) e_1 + y : y \in \{e_i\}^{\perp}, |y|_{\infty} \leq \eta\}$. Let $A = (A_i)_{i=1}^{d}$ be the last point on the end-cap before $g$ enters the tube, and let $B=(B_i)_{i=1}^{d} $ be the first point on the end-cap when $g$ exits the tube. Assume a parametrization such that $g(0) = A$ and $g(1) = B$.
   The total cost of $g$ is
    \begin{equation*}
%      \begin{split}
        I^{\eta}_c(g) = \int_0^1 p \cdot g'(s) ~ds = p \cdot (B - A). 
%      \end{split}
    \end{equation*}
  If $A$ and $B$ are on the same end-cap, the edge-path $\gamma$ does not move at all, and if they're on different end-caps, it travels from end-cap to end-cap along the edge. 
%  {\infty}\[ g(0) = A \AND g(1) = B. \]
  
  The tube is aligned with $e_1$ by construction, and hence $|A_1 - B_1| = 1-2\eta$ and $|A_j - B_j| \leq 2\eta ~\forall j \neq 1$. Hence,
  \[ 
    I^{\eta}_c(g) - I^{\eta}_c(\gamma) \leq C \eta.
  \]
  It's also clear that the $L^1$ length of the path is
  \[ 1 - 2\eta = d(\gamma) \leq d(g). \]
  It follows that
  \begin{equation}
    W^{\eta}_c(\gamma) \leq W^{\eta}_c(g).
    \label{eq:time-of-edge-path-compare-smooth-path}
  \end{equation}
\end{proof}

To obtain a version of~\Lemref{lem:uniform-in-n-delta-estimate} for $T^{\delta}_n(x)$ and $\mathcal{T}_n(x)$, first replace $U$ by $U_{0,x}$. The essential estimate is already present in Claim~\ref{claim:comparing-paths-deltapprox-to-edge-paths}, and the rest of the argument is identical.

\section{Lipschitz estimates on Hamiltonians and time-constants}
We next prove the Lipschitz estimates on the time-constants $m^{\delta}$ and limiting Hamiltonians $\overline{H}^{\delta}$.
\begin{proof}[Proof of~\Propref{prop:lipschitz-continuity-of-limit-hamiltonian}]
  Fix $p_1 \AND p_2 \in \R^d$. Let $l^{\delta}_{i,c}(x,\alpha)$ be the piecewise function defined in~\Eqref{eq:running-cost-per-unit-time-l-extension} for each $p_i$. Let $l^{\delta}_{i}$ be the mollified versions, and let $u^{\delta}_i$ be the corresponding finite time cell-problems with the same $t^{\delta}(x,\alpha)$ (for $i=1,2$). Then,~\Propref{prop:convergence-of-u-to-H-p} states that 
  \[ 
  \lim_{t \to \infty} \frac{u^{\delta}_i(x,s)}{t} = -\overline{H}^{\delta}(p_i) \quad \text{ for } i=1,2. 
  \]
  Now, $l^{\delta}_{1,c}$ and $l^{\delta}_{2,c}$ differ only in the tubes, and hence satisfy
  \[ 
    |l^{\delta}_{1,c}(x,r) - l^{\delta}_{2,c}(x,r)| \leq |p_1 - p_2||r| \quad \forall x,r \in \R^d. 
  \]
  It follows that the same inequality applies for the mollified versions $l_{i}^{\delta}$. So for any path $g(s)$,
  \[ 
  \begin{split}
    I^{\delta}_{p_1}(g) - I^{\delta}_{p_2}(g) & = \int_0^s l_1^{\delta}(g(r),g'(r)) - l_2^{\delta}(g(r),g'(r))dr \\
    & \leq |p_1 - p_2|d(g)  . 
  \end{split}
  \]
  Since the total length of the path is at most $d(g) \leq s/a$, the Lipschitz estimate follows.
\end{proof}

\begin{proof}[Proof of~\Propref{prop:lipschitz-continuity-of-time-constant}]
  Let $T^{\delta}(x,y)$ be the first-passage time from $x$ to $y$ in the \deltaapprox. Since
  \[ |T^{\delta}(0,[x]) - T^{\delta}(0,x)| \leq \frac{1}{\delta}, \]
  we have that
  \[ \lim_{n \to \infty} \frac{T^{\delta}(0,[nx])}{n} = m^{\delta}(x). \]
  For any $x,y \in \R^d$, using subadditivity, we get the estimate
  \[ 
%  \begin{split}
    T^{\delta}(0,[nx]) \leq T^{\delta}(0,[ny]) + T^{\delta}([ny],[nx]). \\
%            & b |[ny] - [nx]|,
%  \end{split}
  \]
  Then, using the fact that we can take an edge path from $[ny]$ to $[nx]$, and that the time to cross each edge is at most $b$, we get 
  \begin{equation*}
    T^{\delta}(0,[nx]) \leq T^{\delta}(0,[ny]) +  b |[ny] - [nx]|.\\
  \end{equation*}
Dividing by $n$ and taking a limit as $n \to \infty$ gives us the Lipschitz estimate.
\end{proof}

%% file: proof_variational_formula.tex
\chapter{Proofs related to the discrete variational formula}
\label{chap:proofs-related-to-discrete-variational-formula}
%%fakesection
\section{Proof of the discrete variational formula}
\label{sec:proof-of-discrete-variational-formula}
In the following, constants will all be called $C$ and will frequently change from line-to-line. We begin with the proof of~\Propref{prop:approximate-discrete-HJB-for-stationary-problem}, which says that the discrete stationary problem $\nu_{\e}$ approximately satisfies a HJB equation. 

\begin{proof}[Proof of~\Propref{prop:approximate-discrete-HJB-for-stationary-problem}]
  We will first derive a bound and a Lipschitz estimate for $\nu_{\e}(x)$. From the variational definition of $\nu_{\e}$ in~\Eqref{eq:discrete-stationary-control-problem}, it's easy to see that
  \[
    -\frac{|p|_{\infty}}{\e a} \leq \nu_{\e}(x) \leq -\frac{|p|_{\infty}}{\e b} \quad \forall~ x \in \Z^d
  \]
  where $b$ and $a$ are the upper and lower bounds on $\tau(x,\alpha)$ (see~\Eqref{eq:basic-assumption-on-edge-weights}). This simple upper bound can be used to derive a Lipschitz estimate
  \begin{claim}
    The functions $\nu_{\e}(x)$ satisfy (uniformly in $\e$) for all $x \in \Z^d$,
    \begin{align}
%      \begin{split}
\mathcal{H}(\nu_{\e},p,x) = & \sup_{\alpha \in A} \left( \frac{-\mathcal{D}\nu_{\e}(x,\alpha) - p\cdot\alpha}{\tau(x,\alpha)} \right) \leq \frac{1}{a}|p|_{\infty} , \label{eq:2-lipschitz-conditions-nu-epsilon-upper-bound}\\
        0 \leq & \sup_{\alpha \in A} \left( - \mathcal{D} \nu_{\e}(x,\alpha) - p \cdot \alpha \right) . \label{eq:2-lipschitz-conditions-nu-epsilon-lower-bound} 
%      \end{split}
%      \label{eq:2-lipschitz-conditions-nu-epsilon}
    \end{align}
    \label{claim:2-lipschitz-conditions-nu-epsilon}
  \end{claim}
  We will prove~\Claimref{claim:2-lipschitz-conditions-nu-epsilon} after completing the proof of the proposition. Applying~\Eqref{eq:2-lipschitz-conditions-nu-epsilon-upper-bound} at $x$ and $x + \alpha$ gives us the discrete Lipschitz estimate 
 \begin{equation}
   \Norm{\nu_{\e}}{Lip} \leq  \frac{a + b}{a} |p|_{\infty}. 
   \label{eq:discrete-lipschitz-estimate-for-stationary-problem-nu}
 \end{equation}
 
  Recall the DPP
 \[ 
   \nu_{\e}(x) = \inf_{\alpha} \left( \alpha \cdot p + e^{-\e \tau(x,\alpha)} \nu_{\e}(x + \alpha) \right).
 \]
 Expand the exponential in the DPP in a Taylor series, and use the bound on $\nu_{\e}$ to get
 \begin{align*}
    -C\e & \leq\nu_{\e}(x) + \sup_{\alpha \in A} \big( -\alpha \cdot p - (1-\e \tau(x,\alpha)) \nu_{\e}(x+\alpha) \big) \leq C \e ,\\
    -C\e & \leq \sup_{\alpha \in A} \big( -\alpha \cdot p - \mathcal{D}\nu_{\e}(x,\alpha) + \e \tau(x,\alpha) \mathcal{D}\nu_{\e} + \e \tau(x,\alpha) \nu_{\e}(x) \big) \leq C\e .
 \end{align*}
 Divide through by $\tau(x,\alpha)$, and then use the Lipschitz estimate on $\nu_{\e}$ and the bound on $\tau(x,\alpha)$ to get
 \begin{align*}
    -C\e & \leq \sup_{\alpha \in A} \tau(x,\alpha) \left( \frac{-\alpha \cdot p - \mathcal{D}\nu_{\e}(x,\alpha)}{\tau(x,\alpha)} + \e \nu_{\e}(x) \right) \leq C\e , \\
   -C\e & \leq \e \nu_{\e}(x) + \sup_{\alpha \in A} \left( \frac{-\alpha \cdot p - \mathcal{D}\nu_{\e}(x,\alpha)}{\tau(x,\alpha)} \right) \leq C\e.  
 \end{align*}
%\begin{align*}
%   -C\e & \leq \tau(x,a) \left( \frac{-\alpha \cdot p - \mathcal{D}\nu_{\e}(x,\alpha)}{\tau(x,\alpha)} + \e \nu_{\e}(x) \right) \leq C\e ,\\
%\begin{equation*}
%  -C\e \leq \e \nu_{\e}(x) + \left( \frac{-\alpha \cdot p - \mathcal{D}\nu_{\e}(x,\alpha)}{\tau(x,\alpha)} \right) \leq C\e.  
%\end{equation*}
%\end{align*}
%Taking a sup over $\alpha$ finishes the proof.
\end{proof}

\begin{proof}[Proof of~\Claimref{claim:2-lipschitz-conditions-nu-epsilon}]
  Using the DPP, we get for fixed $\alpha \in A$,
  \[
    \begin{split}
       \nu_{\e}(x)  & \leq p \cdot \alpha + e^{-\e \tau(x,\alpha)} \nu_{\e}(x + \alpha),\\
       \nu_{\e}(x)  & \leq p \cdot \alpha + (1 - \e \tau(x,\alpha)) \nu_{\e}(x + \alpha),\\
      %      \nu_{\e}(x) - \nu_{\e}(x + \alpha)  & \leq |p|_{\infty} + \frac{b}{a}|p|_{\infty},
       -\mathcal{D}\nu_{\e}(x,\alpha) - p\cdot\alpha  & \leq \frac{\tau(x,\alpha)}{a}|p|_{\infty}.
    \end{split}
  \]
  This proves the upper bound. The lower bound that will be useful in part~\ref{part:two}. Since $\nu_{\e}(x)$ is negative, for each $\alpha$,
 \[
   e^{-\e \tau(x,\alpha)} \nu_{\e}(x + \alpha) \geq \nu_{\e}(x + \alpha).
 \]
 Hence,
 \begin{equation}
   \begin{split}
     \nu_{\e}(x) \geq \inf_{\alpha \in A} p \cdot \alpha + \nu_{\e}(x + \alpha),\\
     \sup_{\alpha \in A} \left( - \mathcal{D} \nu_{\e}(x,\alpha) - p \cdot \alpha \right) \geq 0.
   \end{split}
 \end{equation}
\end{proof}

Before proving the comparison principle, we first finish the proof of the discrete variational formula in~\Thmref{thm:discrete-variational-formula-for-H-bar}. In the following proof we'll need to use probability, so we'll reintroduce $\w$ wherever necessary.
\begin{proof}[Proof of~\Thmref{thm:discrete-variational-formula-for-H-bar}]
%  \label{proof:proof-of-variational-formula}. This does not seem to work correctly, as it only picks up the section number.
  Let's first prove the upper bound. Let $\phi \in S$, where $S$ is defined in~\Eqref{eq:discrete-set-S-for-variationalf-formula}. Suppose $\phi$ is such that
  \[ 
    \sup_x \mathcal{H}(\phi,p,x,\w) < \infty ~\almostsurely .
  \]
  The form of $\mathcal{H}(\phi,p,x,\w)$ (see~\Eqref{eq:discrete-hamiltonian-1}) implies that it's coercive. Hence, we must have $\Norm{\phi}{Lip} < \infty$.  %  , for if not, it would contradict our assumption about $f$.
  Then, the comparison principle for the finite time-horizon cell-problem in~\Propref{prop:comparison-principle} gives
  \[ 
    \phi(x) - t \sup_x \mathcal{H}(\phi,p,x,\w) \leq \mu(x,t,\w) \quad\forall x \in \Z^d.
  \]
    Divide the inequality by $t$, take a limit as $t\to\infty$, and use~\Propref{prop:convergence-of-u-to-H-p} and~\Thmref{thm:commutation-theorem-for-u-and-T}. Then, rearrange the inequality and take a sup over $x$ to get
  \[ 
    \overline{H}(p) \leq \sup_{x \in \Z^d} \mathcal{H}(\phi,p,x,\w).
  \]
% Similarly, by setting $\nu_{\e}(x,t) = f(x) - t \inf \mathcal{H}(\mathcal{D}f)$, we can also get that
% \begin{equation}
%   \inf_x \mathcal{H}(\mathcal{D}f,x) \leq \overline{H}(p) \leq \sup_x \mathcal{H}(\mathcal{D}f,x) \quad \forall f \in S. 
%   \label{eq:comparison-inequality-for-equations-in-S}
% \end{equation}
% For homogenization, we usually require that $f$ is at the very least sublinear at infinity; this we have thanks to the fact that $f \in S$. 
 
% For the lower bound, we will prove that there is an approximating sequence in $S$; this will give us existence of a minimizer.

Now, consider the discrete stationary problem given by~\Eqref{eq:discrete-stationary-control-problem} with edge-weights $\lambda(x,\alpha) = p\cdot\alpha$. Using the discrete HJB equation for $\nu_{\e}$ from~\Propref{prop:approximate-discrete-HJB-for-stationary-problem}, we get
\[ 
  \e \nu_{\e}(x,\w) + \mathcal{H}(\nu_{\e},p,x,\w) \leq C \e  \quad \forall~ x \in \Z^d. 
\]
%It's known that (discrete viscosity) solution $\nu_{\e}$ solves an optimal-control problem (hopefully stated somewhere in this document; if not, see~\citet{bardi_optimal_1997}).
As in the proof of the continuous variational formula in~\Secref{sec:proof-of-homogenization}, we can normalize this set of functions so that they're zero at the origin. Letting $\hat{\nu}_{\e}(x,\w) = \nu_{\e}(x,\w) - \nu_{\e}(0,\w)$, we get 
\[ 
  \e \hat{\nu}_{\e}(x,\w) + \mathcal{H}(\hat{\nu}_{\e},p,x,\w) \leq C\e -\e \nu_{\e}(0,\w). 
\]
Using the definition of the discrete Hamiltonian, we get for each $\alpha \in A$,
\begin{equation}
  \e \hat{\nu}_{\e}(x,\w) + \frac{-p\cdot\alpha-\mathcal{D}\hat{\nu}_{\e}(x,\alpha,\w)}{\tau(x,\alpha,\w)} \leq C\e -\e \nu_{\e}(0,\w). 
  \label{eq:step-in-lower-bound-of-variational-formula}
\end{equation}
$\hat{\nu}_{\e}$ is normalized to zero at the origin, and inherits the discrete Lipschitz estimate on $\nu_e$~\Eqref{eq:discrete-lipschitz-estimate-for-stationary-problem-nu}. Hence,
\[ 
  C = \sup_{\e} \left\{ \Norm{\hat{\nu}_{\e}(y,\w)(1+|y|)^{-1}}{\infty} + \Norm{\mathcal{D}\hat{\nu}_{\e}}{\infty} \right\} < \infty .
\]
%Further, homogenization implies $-\e \nu_{\e}(0) \to \overline{H}(p)$. 

Let $\psi(\alpha,\w)$ be an $L^2$ weak limit of $\mathcal{D}\hat{\nu}_{\e}(0,\alpha,\w)$ (as $\e \to 0$) for each $\alpha \in \{e_1,\ldots,e_d\}$. With a slight abuse of notation, we use the translation group to define $\psi(x,\alpha,\w) = \psi(\alpha,V_x\w)$. Consider a control $\alpha \in \mathcal{A}$ such that for some $k > 0$, $\gamma_{\alpha,x}(k)=x$; i.e., it forms a loop. For fixed $\e$,
\[ 
  \sum_{i=0}^k \mathcal{D}\hat{\nu}_{\e}(\gamma_{\alpha,x}(i),\alpha(i)) = 0 \quad \almostsurely~\w. 
\]
Since the measure is translation invariant, each $\psi(\gamma_{\alpha,x}(i),\alpha(i),\w)$ is an $L^2$ weak limit of $\mathcal{D}\hat{\nu}_{\e}(\gamma_{\alpha,x}(i),\alpha(i),\w)$ for $i \leq k$. Then, for any $h \in L^2(\W)$, 
\[ 
  \begin{split}
    \lim_{\e \to 0} \int h(\w) \sum_{i =0}^k \mathcal{D}\hat{\nu}_{\e}(\gamma_{\alpha,x}(i),\alpha(i)) \Prob(d\w) & = \int h(\w) \sum_{i=0}^k \psi(\gamma_{\alpha,x}(i),\alpha(i),\w) \Prob(d\w) \\
    & = 0. 
  \end{split}
\]

Since there are only a countable number of loops and a countable number of points, $\psi$ sums over all loops at every location to zero almost surely. Hence, there is a function $\phi(x,\w)$ such that $\mathcal{D}\phi(x,\alpha,\w) = \psi(x,\alpha,\w)$. By $L^2$ weak-convergence and~\Eqref{eq:step-in-lower-bound-of-variational-formula}, we have for each fixed $x,\alpha$ and any nonnegative function $g(\w) \in L^2(\W)$,
\[ 
%\begin{split}
%  \lim_{\e \to 0} & \int g(\w) \left( \e \hat{\nu}_{\e} + \frac{-p\cdot\alpha-\mathcal{D}\hat{\nu}_{\e}(x,\alpha)}{\tau(x,\alpha)} \right) \Prob(d\w) \leq \lim_{\e \to 0} \int g(\w) \left( C\e -\e \nu_{\e}(0,\w) \right),\\
  \int g(\w) \left(\frac{p\cdot\alpha-\mathcal{D}\phi(x,\alpha,\w)}{\tau(x,\alpha,\w)} \right) \Prob(d\w) \leq \overline{H}(p) \int g(\w) \Prob(d\w). 
%\end{split}
\]
using~\Propref{prop:statement-of-tauberian-theorem}. We can take a supremum over $x \in \Z^d$ and $\alpha \in A$ to get
\[ 
  \sup_x \mathcal{H}(\phi,p,x,\w) \leq \overline{H}(p) \quad \almostsurely.
\]
  This proves the other inequality and completes the proof.
\end{proof}

%%fakesection
\section{Proof of the comparison principle}
%The only thing that remains to be proved is the comparison principle. 
The results in this section make no use of probability and hence we'll ignore the $\w$ dependence. Recall the discrete finite time-horizon variational problem from~\Eqref{eq:discrete-finite-time-horizon-problem}:
\begin{equation*}
  \mu(x,t) = \inf_{\alpha \in \mathcal{A}} \inf_{k \in \Z^+} \left\{ \sum_{i=0}^{k} \lambda(\gamma_{\alpha,x}(i),\alpha(i)) + \phi(\gamma_{\alpha,x}(k)) : \mathcal{T}(\gamma_{\alpha,x},k) \leq t \right\}. 
\end{equation*}
Since the time-parameter $t \in \R^+$ is continuous, the DPP for $\mu(x,t)$ is slightly different. 
\begin{prop}
  The DPP for $\mu(x,t)$ takes the form
  \[ 
  \mu(x,t) = \left\{ 
  \begin{array}{cc}
    \inf_{c \in A} \{ \mu(x+c,t-\tau(x,c)) + \lambda(x,c)\} & t \geq \min_{c \in A} \tau(x,c) \\
    \phi(x) & \text{otherwise}
  \end{array} \right. .
  \]
%  where $A$ is the set of $2d$-unit vectors of $\Z^d$ defined.
  \label{prop:DPP-for-discrete-variational-problem-for-comparison-principle}
\end{prop}
\begin{proof}
  If $t < \min_c \tau(x,c)$, then no neighbor of $x$ can be reached, and $\mu(x,t) = \phi(x)$. So assume that at least one neighbor $x + c$ can be reached. For fixed $c \in A$, consider the set of controls whose first step is in the $c$ direction: 
  \[ 
    \mathcal{B}_c := \{ \alpha \in \mathcal{A} : \alpha(0) = c \}. 
  \]
  There is an obvious map from $\mathcal{B}_c$ onto $\mathcal{A}$, obtained by shifting the control $\alpha(\cdot) \to \alpha(\cdot + 1)$ and forgetting the first step. It follows immediately that
  \[ 
    \mu(x,t) \leq \inf_{c \in \mathcal{A}} \mu(x+c,t-\tau(x,c)) + \lambda(x,c).
  \]
 For the opposite inequality, for any $\e > 0$ pick $\alpha$ and $k$ such that
 \begin{equation*}
   \begin{split}
     \mu(x,t) & \geq \sum_{i=0}^{k} \lambda(\gamma_{\alpha,x}(i),\alpha(i)) + \phi(\gamma_{\alpha,x}(k)) - \e, \\
     & \geq \lambda(x,x+\alpha(0)) + \mu(x,x+\alpha(0),t-\tau(x,\alpha(0))) - \e. 
   \end{split}
 \end{equation*}
\end{proof}

We next prove the comparison principle in~\Propref{prop:comparison-principle}.
\begin{define}[Reachable set]
  Following~\citet{bardi_optimal_1997}, define
  \begin{equation*} 
    R(x,t) := \{ y \in \Z^d : \mathcal{T}(x,y) \leq t \} 
%    R_{\delta}(x,t) := \{ y \in \R^d ~:~ T_{\delta}(x,y) \leq t \}
  \end{equation*}
  to be the set of sites that can be reached from $x$ within time $t$.
  \label{def:reachable-set-definition}
\end{define}

\begin{proof}[Proof of~\Propref{prop:comparison-principle}]
Let $\phi(x)$ have bounded discrete derivatives and define  
\begin{equation}
  \zeta(x,t) = \phi(x) - t \sup_x \mathcal{H}(\phi,p,x).
  \label{eq:definition-of-LHS-zeta-in-comparison-principle}
\end{equation}
%From the form of $\mathcal{H}$, it follows that $\sup_x \mathcal{H}(\phi,p,x) \geq 0$ and hence
%\[
%  \zeta(x,t) \leq \phi(x).
%\]
We need to show that $\zeta(x,t) \leq \mu(x,t)$. Let $N(x,t)$ be the cardinality of the reachable set $R(x,t)$. Suppose $N(x,t)$ jumps in value on a finite set of times contained in $(0,t]$. Then, $\mu(x,t)$ can only decrease at these times and remains constant otherwise. So it is enough to do an induction on this set of times to show that $\zeta(x,t) \leq \mu(x,t)$. However, it may well happen that $\mu(x,t)$ decreases on a possibly uncountable set of times, and the induction becomes harder to do. To handle this subtlety, we introduce a truncation of the problem. 

For large $K > 0$, we'll define a truncated variational problem $\mu_K(x,t)$ as follows. Let $B_K(0)$ be the ball of radius $K$ centered at the origin. Inside $Z_K := B_K(0) \cap \Z^d$, paths are allowed to wander freely, but once a path exits $Z_K$, it cannot move further. If a path starts in the set $\Z^d \setminus Z_K$, it cannot move at all. More succinctly, the set of control directions $A_K$ is
\[
  A_K := \left\{ \begin{array}{cc}
    A & \text{inside } Z_K \\
    \emptyset & \text{ otherwise}
  \end{array} \right. .
\]
Clearly $\mu_K(x) = \phi(x) ~\forall x \in \Z^d \setminus Z_K$ and for all $x \in Z^d$, 
\[ 
  \mu(x,t) \leq \mu_K(x,t). 
\]
It is easy to verify that $\mu_K$ satisfies the same DPP as $\mu$ for all $x \in Z_K$.
%\begin{equation}
%  \mu_K(x,t) = \left\{ 
%  \begin{array}{cc}
%    \inf_{c \in A} \{ \mu(x+c,t-\tau(x,c)) + \lambda(x,c)\} & t \geq \min_{c \in A} \tau(x,c) \\
%    \phi(x) & \text{otherwise}
%  \end{array} \right. .
%  \label{eq:dpp-for-muK-in-proof-comparison-principle}
%\end{equation}
For any fixed $x \in \Z^d$ and $t \in \R^+$, we must have $R(x,t) \subset Z_K$ for large enough $K$. Therefore for $K$ large enough, $\mu_K(x,t) = \mu(x,t)$. Hence, it's enough to show for any fixed $K > 0$ that
\[
    \zeta(x,t) \leq \mu_K(x,t) ~\forall~ x \in \Z^d.
\]

%Let
% \[ R_K(t) = \cup_{x \in Z_K} R(x,t) . \]
%  The truncation $\mu_K$ helps us handle this subtlety.
  
  We recursively define the sequence of times $\left\{ t_{x,k} \right\}_{k \in \Z^+}$ at which $R(x,s)$ increases in size for $s \leq t$ as follows:
 \[ 
   t_{x,k} = \inf \left\{ s \leq t : N(x,s) > N(x,t_{x,k-1}) \right\}, \quad t_{x,0} = 0 .
 \]
  Since $N(x,t) < \infty$, $t_{x,k}$ is finite only for a finite number of $k$; by convention, the infimum over an empty set is $+\infty$. Let $j(x) := \max \{ k : t_{x,k} < \infty \}$ be the last jump of $N(x,\cdot)$ before time $t$.
 
  Now, we look at the all the times at which the reachable set of any point $x \in Z_K$ expands. These are also all the possible times at which $\mu_K(x,s)$ can decrease for $s \leq t$. Order the finite set
 \[ \bigcup_{x \in Z_K} \bigcup_{k \leq j(x)} t_{x,k} =: s_1 \leq \ldots \leq s_N .\]
We will do induction on the ordered sequence $\{s_i\}$. Assume as the inductive hypothesis that $\zeta(x,r) \leq \mu_K(x,r) ~\forall~ x \AND r \leq s_{k-1}$. $\mu_K(x,r)$ does not decrease when $s_i < r < s_{i+1}$ because the reachable set $R_K(s)$ does not expand during this time. This implies that in fact,
\[ 
\zeta(x,r) \leq \mu_K(x,r) \quad ~\forall~ x \AND r < s_{k}.
\]
Let $C = \sup_{x \in \Z^d} \mathcal{H}(\phi,p,x)$. Then,
 \[ 
 \begin{split}
   I & := \sup_{c \in A} \left\{ \frac{\zeta(x,s_k) - \zeta(x+c,s_k-\tau(x,c)) - \lambda(x,c)}{\tau(x,c)} \right\}  \\
   & = \sup_{c \in A} \left\{ \frac{(\phi(x) - C s_k) - (\phi(x+\alpha) - C (s_k-\tau(x,c))) - \lambda(x,c)}{\tau(x,c)} \right\}  \\
   & = \sup_{c \in A} \left\{ \frac{-(\phi(x+c)-\phi(x)) - \lambda(x,c)}{\tau(x,c)} \right\} - C \\
   & = \mathcal{H}(\phi,p,x) - C \leq 0.
 \end{split}
 \]
 Since $\tau(x,y) > 0$, this means that for each $c$ in the sup in $I$, we have
 \[ 
   \zeta(x,s_k)-\zeta(x+c,t-\tau(x,c)) - \lambda(x,c) \leq 0 .
 \]
%But there is a $x+c$ for which equality holds, so this valid. 
Hence for all $x \in Z_K$,
 \[
 \begin{split}
   \zeta(x,s_k) & \leq \inf_{c \in A} \left\{ \zeta(x+c,s_k-\tau(x,c)) + \lambda(x,c) \right\},\\
   & \leq \inf_{c \in A} \left\{ \mu_K(x+c,s_k-\tau(x,c)) + \lambda(x,c) \right\} , \\ 
   & = \mu_K(x,s_k),
 \end{split}
 \] 
 where we've used the inductive hypothesis and the fact that $\mu_K$ also satisfies the DPP in~\Propref{prop:DPP-for-discrete-variational-problem-for-comparison-principle}. In case $s_k - \tau(x,c) < 0$ for all $c \in A$, 
 \[
   \mu_K(x,s_k) = \phi(x) \geq \zeta(x,s_k).
 \]
 Letting $K \to \infty$ completes the proof.
\end{proof}

\section{Solution of the HJB equation}
\label{sec:solution-of-HJB-equation}
Next, we prove that the limiting Hamiltonian is a norm on $\R^d$. 
\begin{proof}[Proof of~\Propref{prop:limiting-hamiltonian-is-a-norm}]
  Consider the variational formula again (dropping the sup over $x$ doesn't make a difference, see~\Eqref{eq:variational-formula-at-origin}):
\[
  \overline{H}(p) = \inf_{\phi \in S} \esssup_w \sup_{\alpha \in A} \frac{-\mathcal{D}\phi(0,\alpha,\w) - p\cdot \alpha}{\tau(0,\alpha,\w)},
\]
Replacing $\phi \mapsto \lambda \phi$ leaves $S$ invariant, and it follows that for $\lambda > 0,~ \overline{H}(\lambda p) = \lambda \overline{H}(p)$. 

For any fixed $\phi$, $E[p\cdot\alpha + \mathcal{D}\phi(0,\alpha)] = p \cdot \alpha$ and hence
\[
  \begin{split}
    \esssup_{\w \in \W} \sup_{\alpha} \left( -\mathcal{D}\phi(0,\alpha) - p\cdot\alpha \right) 
    & \geq \sup_{\alpha} E[-p\cdot\alpha - \mathcal{D}\phi(0,\alpha)] \geq |p|_{\infty}
  \end{split}
\]
Therefore, 
 \[ 
   \overline{H}(p) \geq \frac{|p|_{\infty}}{b}.
 \]
 Finally, the triangle inequality for $\overline{H}$ follows from the fact that for each fixed $p$, $-\mu(x,t)/t$ converges to $\overline{H}(p)$~(see~\Chapref{sec:outline-of-paper}). For any $p,q \in \R^d$, we have
 \[
   \begin{split}
   \sup_{\alpha \in \mathcal{A}} \sup_{k \in \Z^+} 
     & \left\{ \sum_{i=0}^{k} -(p+q)\cdot\alpha(i) - \mu_0(\gamma_{\alpha,x}(k)) : \mathcal{W}_{x,k}(\alpha) \leq t \right\} \\
     & \qquad \leq ~\sup_{\alpha \in \mathcal{A}} \sup_{k \in \Z^+} \left\{ \sum_{i=0}^{k} -p\cdot\alpha(i) - \mu_0(\gamma_{\alpha,x}(k)) : \mathcal{W}_{x,k}(\alpha) \right\} \\
     & \qquad \quad + \sup_{\alpha \in \mathcal{A}} \sup_{k \in \Z^+} \left\{ \sum_{i=0}^{k} -q\cdot\alpha(i) - \mu_0(\gamma_{\alpha,x}(k)) : \mathcal{W}_{x,k}(\alpha) \right\} .
   \end{split}
 \]
 Dividing by $t$ and taking a limit $t \to \infty$ shows that $\overline{H}$ satisfies the triangle inequality.
\end{proof}

\begin{proof}[Proof of \Propref{prop:Hopf-Lax-formula-for-u}]
We get $T(x) \leq L(x)$ by considering the straight line path from $0$ to $x$. In fact, the straight line is the minimizing path, as can be seen by an application of the triangle inequality for $L(\cdot)$.  
\end{proof}
%%%%%%%%%%%%%%%

%% file: recap.tex
\chapter{Recap and some basic observations}
\label{chap:recap}
In this section, we note some basic facts about the variational formula in~\Thmref{thm:discrete-variational-formula-for-H-bar}, and some simple corollaries of its proof in~\Secref{sec:proof-of-discrete-variational-formula}. First, note that the sup over $x$ can be dropped. That is, we can rewrite~\Eqref{eq:original-statement-of-discrete-variational-formula-in-intro} as
\begin{equation}
  \overline{H}(p) = \inf_{\phi \in S} \esssup_{\w \in \W} \mathcal{H}(\phi,p,0,\w) . 
  \label{eq:variational-formula-at-origin}
\end{equation}
This is a simple consequence of the fact that $\sup_x \mathcal{H}(\phi,p,x,\w)$ is translation invariant, and hence is a constant almost surely due to ergodicity.

We proved that the sequence of functions $\{\hat{\nu}_{\e}\}_{\e}$ defined in the proof of the variational formula in~\Secref{sec:proof-of-discrete-variational-formula} is minimizing. $\hat{\nu}_{\e}$ is a translate of $\nu_{\e}$, the value function of the stationary cell-problem with DPP~\Eqref{eq:discrete-dpp-stationary-problem}
\[
  \nu_{\e}(x) = \inf_{\alpha \in A} \left( \alpha \cdot p + e^{-\e \tau(x,\alpha)} \nu_{\e}(x + \alpha) \right).
\]
The DPP gave the following estimate in~\Claimref{claim:2-lipschitz-conditions-nu-epsilon}:
\begin{equation*}
  0 \leq \sup_{\alpha \in A} \left( -\mathcal{D}\nu_{\e}(x,\alpha) - p\cdot\alpha \right) \leq \frac{b}{a}|p|_{\infty}.
\end{equation*}
$\hat{\nu}_{\e}$ inherits this estimate, and this means that we may further restrict the set $S$ of functions~\Eqref{eq:discrete-set-S-for-variationalf-formula}. We state this as a corollary of the variational formula.
\begin{cor}[of~\Thmref{thm:discrete-variational-formula-for-H-bar}]
  The variational formula in~\Eqref{eq:original-statement-of-discrete-variational-formula-in-intro} holds with
  \begin{equation}
    S = \left\{ 
    \phi\colon\Z^d \times \Omega \to \R ~\left|~ 
    \begin{split}  
      & \mathcal{D}\phi(x + z,\w) = \mathcal{D}\phi(x,V_z\w), ~\forall x,z \in \Z^d \\
      & E[\mathcal{D}\phi(x,\alpha)] = 0 ~\forall x \in \Z^d \AND \alpha \in A , \\
      &0 \leq \sup_{\alpha \in A} \left( -\mathcal{D}\phi(x,\alpha) - p\cdot\alpha \right) \leq \frac{b}{a}|p|_{\infty} 
    \end{split} 
    \label{eq:discrete-set-S-for-variational-formula-with-lipschitz-constraint}
    \right\} \right.
  \end{equation}
  \label{cor:discrete-set-S-for-variational-formula-with-lipschitz-constraint}
\end{cor}

Now consider $\mu(x,t)$, the discrete finite time-horizon cell-problem. The discrete comparison principle in~\Propref{prop:DPP-for-discrete-variational-problem-for-comparison-principle} says that for any $\phi\colon\Z^d \to \R$ such that $\Norm{\phi}{Lip} < \infty$,
\[ 
  \mu(x,t) \geq \phi(x) - t \sup_{x \in \Z^d} \mathcal{H}(\phi,p,x) \quad \forall x,t.
\]
An almost identical proof ---which we will not repeat--- but with a bunch of inequalities reversed, gives the following proposition: 
\begin{prop}
  Suppose $\phi \in S$, where $S$ is defined in~\Eqref{eq:discrete-set-S-for-variational-formula-with-lipschitz-constraint}. Then,
  \[
    \mu(x,t) \leq \phi(x) - t \inf_{x \in \Z^d} \mathcal{H}(\phi,p,x) \quad \forall x,t.
  \]
  \label{prop:other-inequality-comparison-principle}
\end{prop}
Then, following the same argument in the proof of the variational formula, we get
\begin{cor}[of~\Thmref{thm:discrete-variational-formula-for-H-bar}]
  For each $\phi \in S$~\Eqref{eq:discrete-set-S-for-variational-formula-with-lipschitz-constraint},
  \begin{equation}
    %  \begin{split}
    \essinf_{\w \in \W} \inf_x \mathcal{H}(\phi,p,x,\w) \leq \overline{H}(p) \leq \esssup_{\w \in \W} \sup_x \mathcal{H}(\phi,p,x,\w) . 
%    \label{eq:original-statement-of-discrete-variational-formula-in-intro}
  \end{equation}
  \label{cor:inf-sup-inequality-for-limiting-Hamiltonian}
\end{cor}
\begin{proof}
  We've already proved the upper bound on $\overline{H}(p)$ in~\Secref{sec:proof-of-discrete-variational-formula}. Using the lower bound in~\Propref{prop:other-inequality-comparison-principle}, we have for any $x \in \Z^d$,
  \[ 
    \mu(x,t) \leq \phi(x) - t \inf_x \mathcal{H}(\phi,p,x,\w) 
  \]
  Next, we divide the inequality by $t$, and take a limit as $t \to \infty$. Using $\mu(x,t)/t \to -\overline{H}(p)$ as $t \to \infty$ (\Propref{prop:convergence-of-u-to-H-p} and \Thmref{thm:commutation-theorem-for-u-and-T}) we get
  \[ 
    \inf_{x \in \Z^d} \mathcal{H}(\phi,p,x,\w) \leq \overline{H}(p) \quad\almostsurely.
  \]
\end{proof}

\begin{define}[Discrete corrector]
  For some constant $C$, if $\phi \in S$ satisfies
  \[
    \esssup_{\w \in \W} \mathcal{H}(\phi,p,x,\w) =  C \quad \almostsurely.,
  \]
  $\phi$ is called a corrector for the variational formula.
\end{define}
This definition is consistent with the definition of corrector in continuum homogenization theory~\citep{lions_homogenization_1987,lions_correctors_2003}; i.e., it's a function that solves the discrete cell-problem. If $\phi$ is a corrector, then~\Corref{cor:inf-sup-inequality-for-limiting-Hamiltonian} tells us that it's a minimizer of the variational formula.
%Some authors prefer to write the set of functions $S$ that we optimize over in the variational formula as follows:
%\[

%% file: abstract_algorithm.tex
\chapter{Explicit algorithm to produce a minimizer}
\label{chap:explicit-algorithm-to-produce-a-minimizer}
%Next, we show how to explicitly calculate the limit shape when a special symmetry is present. 
Let  $\{ V_{e_1}, \ldots, V_{e_d} \}$ be commuting, invertible, measure-preserving ergodic transformations on $\W$. They generate the group of translation operators in~\Eqref{eq:translation-group-of-operators} under composition. Suppose we have \FPP~on the undirected graph on $\Z^d$, i.e.,
\begin{equation}
  \tau(x,\alpha,\w) = \tau(x + \alpha, - \alpha,\w).
  \label{eq:edge-FPP-assumption}
\end{equation}
Let $A_+ = \{e_1, \ldots, e_d\}$. Let $t\colon A_+ \times \W \to \R$ be a function representing the edge-weight at the origin. For example, it could consist of $d$ i.i.d.~edge-weights, one for each direction. Let $x=(x_1,\ldots,x_d) \in \Z^d$. Then, the edge-weight function is given by
\[ 
  \tau(x,\alpha,\w) = t\left(\alpha,V_{e_1}^{x_1}\cdots V_{e_d}^{x_d}\w \right). 
\]
In this section, we will assume the following symmetry on the medium: 
\begin{equation}
  V_{e_1} = \cdots = V_{e_d} = V .
  \label{eq:symmetry-assumption-on-medium}
\end{equation}
\
This means that for each $\w$ the function $\tau(\cdot,\cdot,\w)$ is constant along the hyperplanes $\{ x \in \Z^d : \sum_{i=1}^d x_i = z \}$ for each $z \in \Z$. Despite this symmetry, the medium is still quite random, and it's not so obvious ---although one ought to be able to calculate it--- what the time-constant is. However, the set $S$ in~\Eqref{eq:discrete-set-S-for-variationalf-formula} is tremendously simplified. 
\begin{prop}
  If $\phi \in S$ and~\Eqref{eq:symmetry-assumption-on-medium} holds,
  the derivative points in the $\sum_i e_i$ direction; i.e., 
  \[
    \mathcal{D}\phi(x,\alpha,\w) = \mathcal{D}\phi(0,e_1,\w) \quad \forall \alpha \in A \AND ~\forall x \in \Z^d ~\almostsurely.
  \]
  \label{prop:simplifying-the-set-of-functions-S}
\end{prop}
\begin{proof}
  The derivative of $\phi$ sums to $0$ over any discrete loop in $\Z^d$. In particular, for any $i \neq j \in \{1,\ldots,d\}$
  \begin{multline}
    \mathcal{D}\phi(x,e_i,\w) + \mathcal{D}\phi(x+e_i,e_j,\w) + \\  + \mathcal{D}\phi(x+e_i+e_j,-e_i,\w) + \mathcal{D}\phi(x+e_j,-e_j,\w)
    = 0.
  \end{multline}
 Since the derivative is stationary and $\mathcal{D}\phi(x,\alpha,\w) = -\mathcal{D}\phi(x+\alpha,-\alpha)$, we have
  \[
    \begin{split}
      \mathcal{D}\phi(x,e_i,\w)  - \mathcal{D}\phi(x,e_j,\w) & = \mathcal{D}\phi(x,e_i,V_{e_j}\w) - \mathcal{D}\phi(x,e_j,V_{e_i}\w), \\
      & = \mathcal{D}\phi(x,e_i,V\w) - \mathcal{D}\phi(x,e_j,V\w).
    \end{split}
  \]
  Hence, $\mathcal{D}\phi(0,e_i,\w) - \mathcal{D}\phi(0,e_j,\w)$ is invariant under $V$. Since it also has zero mean, it follows from ergodicity that 
  \[
    \mathcal{D}\phi(x,\alpha,\w) = \mathcal{D}\phi(x,e_1,\w) \quad \forall \alpha \in A ~\almostsurely.
  \]
%  Then, the result follows from~\Thmref{thm:discrete-variational-formula-for-H-bar}.
\end{proof}

Next, we simplify the variational formula under the symmetry assumption~\Eqref{eq:symmetry-assumption-on-medium}. Redefine the discrete Hamiltonian for $t \in \R$, $p \in \R^d$ to be
\begin{equation}
  \mathcal{H}_{sym}(t,p,\w) := \sup_{\alpha \in A^+} \frac{|t + p\cdot\alpha|}{\tau(0,\alpha,\w)} . 
  \label{eq:discrete-hamiltonian-2}
\end{equation}
%\label{adas}
\begin{prop}
  If we assume~\Eqref{eq:symmetry-assumption-on-medium} and~\Eqref{eq:edge-FPP-assumption}, the variational formula becomes
  \begin{equation}
    \overline{H}(p) = \inf_{f \in F} \esssup_w \mathcal{H}_{sym}(f(\w),p,\w),
  \end{equation}
  where
  \begin{equation}
    F := \left\{ f:\W \to \R ,~ E[f] = 0,~ \sup_{\alpha \in A^+} \left|f + p\cdot\alpha \right| \leq (b/a)|p|_{\infty} \right\}.
    \label{eq:set-S-under-symmetry-assumption}
  \end{equation}
  \label{prop:variational-formula-in-symmetric-situation}
\end{prop}
\begin{proof}
  For each $\phi \in S$, let $f(\w) = D\phi(0,e_1,\w)$. The proof is easy using the assumption~\Eqref{eq:edge-FPP-assumption} on the edge-weights and~\Propref{prop:simplifying-the-set-of-functions-S}. To wit,
\[
  \begin{split}
    \esssup_{\w} \sup_x \mathcal{H}(\phi,p,x,\w) 
    & = \esssup_{\w} \sup_x \sup_{\alpha \in A} \left\{ \frac{-\mathcal{D}\phi(x,\alpha,\w) - p\cdot \alpha}{\tau(x,\alpha,\w)} \right\}, \\
    & = \esssup_{\w} \sup_{\alpha \in A^+} \left\{ \frac{|\mathcal{D}\phi(0,e_1,\w) + p\cdot \alpha|}{\tau(0,\alpha,\w)} \right\}, \\
    & = \esssup_{\w} \mathcal{H}_{sym}(f(\w),p,\w).
  \end{split}
\]
%  Using $\mathcal{D}f(x,-\alpha,\w) = \mathcal{D}f(x-\alpha,\alpha,\w)$ and~\Propref{prop:simplifying-the-set-of-functions-S} in the following, we get
%  \[
%    \begin{split}
%      \esssup_{\w} \sup_x \mathcal{H}(f,p,x,\w) 
%      & = \esssup_{\w} \sup_x \sup_{\alpha \in A} \left\{ \frac{-\mathcal{D}f(x,\alpha,\w) - p\cdot \alpha}{\tau(x,\alpha,\w)} \right\},  \\
%      & = \esssup_{\w} \sup_x \sup_{\alpha \in A^+} \left\{ \frac{|-\mathcal{D}f(x,\alpha,\w) - p\cdot \alpha|}{\tau(x,\alpha,\w)} \right\} , \\
%%      & = \esssup_{\w} \sup_{\alpha \in A^+} \left\{ \frac{|-\mathcal{D}f(0,e_1,\w) - p\cdot \alpha|}{\tau(0,\alpha,\w)} \right\} \\
%      & = \esssup_{\w} \mathcal{H}_{sym}(g(\w),p,\w)
%    \end{split}
%  \]
%  where $g(\w) = \mathcal{D}f(0,e_1,\w)$.
%Under the symm$f(x,\w) \mapsto \mathcal{D}f(0,e_1,\w)$ is a bijection from $S$ to $F$.
\end{proof}
In the following, we will write $\mathcal{H}_{sym}(f,\w)$ and drop reference to $p$ since it's irrelevant to our arguments. We present an algorithm that produces a minimizer for the variational problem under the symmetry assumption. The idea behind the algorithm is simple. At each iteration, we try to reduce the essential supremum over $\w$ by modifying $f(\w)$, while simultaneously keeping it inside the set $F$. If it fails to reduce the sup, we must be at a minimizer. We explain what we're trying to do in each step in the proof of convergence of the algorithm. So we suggest skimming the definition of the algorithm first, and returning to the definition of each step when reading the proof.\newline

\noindent\textbf{Start algorithm}
\vspace{-1em}
\begin{enumerate}
  \item \label{step:find-distance-between-sup-and-mean} Start with any $f_0 \in F$, for example, $f_0 = 0$. Let $\mu_0 = E[\mathcal{H}_{sym}(f_0,\w)]$, and let 
    \[
    d = \esssup_{\w \in \W} \mathcal{H}_{sym}(f_0,\w) - \mu_0.
    \]
    If $d = 0$, stop.
  \item \label{step:define-min-set-S-and-I} Define the sets
    \begin{align}
      MIN_0 & :=  \{ \w : \mathcal{H}_{sym}(f_0,\w) = \min_x \mathcal{H}_{sym}(x,\w)  \},\\
      S & :=  \{ \w : \mathcal{H}_{sym}(f_0,\w) > \mu_0 \},\\
      I & :=  \{ \w : \mathcal{H}_{sym}(f_0,\w) < \mu_0 \}.
      \label{eq:definition-of-the-sets-MIN-S-I}
    \end{align}
    If 
    \[ \esssup_{\w \in MIN_0} \mathcal{H}_{sym}(f_0,\w) = \esssup_{\w \in \W} \mathcal{H}_{sym}(f_0,\W), \]
    stop.
  \item \label{step:construct-delta-f} Let $\Delta f^*(\w)$ be such that 
    \[ \mathcal{H}_{sym}(f_0 + \Delta f^*(\w),\w) = \min_x \mathcal{H}_{sym}(x,\w). \]
    Define the sets
    \begin{align*}
      S_+ & :=  \{ \w \in S \setminus MIN_0 : D_+ \mathcal{H}_{sym} \subset (-\infty,0) \},\\
      S_- & :=  \{ \w \in S \setminus MIN_0 : D_- \mathcal{H}_{sym} \subset (0,\infty) \},
    \end{align*}
    % D\mathcal{H}_{sym}(f_0,\w), D\mathcal{H}_{sym}(f_0,\w) $D\mathcal{H}_{sym}$
    where $D_+ \AND D_-$ are the left and right derivatives of the convex function $\mathcal{H}_{sym}(\cdot,\w)$. Let 
    \[ 
      \Delta f(\w) = \left\{ \begin{array}{cc}
        \max\left(- a (\mathcal{H}_{sym}(f_0,\w) - \mu_0),~\Delta f^*(\w) \right) 
        & \w \in S_+ \\
        \min\left( a (\mathcal{H}_{sym}(f_0,\w) - \mu_0),~\Delta f^*(\w) \right) 
        & \w \in S_-  \\
        a \xi (\mu_0 - \mathcal{H}_{sym}(f_0,\w) ) & \w \in I \\
        0   & \text{elsewhere}
      \end{array} \right. ,
    \]
    where 
    \[ 
    \xi = - \frac{ \int_{S_+ \cup S_-} \Delta f(\w)~\Prob(d\w)}{ \int_I  a (\mu_0 - \mathcal{H}_{sym}(f_0,\w) )~\Prob(d\w)}.
    \]
    Let $f_1 = f_0 + \Delta f(\w)$. Return to step $1$.
%    and define the corresponding set $MIN_1$ defined above for $f_0$.
%\item \label{step:check-if-new-min} If 
%    \[ 
%      \esssup_{\w \in \W} \mathcal{H}_{sym}(f_1,\w) > \esssup_{\w \in \W} \mathcal{H}_{sym}(f_0,\w) - \frac{da}{2b}, 
%    \]
%      stop. If not, go back to the first step with $f_0 = f_1$ and repeat. %\todo{This step is not necessary, move it to a claim in the proof}
% Here, $a = \min \tau(\cdot,\cdot)$ and $b = \max(\cdot,\cdot)$.
\end{enumerate}
\noindent\textbf{End algorithm}

\begin{theorem}
  There are three possibilities for the algorithm:
  \begin{enumerate}
    \item If it terminates in a finite number of steps with $d = 0$, we have a minimizer that's a corrector.
    \item If it terminates in a finite number of steps with $d > 0$, we have a minimizer that's not a corrector
    \item If it does not terminate, we produce a corrector in the limit.
  \end{enumerate}
  \label{thm:algorithm-convergence}
\end{theorem}
We need the following lemma to prove~\Thmref{thm:algorithm-convergence}.
%We first begin by defining~\citet{rockafellar_convex_1970}
%\begin{define}[subdifferential]
%For a function $f:\R^d \to \R$, let 
%\[
%  \mathcal{\mathcal{D}}f(x_0) = \{c \in \R^d : f(x) - f(x_0) \geq c\cdot(x- x_0) \quad \forall x \in \R^d \}
%\] 
%be the subdifferential of $f$ at point $x_0$. An element of the $\mathcal{D}f(x_0)$ is called a subderivative.
%\end{define}
\begin{lemma} \mbox\newline %forces the list onto a new line
  The function $\mathcal{H}_{sym}(x,\w)$ has the following properties:
  \begin{enumerate}
    \item For each $\w$, it is convex in $x$.
    \item It has a unique measurable minimum $x^*(\w)$.
    \item Its left and right derivatives satisfy $D_{-} \mathcal{H}_{sym}(x,\w) \in [b^{-1},a^{-1}]$ or $D_+ \mathcal{H}_{sym} \in [-a^{-1},-b^{-1}]$ a.s. $\w$. 
  \end{enumerate}
  \label{lem:H-sym-is-convex-and-its-min-is-measurable}
\end{lemma}
We will prove~\Lemref{lem:H-sym-is-convex-and-its-min-is-measurable} after proving~\Thmref{thm:algorithm-convergence}.
%With these results, we can prove~\Thmref{thm:algorithm-convergence}. 
  \begin{figure}
    \centering
    \includegraphics[height=0.4\textheight]{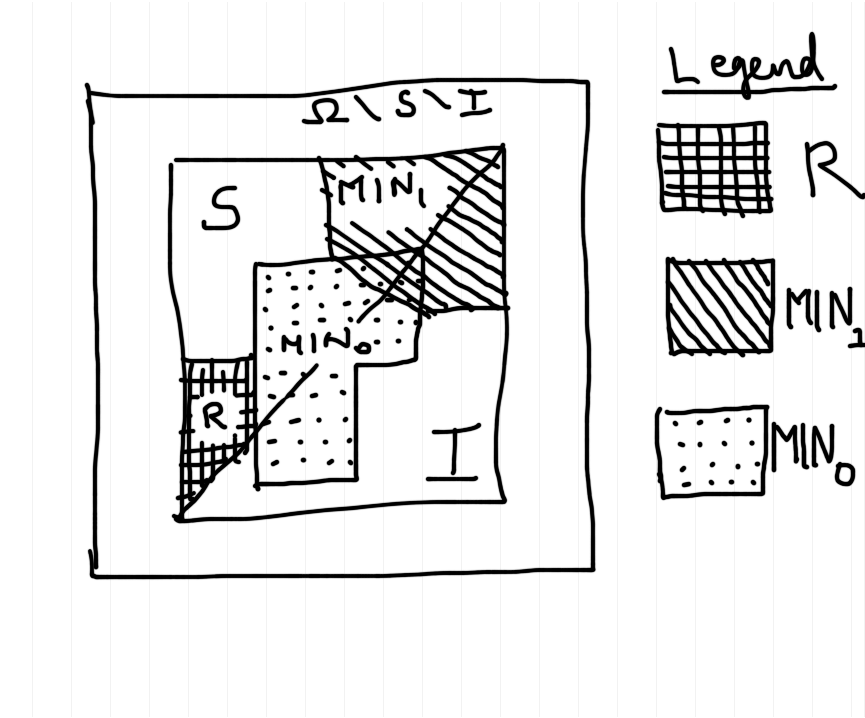}
    \caption{Sketch of sets in algorithm. The outermost square represents the probability space $\W$. $S$ is the upper triangle in the inner square, and $I$ is the lower triangle in the inner square. $\W \setminus S \setminus I$ is the annular region between the two squares.}
    \label{fig:sets-in-algorithm}
  \end{figure}

\begin{proof}[Proof of~\Thmref{thm:algorithm-convergence}]
%  We will complete the proof of the following Proposition before proving~\Propref{lem:H-sym-is-convex-and-its-min-is-measurable}.
%  First, we claim that
%\begin{claim}
%  all the functions in the algorithm are measurable, and each step is well-defined.
%  \label{claim:algorithm-is-well-defined}
%\end{claim}
%  \Claimref{claim:algorithm-is-well-defined} is routine definition checking and as usual, we'll postpone its proof until after we've proved the theorem.
\mbox{}\newline %to get off the proof line.
\noindent $1$. In the first step, we compute $d$, the distance between the mean and supremum of $\mathcal{H}_{sym}(f,\w)$. If $d = 0$, $f$ must be a corrector and from~\Corref{cor:inf-sup-inequality-for-limiting-Hamiltonian}, it must be a minimizer. Therefore, we stop the algorithm.

\noindent $2$. $MIN_0$ is the set on which $\mathcal{H}(f_0,\w)$ cannot be lowered further. $S$ and $I$ are the sets on which $\mathcal{H}_{sym}(f_0,\w)$ is bigger and lower than its mean $\mu_0$. $f$ will be modified on these two sets in step 3. 

Lemma~\ref{lem:uniform-in-n-delta-estimate} says that $\mathcal{H}_{sym}(\cdot,\w)$ is convex and has a minimum. So there is the possibility of the algorithm getting ``stuck'' at a minimum of $\mathcal{H}_{sym}$. That is, $f_0$ might be such that $\mathcal{H}_{sym}(f_0,\w) = \mathcal{H}_{sym}(x^*(\w),\w)$ on a set of positive measure, and also  
    \[ 
      \esssup_{\w \in MIN_0} \mathcal{H}_{sym}(f_0,\w) = \esssup_{\w \in \W} \mathcal{H}_{sym}(f_0,\W). 
    \]
    For any other $g \in F$, we clearly have $\mathcal{H}_{sym}(g(\w),\w) \geq \mathcal{H}_{sym}(f_0(\w),\w)$ on $MIN_0$. Hence $f_0$ must be a minimizer, and we stop the algorithm.

  \noindent $3$. $\Delta f$ is first defined on the sets $S_+$ and $S_-$ so that the supremum falls. Then, $\Delta f$ is defined on $I$ so that it satisfies 
    \[
%      \begin{split}
        E[\Delta f] = 0.
%      \end{split}
    \]
    We need to make sure that $\xi$ is not infinite. Notice that $E[(\mu_0 - \mathcal{H}_{sym}(f_0,\w) ), I] > 0$; for if not, $\essinf_{\w \in I} \mathcal{H}_{sym}(\w) = \mu_0$. Hence $\Delta f$ is well-defined on $I$. We derive a useful estimate on $\xi$ next. Since
  \begin{equation*}
    \left|E[ \left( \mathcal{H}_{sym}(f_0,\w) - \mu_0 \right) , S \setminus MIN_0 ] \right| \leq 
    E[ \left( \mu_0 - \mathcal{H}_{sym}(f_0,\w) \right) , I ] ,
  \end{equation*}
  we have 
  \[
    \begin{split}
      E[\Delta f, S_+ \cup S_-]  
%      & = - a E( \max\left( \mathcal{H}_{sym}(f_0,\w) - \mu_0) \Delta f^*(\w)\right) , S_+] \\
%      & \quad ~  + E[a \min\left(\mathcal{H}_{sym}(f_0,\w) - \mu_0), \Delta f^*(\w) \right), S_-] \\
      & \leq a E[(\mathcal{H}_{sym}(f_0,\w) - \mu_0),S_+ \cup S_-],  \\
%      & \leq a E[(\mathcal{H}_{sym}(f_0,\w) - \mu_0),S_+ \cup S_- \cup I] - a E[(\mathcal{H}_{sym}(f_0,\w) - \mu_0),I] , \\
      & \leq a E[\mu_0 - \mathcal{H}_{sym}(f_0,\w), I] .
    \end{split}
  \]
  Therefore,
  \begin{equation}
   -1 < \xi  < 1  .
    \label{eq:algorithm-xi-is-bounded}
  \end{equation}
%  and hence
%  \begin{equation}
%    |\Delta f| \leq a d 
%    \label{<++>}
%  \end{equation}<++>
  Finally, we prove that if the algorithm does not terminate in either step~\ref{step:find-distance-between-sup-and-mean} or~\ref{step:define-min-set-S-and-I}, we produce a corrector in the limit. We claim that if at the end of step $3$ of the algorithm, $\esssup \mathcal{H}(f_1)$ does not fall enough, the algorithm will terminate at the next step. 
  \begin{claim}
    If  
    \begin{equation}
      \esssup_{\w \in \W} \mathcal{H}_{sym}(f_1,\w) > \esssup_{\w \in \W} \mathcal{H}_{sym}(f_0,\w) - \frac{da}{b}, 
      \label{eq:suppose-that-H-f1-does-not-fall-enough}
    \end{equation}
  the algorithm will terminate when it goes to step $2$ in the following iteration. That is,
  \[
    \esssup_{\w \in \W} \mathcal{H}_{sym}(f_1,\w) = \esssup_{MIN_1}\mathcal{H}_{sym}(f_1,\w),
  \]
  where $MIN_1$ is defined in~\Eqref{eq:definition-of-the-sets-MIN-S-I} with $f_0$ replaced by $f_1$.
  \label{claim:intermediate-termination-possibility-if-sup-does-not-fall-enough}
  \end{claim}

  Now suppose that the algorithm does not terminate, and let $f_n$ be the $n$\textsuperscript{th} iterate. Claim~\ref{claim:intermediate-termination-possibility-if-sup-does-not-fall-enough} gives us the estimate
  \[
    \esssup_{\w \in \W} \mathcal{H}_{sym}(f_n,\w) \leq \esssup_{\w \in \W} \mathcal{H}_{sym}(f_{n-1},\w) - d_n \frac{a}{b}.
  \]
  Since $\mathcal{H}_{sym} \geq 0$, we must have $d_n \to 0$. Since for all $n$, we have
  \[ 
    \esssup_{\w \in \W} \mathcal{H}_{sym}(f_n,\w) < \esssup_{\w \in \W} \mathcal{H}_{sym}(f_0,\w),
  \]
  the coercivity of $\mathcal{H}_{sym}$ implies that $f_n$ must be bounded uniformly in $n$. By our construction, $\{f_n\}$ is a bounded martingale with respect to the filtration $\mathcal{F}_n = \sigma(f_1,\ldots,f_{n})$. Hence by the martingale convergence theorem, $f_{\infty}(\w) = \lim_n f_n(\w)$ exists a.s., and further the convergence is uniform in every $L^p$ norm. Then, by the continuity of the $\mathcal{H}$, its nonnegativity, and its uniform boundedness on compact sets, we get for any $p \in \R$,
  \[ 
    \begin{split}
      0 = & \lim_n d_n \\
      = & \lim_{n \to \infty} \esssup \mathcal{H}_{sym}(f_n(\w),\w) - \int \mathcal{H}_{sym}(f_n(\w),\w),\\
      \geq & \Norm{\mathcal{H}_{sym}(f_{\infty},\w)}{p} - \int \mathcal{H}_{sym}(f_{\infty},\w).\\
    \end{split}
  \]
  Taking $p \to \infty$ proves that $f_{\infty}$ is a corrector. This completes the proof except for~\Claimref{claim:intermediate-termination-possibility-if-sup-does-not-fall-enough}. We prove this next.
%  By our construction, we claim that
%  \begin{claim}
%    \[ 
%      E[|\Delta f_n(\w)|] \leq a d_n.
%    \]
%    \label{claim:bound-on-increments-of-f-in-algorithm}
%  \end{claim}
\end{proof}

% This is the claim that everything checks out in the algorithm.
% Apr 22 2014 I removed this claim since I think its too trivial to state now.
%\begin{proof}[Proof of~\Claimref{claim:algorithm-is-well-defined}]
%  This is mostly routine definition checking, and it mostly follows from~\Propref{lem:H-sym-is-convex-and-its-min-is-measurable}. Step~\ref{step:find-distance-between-sup-and-mean} is non-controversial. 
%
%  Since $\min_x \mathcal{H}_{sym}(x,\w) = \mathcal{H}_{sym}(x^*(\w),\w)$ is a measurable function of $\w$, $MIN_0$ is a measurable set and all the quantities in step~\ref{step:define-min-set-S-and-I} are well-defined. 
%  
%  \Propref{lem:H-sym-is-convex-and-its-min-is-measurable} also implies that $\Delta f^*$ is measurable, since $\Delta f^*(\w) = x^*(\w) - f_0(\w)$. Since $D_+ \mathcal{H}_{sym}$ and $D_- \mathcal{H}_{sym}$ are measurable functions from $\R^d \times \Omega \to \R$, $S_+$ and $S_-$ are measurable.
%
% The following was also removed quite early, too much definition checking.
%  For fixed $\w$, let $B_t = \{ \alpha \in A : \mathcal{H}_{sym}(t,p,\w) = |t + p\cdot\alpha|/\tau(0,\alpha,\w) \}$. 
%  \begin{claim}
%    Let $\tau^*(t,\w) = \min\{ \tau(0,\alpha) : \alpha \in B_t\}$. 
%    \[
%      \mathcal{D}\mathcal{H}_{sym}(x,\w) = \frac{1}{\tau^*(t,\w)}.
%    \]
%  Further $\mathcal{D}\mathcal{H}_{sym}(x,\w)$ is a closed subset of $[b^{-1},a^{-1}]$ or $[-a^{-1},-b^{-1}]$.
%    \label{claim:formula-for-H-sym-derivative}
%  \end{claim}
  
%  In step~\ref{step:find-distance-between-sup-and-mean}, there is nothing to check.
%\end{proof}

\begin{proof}[Proof of~\Claimref{claim:intermediate-termination-possibility-if-sup-does-not-fall-enough}]
  Let 
  \[
    R := \{ \w \in S \setminus MIN_0 : a~|\mathcal{H}_{sym}(f_0,\w) - \mu_0)| < \Delta f^*(\w) \},
  \]
  be the set on which we can modify $f_0$ without hitting the minimum of $\mathcal{H}(f_0,\w)$; i.e., $\mathcal{H}_{sym}(f_1,\w) > \mathcal{H}_{sym}(x^*(\w),\w)$. By the definition of $\Delta f$ and the bound on the derivatives of $\mathcal{H}(\cdot,\w)$ in~\Propref{lem:H-sym-is-convex-and-its-min-is-measurable}, we have 
  \begin{multline*}
    \mathcal{H}(f_0,\w)  - \frac{1}{a}a(\mathcal{H}_{sym}(f_0,\w) - \mu_0) \leq \mathcal{H}(f_1,\w) \\ \leq  \mathcal{H}(f_0,\w) - \frac{1}{b} a (\mathcal{H}_{sym}(f_0,\w) - \mu_0)  
    \qquad \w \in R ~\almostsurely.
  \end{multline*}
  Therefore,
  \begin{equation}
    \mu_0 \leq  \esssup_{\w \in R} \mathcal{H}(f_1,\w) \leq  \esssup_{\w \in R}  \mathcal{H}(f_0,\w) - \frac{da}{b}.
    \label{eq:bound-on-H-f1-on-set-R}
  \end{equation}
  %  It's also certainly true that on $R$, 
  %  \[
  %    \esssup_{\w \in R} \mathcal{H}(f_1,\w) \geq \esssup_{\w \in R} \mathcal{H}(f_0,\w)  - \frac{da}{a} = \mu_0.
  %  \]
  Similarly for $\w \in I$, we use the bound on $\xi$ in~\Eqref{eq:algorithm-xi-is-bounded} to get
  \begin{equation}
    \mathcal{H}_{sym}(f_1,\w) \leq \mathcal{H}_{sym}(f_0,\w) + \xi (\mu_0 - \mathcal{H}_{sym}(f_0,\w)) \leq \mu_0 \quad \w \in I ~\almostsurely .
    \label{eq:bound-on-H-f1-on-set-I}
  \end{equation}
  From the definition of $R$, it follows that $S \setminus R \subset MIN_1$. Since $\Delta f = 0$ on $S \cap MIN_0$, we must have $S \cap MIN_0 \subset MIN_1$. 

  Consider the condition in~\Eqref{eq:suppose-that-H-f1-does-not-fall-enough} again. Equations~\eqref{eq:bound-on-H-f1-on-set-R} and~\eqref{eq:bound-on-H-f1-on-set-I} imply that we can ignore the sets $R$ and $I$ when taking a sup over $\W$. It's clear that we can ignore $\W \setminus S \setminus I$ too, since $\mathcal{H}(f_0,\w) = \mu_0$ on this set. Summarizing, we get 
  \[
    \begin{split}
      \esssup_{\W} \mathcal{H}_{sym}(f_1,\w) & = \esssup_{S \setminus R ~\cup~ MIN_0 \cap S} \mathcal{H}_{sym}(f_1,\w) = \esssup_{MIN_1} \mathcal{H}(f_1,\w). 
    \end{split}
  \] 
  Hence, the algorithm terminates at step $2$ in the next iteration.
\end{proof}

%\begin{proof}[Proof of~\Claimref{claim:bound-on-increments-of-f-in-algorithm}]
%  It's clear that
%  \begin{equation*}
%    \begin{split}
%      E[\mu_0 - \mathcal{H}_{sym}(f_0,\w), I] & = E[\mathcal{H}_{sym}(f_0,\w) - \mu_0 , S] \\ 
%      & \leq \esssup \mathcal{H}_{sym}(f_0,\w) - \mu_0 = d_n.
%    \end{split}
%  \end{equation*}
%  Using the fact that $|\xi| < 1$ from~\Eqref{eq:algorithm-xi-is-bounded}, we get $|\Delta f| \leq a d_n$.
%\end{proof}
%
To finish, we complete the proof of~\Lemref{lem:H-sym-is-convex-and-its-min-is-measurable}. 
\begin{proof}[Proof of~\Lemref{lem:H-sym-is-convex-and-its-min-is-measurable}]
%  For $\alpha \in A_+$, let 
%  \[
%    r_{\alpha}(t) = \frac{|t + p\cdot\alpha|}{\tau(0,\alpha,\w)}.
%  \]
  Since
  \begin{equation*}
    \mathcal{H}_{sym}(t,p,\w) = \sup_{\alpha \in A^+} \frac{|t + p\cdot\alpha|}{\tau(0,\alpha,\w)},
  \end{equation*}
  Clearly $\mathcal{H}_{sym}$ is convex. Its minimum is unique since $\tau(0,\alpha,\w) \geq a$, and therefore, it cannot have a ``flat spot'' parallel to the $t$-axis. 

  Now, $\mathcal{H}_{sym}$ can only take its minimum at a minimum of $|t + p\cdot\alpha|/\tau(0,\alpha,\w)$ or when $t$ is such that $|t + p\cdot\alpha|/\tau(0,\alpha_1,\w) = |t + p\cdot\alpha|/\tau(0,\alpha_2,\w)$ for any $\alpha_1,\alpha_2 \in A_+$. There are only a finite number of such possibilities, we can compute all of them, and hence its easy to see that $x^*(\w)$ is measurable. 

  The fact that $D_{-} \mathcal{H}_{sym}(t,\w) \in [b^{-1},a^{-1}]$ or $D_+ \mathcal{H}_{sym}(t,\w) \in [-a^{-1},-b^{-1}]$ for all $t$, follows easily from the form of $\mathcal{H}_{sym}$.
\end{proof}

%The algorithm has an important implication for numerical computations. This is as explicit as it gets. We can solve this by our favorite numerical method instead of the rather inefficient abstract algorithm.  
%\begin{remark}
Suppose the vector $\vec{t}(\w) = (t(e_1,\w), \ldots, t(e_d,\w))$ takes at most a finite number of different values $\{\vec{t}_0,\ldots,\vec{t}_{n-1}\} =: \W_0$. Let our probability space be $\W = (\R^d)^{\Z}$, let $\tau(z,\cdot,\w) = \w_{z_1}$ ($z_1$ is the first coordinate of $z$), and let the marginal of $\Prob$ on any coordinate of $\W$ be supported on $\W_0$. 

We show that even if $\Prob$ is a product measure, under the symmetry assumption, the structure of the problem is nearly equivalent to a periodic medium. Define the sets
\[ 
  A_i := \{ \w \in \W : \tau(0,\cdot,\w) = \vec{t}_i \}, \quad i=1,\ldots,n-1. 
\]
 The set $F$ of functions in~\Eqref{eq:set-S-under-symmetry-assumption} can be restricted to
 \begin{equation}
   F := \left\{ f(\w) : f(\w) = \sum_{i=0}^{n-1} f_i 1_{A_i}(\w),~f_i \in \R,~E[f] = 0 \right\},
   \label{eq:set-of-functions-F-in-variational-formula}
 \end{equation}
 and the algorithm continues to produce a minimizer. Now suppose we have a periodic medium with equal periods in both directions; i.e., the translations satisfy
 \[
    \tau(0,\cdot,V^n \w) = \tau(0,\cdot,V \w) ~\almostsurely.
 \]
 in addition to~\Eqref{eq:symmetry-assumption-on-medium}. Periodicity only forces the additional constraint $\Prob(A_i) = 1/n$, and except for this, the problem is nearly unchanged. Periodic homogenization has been well-studied and there are many algorithms to produce the effective Hamiltonian; see for example,~\citet{gomes_computing_2004} or~\citet{oberman_homogenization_2009}. 
 
 Our algorithm works even if $\tau(0,\cdot,\w)$ takes an uncountable number of values; i.e., the period is infinite. Notice that what we have here is an $n$-dimensional deterministic convex minimization problem with linear constraints (see~\Propref{prop:variational-formula-in-symmetric-situation} and~\Eqref{eq:set-of-functions-F-in-variational-formula}). It's worth stating (without proof, of course) that our algorithm is computationally much faster than conjugate gradient and other standard constrained optimization methods.

 \begin{remark}
   The symmetry assumption is a massive simplification, and removing this is a real challenge. If the translations $V_i$ are rationally related, we ought to be able to generalize the algorithm with a little work. However, taking this route ---solving the loop/cocycle condition--- in general is probably hopeless. It appears that working on an instance $\w \in \W$ of the probability space would be the most convenient way to proceed, since we can work directly with a function $f\colon\Z^d \to \R$ (instead of its derivative) and forget the cocycle condition.
 \end{remark}

%But this means that
%\[ \overline{H}(p) = \min_{f \in F} \max\{ \mathcal{H}(f_1), \ldots, \mathcal{H}(f_n) \}, \]
%is a finite dimensional constrained convex optimization problem in $\R^n$.  

%% file: comparing_two_distributions.tex
\chapter{Comparing two distributions}
\label{chap:comparing-two-distributions}
%As is frequently done in ergodic theory, we might as well assume that our probability space is $([0,1],\mathcal{F},\Leb)$ (see Petersen or Caratheodory 1939). 
\section{A simple coupling based argument}
For $i=1,2$, let $(\Omega_i,\mathcal{F}_i,\Prob_i)$ be probability spaces, and let $\tau_i\colon\Z^d \times A \times \W \to \R$ be two edge-weight functions. Assume that
\begin{equation*}
  \begin{split}
    0 <~& a_i = \essinf_{x,\alpha,\w} \tau_i(x,\alpha,\w), \\
    & b_i = \esssup_{x,\alpha,\w} \tau_i(x,\alpha,\w) < \infty. 
  \end{split}
\end{equation*}

We wish to compare $m_1(x)$ and $m_2(x)$, the corresponding time-constants. There is an elementary argument to obtain a very basic estimate between the two time-constants\footnote{told to me by M. Damron}. We will reproduce it using the variational formula to highlight the duality in the problem.

It will be easier to compare the two \FPP~problems if they're both on the space $(\R^{d})^{\Z^d}$, and we first show that we can always assume this. Consider the map $M\colon\W \to (\R^{d})^{\Z^d}$ defined as
\begin{equation}
  M(\w) = (\tau(V_z \w))_{z \in \Z^d}.
  \label{eq:map-M-of-omega-into-R-infinity}
\end{equation}
Let 
\begin{equation}
  \mathcal{C} = \sigma(M) \subset \mathcal{F}
  \label{eq:sigma-algebra-C-generated-by-M}
\end{equation}
be the sigma-algebra generated by $M$. We next show that it's enough to consider functions $f \in S$ that are measurable with respect to $\mathcal{C}$.
\begin{prop}
  Assume that every $f \in S$ (defined in~\Corref{cor:discrete-set-S-for-variational-formula-with-lipschitz-constraint}) also satisfies the additional condition that $f(x,\w)$ is $\mathcal{C}$ measurable for each $x \in \Z^d$. Then, the variational formula in~\Eqref{eq:variational-formula-at-origin} is unchanged.
  \label{prop:functions-measurable-wrt-sigma-tau}
\end{prop}
The proof is a simple consequence of convexity and can be found in~\Secref{sec:some-proofs-from-comparing-two-distros}.

%For this, we specialize to the setting where each edge-weight is chosen independently. Let
With~\Propref{prop:functions-measurable-wrt-sigma-tau}, it's easy to show that pushing the problem forward to the space $(\R^{2d})^{\Z^d}$ does not change the first-passage percolation problem. Let $\textrm{Im}(M) \subset (\R^{2d})^{\Z^d}$ be the image of $\W$ under the map $M$. Let $\Prob_M$ be the push-forward measure of $\Prob$ under $M$. It is enough to show that for each $f \in S$, there is a $g\colon(\R^{2d})^{\Z^d} \to \R$ such that
\[
  g(M(\w)) = f(\w) \quad \almostsurely .
\]
By~\Propref{prop:functions-measurable-wrt-sigma-tau}, we can assume that $f$ is $\sigma(M)$ measurable and hence by an elementary measurability lemma (see, for example~\citet{williams_probability_1991}) there is a function $g$ as required above. 

Therefore, we will henceforth assume that $\W = (\R^{d})^{\Z^d}$, $\mathcal{F}$ is the infinite product $\sigma$-algebra, $\Prob_1 \AND \Prob_2$ are the probability measures, the group of translations are just shift maps, and 
\[
  t_i(\w) := (\tau_i(0,\alpha,\w_i))_{\alpha \in A},
\] 
is the first coordinate of $\w$.

A coupling of two measures $\Prob_1$ and $\Prob_2$ is a probability measure on the product space $\W \times \W$ with product sigma-algebra $\mathcal{F} \times \mathcal{F}$ and $\Prob_1 \AND \Prob_2$ as marginals. Let $\Pi(\W \times \W)$ be the space of all couplings on $\W \times \W$. 
\begin{define}[a type of $W^{\infty}$ Wasserstein distance]
  \[
    d(\Prob_1,\Prob_2) = \inf_{\pi \in \Pi(\W \times \W)} \esssup_{(\w_1,\w_2) \in \W \times \W} \sup_{\alpha \in A} |t(\w_1) - t(\w_2)|.
  \]
  \label{define:l-infinity-coupling-distance}
\end{define}
% Apr 22 2014 I had this when I had the old definition for d(\W_1,\W_2).
%When the spaces and edge-weight functions are understood, we will use the abbreviation
%\[
%  d(\Prob_1,\Prob_2) = d( (\W,\mathcal{F}_2,\Prob_1),(\W,\mathcal{F}_2,\Prob_2)).
%\]
%This is a metric on the space of probability measures, but we will not need this fact here. Is this a metric? 

%%Does a minimizing coupling exist in general? If $\W = \W = [a,b]$, i.e., they're compact, then it certainly does. The theory of optimal transport worries about things like this a lot. 
%We restate a standard result (which we reprove in~\Secref{sec:some-proofs-from-comparing-two-distros}) about the existence of a minimizing coupling.
%\begin{prop}[Existence of minimizing coupling]
%  If~$\W_i$ are Polish and nonatomic, and $d(P_1,P_2) < \infty$, there is a minimizing coupling.
%  \label{prop:existence-of-minimizing-coupling}
%\end{prop}

The primal version of the comparison result is easily proved:
\begin{prop}
  For all $x \in \R^d$,
  \begin{equation*}
    | m_1(x) - m_2(x) | \leq \max\left( \frac{b_1}{a_1}, \frac{b_2}{a_2} \right) d(\Prob_1,\Prob_2) |x|_1 .
  \end{equation*}
  \label{prop:distance-estimate-between-time-constants-primal-argument}
\end{prop}
\begin{proof}
  Let $\gamma$ be a path connecting the origin to $[nx]$, let $\pi$ be a coupling, and let $d(\gamma)$ be the $l^1$ length of the path. Then,
  \[
    \sum_{i=1}^{d(\gamma)} | \tau_1(\gamma_i,\gamma_{i+1} - \gamma_i,\w_1) - \tau_2(\gamma_i,\gamma_{i+1} - \gamma_i,\w_2) | \leq d(\gamma) \esssup_{\w_1,\w_2} |t(\w_1) - t(\w_2)|_{\infty}.
  \]
  Since we can always take a shortest $l^1$ distance path between $0$ and $[nx]$, its enough to consider paths from $0$ to $[nx]$ that satisfy 
  \[
    d(\gamma) \leq \max\left( \frac{b_1}{a_1}, \frac{b_2}{a_2} \right) |[nx]|_1.
  \]
  Take an inf over all such paths to get
  \[
    |T_1([nx]) - T_2([nx])| \leq  \max\left( \frac{b_1}{a_1}, \frac{b_2}{a_2} \right) \esssup_{\w_1,\w_2} |t(\w_1) - t(\w_2)|_{\infty} |[nx]|_1
  \] 
  Divide by $n$ and take a limit as $n \to \infty$. Taking an infimum over couplings $\pi$, we get the result. 
\end{proof}

We can prove a similar version of~\Propref{prop:distance-estimate-between-time-constants-primal-argument} using just the variational formula.
\begin{prop}
  For all $x \in \R^d$,
  \begin{equation*}
    | m_1(x) - m_2(x) | \leq \max\left(\frac{b_1}{a_1},\frac{b_2}{a_2}\right) \frac{b_1 b_2}{a_1 a_2} d(\Prob_1,\Prob_2) |x|_1 
  \end{equation*}
  \label{prop:distance-estimate-between-time-constants-dual-argument}
\end{prop}
With the following lemma, the proof of~\Propref{prop:distance-estimate-between-time-constants-dual-argument} is easy.
\begin{lemma}
  Let $H_1$ and $H_2$ be the corresponding limiting Hamiltonians. Then,
  \[ 
    |\overline{H}_1(p) - \overline{H}_2(p)| \leq \max\left(\frac{b_1}{a_1},\frac{b_2}{a_2}\right) |p|_{\infty}\frac{1}{a_1 a_2} d(\Prob_1,\Prob_2)
  \]
    \label{lem:distance-estimate-between-Hamiltonians}
\end{lemma}
%We complete the proof of before proving~\Lemref{lem:distance-estimate-between-Hamiltonians}.
\begin{proof}[Proof of~\Propref{prop:distance-estimate-between-time-constants-dual-argument}]
%  Before proving the lemma, we can establish the inequality for the time-constant by using convex duality. 
  We established the elementary inequality
    \[
      \overline{H}_i(p) \geq \frac{|p|_{\infty}}{b_i}
    \]
    in the proof of Proposition~$3.9$ in part I. Hence for each $x,p \in \R^d$,
    \[
      \begin{split}
        \left| \frac{p\cdot x}{\overline{H}_1(p)} - \frac{p\cdot x}{\overline{H}_2(p)} \right|
        & = \left| p\cdot x \left( \frac{\overline{H}_1(p) - \overline{H}_2(p)}{\overline{H}_1(p) H_2(p)} \right) \right| \\
%        & \leq |p_{\infty}|^2 |x|_1 \max\left(\frac{b_1}{a_1},\frac{b_2}{a_2}\right) \frac{b_1 b_2}{a_1 a_2} \frac{1}{|p|_{\infty}^2} d(\Prob_1,\Prob_2)\\
        & \leq \max\left(\frac{b_1}{a_1},\frac{b_2}{a_2}\right) \frac{1}{a_1 a_2} |x|_1 d(\Prob_1,\Prob_2).
      \end{split}
    \] 
    We've used the H\"older inequality and~\Lemref{lem:distance-estimate-between-Hamiltonians} in the above computation. Since $m_i$ are the dual norms of $\overline{H}_i$, the proof is complete.
\end{proof}
\begin{proof}[Proof of~\Lemref{lem:distance-estimate-between-Hamiltonians}]
  First, fix $f \in S$, where $S$ is defined in~\Eqref{eq:discrete-set-S-for-variational-formula-with-lipschitz-constraint}. For each measure $\Prob_1 \AND \Prob_2$, the constraint on $S$ is different: 
  \[ 
    \sup_{\alpha \in A} \left( - \mathcal{D}f(x,\alpha) - p \cdot \alpha \right) \leq \frac{b_1}{a_1}|p|_{\infty} \qquad \forall x \in \Z^d.
  \]
  Hence, we might as well assume that 
  \[
    \sup_{\alpha \in A} \left| \mathcal{D}f(x,\alpha) + p \cdot \alpha \right| \leq \max\left(  \frac{b_1}{a_1}, \frac{b_2}{a_2}\right) \qquad \forall x \in \Z^d.
  \]
  %    and $\alpha \in A$, 
  Then, for a fixed coupling $\pi \in \Pi(\W \times \W)$, 
    \begin{align*}
      % the vphantom allows me to set the height correctly so I can break the left and right things and insert a & in between them.
      \left|  %begin vphantom
      \vphantom{\frac{\mathcal{D}f(0,\alpha,\w_1) + p\cdot\alpha}{\tau_1(0,\alpha,\w_1)} - \frac{\mathcal{D}f(0,\alpha,\w_2) + p\cdot\alpha}{\tau_2(0,\alpha,\w_2)}} %end vphantom
      \frac{\mathcal{D}f(0,\alpha,\w_1) + p\cdot\alpha}{\tau_1(0,\alpha,\w_1)} -
      \right. & \left. %begin vphantom
      \vphantom{\frac{\mathcal{D}f(0,\alpha,\w_1) + p\cdot\alpha}{\tau_1(0,\alpha,\w_1)} - \frac{\mathcal{D}f(0,\alpha,\w_2) + p\cdot\alpha}{\tau_2(0,\alpha,\w_2)}} %end vphantom
      \frac{\mathcal{D}f(0,\alpha,\w_2) + p\cdot\alpha}{\tau_2(0,\alpha,\w_2)} \right| \\
      & \leq \max\left( \frac{b_1}{a_1},\frac{b_2}{a_2} \right) |p|_{\infty} \frac{|\tau_1(0,\alpha,\w_1) - \tau_2(0,\alpha,\w_2)|}{\tau_1(0,\alpha,\w_1)\tau_2(0,\alpha,\w_2)} \\
      & \leq \max\left(  \frac{b_1}{a_1}, \frac{b_2}{a_2}\right) |p|_{\infty}\frac{1}{a_1 a_2} \esssup_{\w_1,\w_2} |t(\w_1) - t(\w_2)|_{\infty},
    \end{align*}
  using~\Corref{cor:discrete-set-S-for-variational-formula-with-lipschitz-constraint}. Since this is true for all functions in $S$, and all couplings in $\Pi(\W \times \W)$, we can take supremums and infimums as appropriate to get the result.
\end{proof}

\begin{remark}
  The estimate through the variational formula in~\Propref{prop:distance-estimate-between-time-constants-dual-argument} is worse than the estimate in~\Propref{prop:distance-estimate-between-time-constants-primal-argument}. However, the estimates used ---the lower bound for $\overline{H}(p)$ and the bound for $\nu_{\e}$ from~\Eqref{eq:2-lipschitz-conditions-nu-epsilon-upper-bound}--- were quite crude, and these are easily improved. 
\end{remark}  
\begin{remark}
  The basic step in the primal argument was to take the worst case path in the $x$ direction, and the corresponding step in the dual argument was to take the worst case function $f$ in the $p$ direction. This seems to indicate some (nonlinear) duality between paths on the lattice and functions in $S$. Is there a structural theory of this duality? 
\end{remark}

\section{A more convenient coupling distance}
The coupling distance in~\Defref{define:l-infinity-coupling-distance} is not very useful in general. However, when the medium is i.i.d, it's easy to get an upper bound for it in terms of a more familiar distance on the marginal distribution of the edge-weight $\tau(0,\alpha,\w)$. When $\Prob = \mu^{\otimes \Z^d}$, where $\mu$ is a measure on $\R^{d}$, couplings on $\R^d \times \R^d$ can be turned into a coupling on $\W \times \W$ by taking a product. Suppose further that the marginal measure $\mu$ on $\R^d$ is also an i.i.d.~product measure, and let $F_i$ be the cumulative distribution function of $\tau_i(0,\alpha,\w)$ for $i=1,~2$. Then, for example, we can write $d(\Prob_1,\Prob_2)$ in terms of the Kolmogorov-Smirnov distance between $F_1 \AND F_2$, assuming $F_1 \AND F_2$ are nice enough. 

Let $F_i$ have density $\rho_i$, and assume 
\[
\begin{split}
  \min(\text{supp}(F_i)) = & [a_i,b_i] \subset (0,\infty), \\
  \rho^* = & \min_i \min_{a_i \leq x \leq b_i} \rho_i(x) > 0,
\end{split}
\]
where supp denotes the support of the distribution. Let 
\[
    d_{Kol}(F_1,F_2) = \sup_x |F_1(x) - F_2(x)|
\]
be the Kolmogorov-Smirnov distance between the two distributions.

% I commented this out, but there is nothing to it.
%For my own satisfaction, we can write a proof down since it takes two lines. Approximate $f$ by a simple function
%\[
%  \phi(\w) = \sum_i a_i 1_{A_i}.
%\]
%Since each $A_i \in \sigma(M)$, there are $B_i \in \mathcal{F}_M$ such that $A_i = M^{-1}(B_i)$. Form the corresponding $g$ using the indicators of the sets $A_i$ and finally let $\phi(\w) \to f(\w)~\almostsurely$. It is enough to define $g$ on $\Ran(M)$ since $\Ran(M)$ has full measure. I will probably remove this from the main text. 
%
We will use the standard Skorokhod representation to define random variables $Y_i(x)$ on $([0,1],\mathcal{F},\Leb)$ with distributions $F_i$. Since $\rho^{\ast} > 0$, $F_i$ is strictly monotone, and hence we define
%Let $Y_i(x) = \inf \{ z : F_i(z) \geq x \}$. 
\begin{define}[Skorokhod representation of edge-weights]
  \[
    Y_i(x) = F^{-1}_i(x).  
  \]
  \label{def:skorokhod-representation-of-tau}
\end{define}
%\begin{theorem}
%  For all $x \in \R^d$,
%  \begin{equation*}
%    | m_1(x) - m_2(x) | \leq 2 \frac{b_1~b_2}{a_1~a_2}\frac{d_{Kol}(F_1,F_2)}{\rho^*} |x|_1 .
%  \end{equation*}
%  \label{prop:distance-estimate-between-time-constants-dual-argument}
%\end{theorem}
It's clear that $\Leb(Y_i(x) \leq c) = \Leb(F^{-1}_i(x) \leq c) = F(c)$. 
\begin{prop}
  \[ 
  |Y_1(x) - Y_2(x)| \leq \frac{d_{Kol}(F_1,F_2)}{\rho^*} 
  \]
  \label{prop:distance-between-embedded-tau1-and-tau2}
\end{prop}
\begin{proof}
  Fix $s \in [0,1]$, let $x = F_1^{-1}(s)$, and use $d = d_{Kol}(F_1,F_2)$ as shorthand. 
% Assume wlog that $F_1^{-1}(s) \leq F_2^{-1}(s)$ and 
  From the definition of the Kolmogorov-Smirnov distance,
  \[
  F_2(x)  \geq F_1(x) - d = s - d.
  \]
  If $r \geq d/\rho^* + x$, we have
  \[
%  \begin{split}
    F_2(r) \geq (s-d) + \rho^* (r - x) = s.
%  \end{split}
  \]
  It follows that 
  \[
    F^{-1}_2(s) \leq \frac{d}{\rho^*} + F_1^{-1}(s) .
  \]
  Repeating the argument for $x = F_2^{-1}(s)$, we get the result.
\end{proof}
Let $\pi$ be the coupling on $\R \times \R$ defined as the pushforward measure of $\Leb([0,1] \times [0,1]$ under the map $(x_1,x_2) \mapsto (Y_1(x_1),Y_2(x_2))$. We can take a product of $\pi$ to get a coupling on $\W \times \W$. Thus, we have proved that
\[
  d(\Prob_1,\Prob_2) \leq d_{Kol}(F_1,F_2).
\]
%The inequalities in ~\Eqref{eq:inf-condition-for-optimality} and~\Eqref{eq:sup-condition-for-optimality} may not become an inequality in general, but it's easi 
%\begin{equation*}
%  (p_1 \mu_{A_r} + p_2 \mu_{A_r^c}) r  \leq (p_1 - p_2) (\tau_{A_r}^1 + \tau_{{A_r}}^2).
%\end{equation*}
%For example, when $\tau(e_1)$ and $\tau(e_2)$ are distributed uniformly on $(1,2)$, this produces the cubic
%\[ 
%  \frac{p_1 + p_2}{2}r = (p_1 - p_2)\left( \frac{4}{3} - \frac{3}{2}r^2 - \frac{2}{3}r^3 \right) 
%\]
%Which we plot as,

%% file: future_work.tex
\chapter{Future work}
There are several areas to explore, and we've listed a few below.

\begin{enumerate}
  \item An interesting challenge is to remove the symmetry constraint in the algorithm. Since the idea behind the algorithm is so simple, it's reasonable to believe that it can be generalized.
  \item Another related question is to find a rich enough subclass of problems (Hamiltonians) where correctors exist. Can the algorithm be tuned to produce correctors for these problems? Questions about regularity and strict convexity appear more accessible if the existence of correctors can be guaranteed. 
  \item What does the existence of correctors for the cell-problem tell us about the percolation problem? 
  \item As stated in~\Secref{sec:generalization}, there are several possible generalizations of our work. It's probably quite easy to remove the bounds on $\tau(\cdot,\cdot,\cdot)$ in~\Eqref{eq:basic-assumption-on-edge-weights} and replace it with a moment condition. Other types of lattices and other control problems can also be explored. For example, the directed versions of first-passage percolation results has a monotone Hamiltonian, and these appear to be easier to work with.
  \item It will be interesting to explore the behavior of the so-called integrable models under the variational formula. This has already been begun in the context of last-passage percolation and polymer models by~\citet{georgiou_variational_2013}.
\end{enumerate}

%% file: miscellaneous_proofs.tex
\chapter{Miscellaneous Proofs}
\section{Proof of continuum homogenization with $G=\Z^d$}
\label{sec:proof-of-homogenization}
\citet{lions_homogenization_2005} consider a much more general version of the homogenization theorem stated in~\Thmref{thm:lions-homogenization-theorem}: the problem includes a ``viscous'' second-order term, the Hamiltonian can depend on $u$, and it can have an unhomogenized variable. Their general version of~\Propref{prop:convergence-of-u-to-H-p} requires the analysis of an equation of the form
%As stated in~\Secref{sec:intro-fpp},
\begin{align}
  & U^{\e}_t - \e \tr A(\e^{-1}y,\w)~ D^2 U^{\e}(y) + H(p+D U^{\e}, \e^{-1}y, \w) = 0 \textrm{ in } \R^N \times (0,T] \\
  & U^{\e} = u_0 \text{ in } \R^N \times \{0\}
  \label{eq:lions-soug-main-homo-problem}
\end{align}
where $A(y,\w)$ is a symmetric matrix, and $T > 0$. 
%for each $y,\w$.

They first prove the theorem assuming that $A$ and $H$ are ``nice'', and then obtain the general version of the theorem through penalization arguments. The specifics can be found in~\citet{lions_homogenization_2005}. Their general result includes our case of interest: $A = 0$, and $H(p)$ given by~\Eqref{eq:H-for-general-continuous-FPP}. 

When $H$ and $A$ are assumed to be nice, $H(p)$ satisfies the assumptions in~\Secref{sec:results-from-stochastic-homogenization}, grows super-quadratically in $|p|$, and $A$ is uniformly elliptic; i.e., for positive constants $C_1$ and $C_2$,
    \[ C_1 |\xi|^2 \leq (A\xi,\xi) \leq C_2 |\xi|^2.\]
    Then, a (special) supersolution $\tilde{U}^{\e}(y,t,\w)$ of~\Eqref{eq:lions-soug-main-homo-problem} has a representation in terms of a value function $L(y,y';s,t)$ of a stochastic control problem~\citep{fleming_stochastic_1985}. Let 
  \[ 
    L^{\e}(y,y';s,t) = \e L(\e^{-1}y,\e^{-1}y';\e^{-1}s,\e^{-1}t).
  \]
  Then, for $u_0 \in C^{1,1}(\R^d)$, 
\begin{equation}
  \tilde{U}^{\e}(y,t,\w) = \inf_{y'} \{ u_0(y') + L^{\e}(y,y';s,t) \}. 
  \label{eq:supersolution-representation-of-U-epsilon}
\end{equation}
%The function $L^{\e}(y,y;s,t,\w)$ is a scaling of the function $L(y,y;s,t)$ which is 

$L(y,y';s,t)$ has the following properties:
\begin{enumerate}
  \item Stationarity: for all $y \in \R^d$ and $g \in \R^d$,
  \begin{equation}
      L(y+g,y'+g;s,t,\omega) = L(y,y';s,t,V_g\omega).
      \label{eq:L-stationarity}
  \end{equation}
  \item Uniform Continuity: (Prop. $6.12$ in~\citet{lions_homogenization_2005})
Fix any $R > 0 \AND h > 0$. Then, $L^{\e}$ is uniformly continuous with respect to $(y,t), (y',s)$ where $h \leq s \leq t \leq T$ and $|y - y'| \leq R$, uniformly in $\e$ and $\w$.
  \item Boundedness: (follows from Prop. $6.9$ from~\citet{lions_homogenization_2005} and an elementary estimate) \newline 
    For all $(y,y',s,t) \in \R^d \times \R^d \times [0,T] \times [0,T]$, there exist independent of $\w$ and $\e$, constants $C_1,C_2,C_3 > 0$, and $k \in (1,2)$ such that 
    \[ L^{\e}(y,y';s,t) \leq C_1[|y-y'|^k (t-s)^{1-k} + \e ^{k/2} (t-s)^{1-k/2} + (t-s)],\]
    and 
    \[  L(y,y';s,t) \geq C_2 |y-y'| - C_3 (t-s). \]
  \item Subadditivity: for all $y,y',z \in \R^d$ and $0 < s < \tau < t$,
    \[ L(y,y';s,t) \leq L(y,z;s,\tau) + L(z,y';\tau,t).\]
\end{enumerate}
\citet{lions_homogenization_2005} use the subadditive ergodic theorem from~\citet{dal_maso_nonlinear_1986} to prove
\begin{prop}
\begin{equation*}
  \lim_{\e \to 0} L^{\e}(y,y';0,t;\w) = t\overline{L}\left(\frac{y'-y}{t}\right).
\end{equation*}
  \label{prop:L-epsilon-convergence-x-xhat-subadditive-theorem}
\end{prop}

When $G=\Z^d$, we can use the uniform continuity of $L^{\e}$, and the discrete subadditive ergodic theorem~\citep{kingman_ergodic_1968,akcoglu_ergodic_1981} to prove~\Propref{prop:L-epsilon-convergence-x-xhat-subadditive-theorem}. We first fix $y, y' \in \Q^d$, and $s,t \in \Q$ such that $0 < s \leq t \leq T$. Then, we can apply the classical subadditive ergodic theorem~\citep{kingman_ergodic_1968} on the subsequence $\e = n^{-1}$. The continuity estimates for $L^{\e}(y,y';s,t)$ takes care of the rest. We will not repeat this standard argument here; a version of this argument, for example, appears in~\citet{seppalainen_exact_1998}. 

\Propref{prop:L-epsilon-convergence-x-xhat-subadditive-theorem} allows~\citet{lions_homogenization_2005} to take a limit $\e \to 0$ in~\Eqref{eq:supersolution-representation-of-U-epsilon}. They show that the error between the supersolution $\tilde{U}^{\e}$ and the actual solution $U^{\e}$ remains small, and hence~\Propref{prop:convergence-of-u-to-H-p} follows. The rest of the proof of the homogenization theorem does not make use of specifics of the translation group.

\begin{remark}
  \citet{lions_homogenization_2005} do not state~\Thmref{thm:lions-homogenization-theorem} in the metric form; they state it for the stationary equation in~\Eqref{eq:stationary-problem-pde}. However, since both the metric problem in~\Eqref{eq:continuousFPT-HJB-eqn-for} and the finite-time horizon problem in~\Eqref{eq:HJ-eqn-nonviscous-time-dependent} have comparison principles~\citep{bardi_optimal_1997}, their proof goes through without much alteration. 
\end{remark}
% Minimum-time Hamiltonians in the metric-form are easily translated into stationary problems via the Kruzkov transform~\citep{bardi_optimal_1997}. 
%%% proof of variational formula
\section{Variational formula on $\R^d$ with $G=\Z^d$}
The argument we follow is again nearly identical to~\citep{lions_homogenization_2005}. But it does involve a few subtle changes to make it work, and this is interesting to write down. In any case, no one reads appendices, so it doesn't hurt to repeat an argument.Following~\citet{lions_homogenization_2005}, we begin with the approximate problem 
\begin{equation}
  \e v_{\e} + H(Dv_{\e},x) = 0 \quad \forall~ x \in \R^d.
  \label{eq:approximate-problem-to-obtain-corrector}
\end{equation}
From the variational interpretation of $v_{\e}$ in~\Eqref{eq:continuous-stationary-problem} and its dynamic programming principle, it follows that $v_{\e}$ is globally Lipschitz (uniformly in $\e$ and $\w$). Define the normalized set of functions 
\[ \hat{v}_{\e}(x) = v_{\e}(x) - v_{\e}(0). \] 
%Since $\hat{v}_{\e}$ are normalized to $0$ at the origin, 
Since $\hat{v}_{\e}$ is also Lipschitz and normalized to $0$ at the origin,
\begin{equation}
  C:= \sup_{\e} \left\{ \Norm{\hat{v}_{\e}(y)(1+|y|)^{-1}}{\infty} + \Norm{D\hat{v}_{\e}}{\infty} \right\} < \infty .
  \label{eq:uniform-bound-on-norms-of-vhat-epsilon}
\end{equation}
From the PDE~\Eqref{eq:approximate-problem-to-obtain-corrector}, it follows that functions $v_{\e}(x,\w)$ are stationary and hence have stationary, mean-zero increments. Hence, the normalized functions are in the set $S$ defined in~\Eqref{eq:lions-S-set-definition-translation-group-R}. 
%\begin{equation}
%%  \begin{split}
%%    \hat{v}_{\e}(x+y+g,\w) - \hat{v}_{\e}(x+g,\w) & = \hat{v}_{\e}(x+y,V_g\w) - \hat{v}_{\e}(x,V_g\w) \\
%    \hat{v}_{\e}(x+g,\w) - \hat{v}_{\e}(x,\w) = \hat{v}_{\e}(x,V_g\w),
%%  \end{split}.
%  \label{eq:v-epsilon-stationary-increments}
%\end{equation}
We're now ready to prove the variational formula in~\Propref{prop:lions-variational-characterization-of-h-bar} with $G=\Z^d$. In the following, all constants will be called $C$ and  might change value from line-to-line.
%One half of the proof is identical to~\citet{lions_homogenization_2005}, and the other requires a few additional arguments. 

\begin{proof}[Proof of~\Propref{prop:lions-variational-characterization-of-h-bar} with $G=\Z^d$]
 Denote the right side of~\Eqref{eq:variational-characterization-of-H-bar-discrete-translation-group} by $RHS$. Using the comparison principle for HJB equations,~\citet{lions_homogenization_2005} show that 
\[ \overline{H}(p) \leq RHS.\]
The same argument works for us.
% Let $f \in BUC(\R^d)$ be such that 
% \[ \esssup_x H(Df + p,x) \leq RHS + \delta, \]
% and consider the initial value problem
% \begin{equation*}
%   \left\{ 
%   \begin{split}
%     u_t + H(Du+p,x,\w) & = 0 \\
%     u(x,0) & = f(x)
%   \end{split} 
%   \right. .
% \end{equation*}
% It's clear that $f(x) - (RHS + \delta)t$ is a subsolution. Since the above equation has a comparison principle~\citep{crandall_users_1992},
% \[ u(x,t) \geq f(x) - (RHS + \delta)t.\]
% Scaling~\Thmref{thm:lions-homogenization-theorem}, we get that  
% \[ -\overline{H}(p) = \lim_{t \to \infty} \frac{u(x,t)}{t} \geq - (RHS + \delta),\]
% which implies that $\overline{H}(p) \leq RHS$.

 Consider the normalized approximating functions $\hat{v}_{\e}$ defined above. We will use these functions to construct functions in $S$ that give the other inequality. Using the optimal-control characterization of $H$ in~\Eqref{eq:continuous-hamiltonian-definition} (or plain old convexity), we get for fixed $a \in A$
  \begin{equation}
    \e \hat{v}_{\e}(x,\w) - f(x,a,\w)\cdot(p+D\hat{v}_{\e}) - l(x,a,\w) \leq - \e v_{\e}(0,\w) \quad x \in \R^d.
    \label{eq:inequality-for-approximating-sequence}
  \end{equation}

  We require some extra smoothness on $\hat{v}_{\e}$, and so we convolve it with the standard mollifier $\eta_{r}$, where $r$ is the size of its support. Let $\bar{v}_{\e} = \eta_{r} * \hat{v}_{\e}$, and let $f(x,a,\w)$ and $l(x,a,\w)$ have Lipschitz constant $C$ in $x$. For fixed $y$, multiply~\Eqref{eq:inequality-for-approximating-sequence} by $\eta_{r}(x-y)$ and integrate over $x$ to get 
\begin{align}
%  \int \eta_{r}(y-x) (\e \hat{v}_{\e}(x,\w) - f(x,a,\w)\cdot(p+D\hat{v}_{\e}) - l(x,a,\w) )~dx & \leq - \e v_{\e}(0,\w) ,\\
  \e \bar{v}_{\e}(y,\w) - f(y,a,\w)\cdot(p+D\bar{v}_{\e}) - l(y,a,\w)  - Cr & \leq - \e v_{\e}(0,\w). 
  \label{eq:inequality-smoothed-for-approximating-sequence}
\end{align}
 The mollified functions also satisfy the bound in~\Eqref{eq:uniform-bound-on-norms-of-vhat-epsilon}. 
%It's clear that
%\[ \Norm{\bar{v}_{\e}}{\infty} + \Norm{D\bar{v}_{\e}}{\infty} \leq C. \]
Moreover,
\[
%\begin{split}
  \left|D^2 \bar{v}_{\e}\right| = \left|D \eta_{r} * D \hat{v}_{\e}\right|  \leq \int \left|D\eta_{r}(y-x) D \hat{v}_{\e}(x)dx \right| \leq C(r).
%\end{split}
\]
We will take a weak limit (vague, to be precise) as $\e \to 0$ on the patch $[0,1]^d \times \W$, and then translate it using the group of translation operators $\{V_z\}_{z \in \Z^d}$ to obtain a function on $\R^d \times \W$. Consider the complete separable metric space $W := C^1([0,1]^d)$ with metric corresponding to the norm $\Norm{u}{} = \Norm{u}{\infty} + \Norm{|Du|}{\infty}$. 
The random functions $\bar{v}_{\e}(x,\w)$ are in the set
\[ K_{r}:= \left\{u(x) \in W : \Norm{u}{\infty} + \Norm{Du}{\infty} + \Norm{D^2u}{\infty}  \leq C + C(r) \right\}.\]
The set $K_{r}$ is compact in the metric space by the Arzela-Ascoli theorem. Then, the family $\{\overline{v}_{\e}\}_{\w > 0}$ is tight, and we can pass to a subsequence to obtain a weak limit $u_{r}(x,\w)$. Since $f(x,r,w)$ and $l(x,r,\w)$ are continuous, it follows that 
\[ f(x,a,\w)\cdot(p+Dv_{\e})  + l(x,a,\w) \overset{w}{\to} f(x,a,\w)\cdot(p+Du_{r})  + l(x,a,\w) \]
as $\e \to 0$ vaguely in $C([0,1]^d,\R)$. Hence, it follows from~\Eqref{eq:inequality-for-approximating-sequence} that for any fixed $\eta > 0$ and $r$ small enough,
\begin{equation*}
  -f(x,a,\w)\cdot(p+Du_{r})  - l(x,a,\w) \leq \overline{H}(p) + \eta \quad \forall~ x \in [0,1]^d \quad \almostsurely . 
\end{equation*}
Now, extend $u_{r}$ to all of $\R^d$ by defining $u_{r}(x+g,\w) = u_{r}(x,V_g\w)$. Take a sup over $r \in A$, followed by a sup over $x$ to get for arbitrary $\eta > 0$ 
\[
    \sup_x H(p+Du_{r},x,w) \leq \overline{H}(p) + \eta.
\]
Letting $\eta \to 0$ gives us the other inequality and completes the proof.
% I can include this argument if needed: it just describes what the set is.
%For fixed $\eta > 0$, consider the set
%\[
%G = \{ g \in W : -f(x,a,\w)\cdot(p+Dg) - l(x,a,\w) > \overline{H}(p) + \eta ~\forall~ x \in [0,1]^d \}.
%\]
%It's clear that $G$ is open in $W$. However, for $r$ sufficiently small, 
%\[  0 = \varliminf_{\e \to 0} \Prob ( \bar{v}_{\e} \in G )  \geq \Prob ( u_{r} \in G )\]
%by weak convergence. 
\end{proof}

\begin{remark}
  \label{rem:minimizer-for-variational-formula}
  When $G=\R^d$, there is a minimizer in $S$~\citep{lions_homogenization_2005}. When $G=\Z^d$, we don't have the estimates to prove this. 
\end{remark}

%% proof of commutation theorem

% I'll move a lot of the proofs to this section.
\section{Some proofs from~\Chapref{chap:comparing-two-distributions}}
\label{sec:some-proofs-from-comparing-two-distros}
\begin{proof}[Proof of~\Propref{prop:functions-measurable-wrt-sigma-tau}]
  \begin{claim}
    Let $F\colon\R^n \times \W \to \R$ be convex in its first variable, and bounded in its first argument on any compact subset of $\R^n$ uniformly in $\w$. For each fixed $p \in \R^n$, let $F(p,\w)$ be measurable with respect to a \sigmaalgebra~ $\mathcal{C} \subset \mathcal{F}$. If $h\colon\W \to \R^{n}$ is any bounded $\mathcal{F}$ measurable function,
    \[ 
      \esssup_{\w \in \W} F(E[h | \mathcal{C}],\w) \leq \esssup_{\w \in \W} F(h,\w) .
    \]
    \label{claim:about-conditional-expectation-convex-functions-and-esssup}
  \end{claim}
  We return to the proof of~\Claimref{claim:about-conditional-expectation-convex-functions-and-esssup} after completing the proof of the proposition. Let $\phi(x,\w) \in S$, where $S$ is defined in~\Corref{cor:discrete-set-S-for-variational-formula-with-lipschitz-constraint}. We apply~\Claimref{claim:about-conditional-expectation-convex-functions-and-esssup} with 
  \[
    \begin{split}
      h(\w) & = \mathcal{D}f = \left( \phi(x+\alpha) - \phi(x,\w) \right)_{\alpha \in A}, \\
      F(h,\w) & = \mathcal{H}(\mathcal{D}\phi,p,0,\w) ,\\
      \mathcal{C} & = \sigma(M)
    \end{split}
  \]
  where $\mathcal{H}$ is the discrete Hamiltonian in~\Eqref{eq:discrete-hamiltonian-1}, and $M$ is defined in~\Eqref{eq:sigma-algebra-C-generated-by-M}. \Claimref{claim:about-conditional-expectation-convex-functions-and-esssup} implies that
  \[
    \esssup_{\w} \mathcal{H}( E[\phi | \mathcal{C}] ,\w) \leq \esssup_{\w} \mathcal{H}( \phi(\w), \w).
  \]
  This means that we might as well take $\phi(x,\cdot)$ to be $\mathcal{C}$ measurable for every $x \in \Z^d$ in~\Thmref{thm:discrete-variational-formula-for-H-bar}. \Claimref{claim:about-conditional-expectation-convex-functions-and-esssup} remains to be proved and this is done below.
\end{proof}
% proof of the claim inside the proposition.
\begin{proof}[Proof of~\Claimref{claim:about-conditional-expectation-convex-functions-and-esssup}]
  We need a conditional version of Jensen's inequality which says that
  \begin{equation}
    E[F(h,\w) | \mathcal{C} ] \geq F(E[h | \mathcal{C}],\w) \quad \almostsurely .
    \label{eq:conditional-Jensens-inequality}
  \end{equation}

  For any constant $c$, suppose $A := \{ \w : F(E[h|\mathcal{C}],\w) \geq c \}$ has positive measure. The set $A$ is $\mathcal{C}$ measurable since both $F(p,\w)$ and $E[h|\mathcal{C}]$ are. By~\Eqref{eq:conditional-Jensens-inequality}, and the definition of conditional expectation
  \[
    E[F(h,\w),A] = E[F(E[h|\mathcal{C}],\w),A] \geq c.
  \]
  Hence, there is a subset of $A$ of positive measure where $F(h,\w) \geq c$. Letting $c$ approach $\esssup F(E[h|\mathcal{C}],\w)]$ completes the proof.

  It remains to prove~\Eqref{eq:conditional-Jensens-inequality}. We mollify $F$ with  $\eta_{\e}$ the standard mollifier on $\R^n$ with support in a ball of radius $\e$ to obtain a smooth function $F_{\e}$. Then, for any measurable functions $h$ and $h_0$ we have almost surely,
  \[
    F_{\e}(h,\w) \geq DF_{\e}(h_0,\w)\cdot(h-h_0) + F_{\e}(h_0,\w),
  \]
  Letting $h_0(\w) = E[h|\mathcal{C}]$, taking conditional expectation and using the fact that $DF_{\e}(h_0,\w)$ is $\mathcal{C}$ measurable, we get
  \[
    E[F_{\e}(h,\w)|\mathcal{C}] \geq F_{\e}(E[h(\w)|\mathcal{C}],\w).
  \]
  Finally letting $\e \to 0$, and using the boundedness of $h$ and the assumptions on $F$, we get~\Eqref{eq:conditional-Jensens-inequality}.
\end{proof}

%% file: thesis.bbl
\begin{thebibliography}{}

\bibitem[\protect\citeauthoryear{Akcoglu and Krengel}{Akcoglu and
  Krengel}{1981}]{akcoglu_ergodic_1981}
Akcoglu, M.~A. and U.~Krengel (1981).
\newblock Ergodic theorems for superadditive processes.
\newblock {\em Journal f{\"u}r die Reine und Angewandte Mathematik\/}~{\em
  323}, 53--67.

\bibitem[\protect\citeauthoryear{Armstrong, Cardaliaguet, and
  Souganidis}{Armstrong et~al.}{2012}]{armstrong_error_2012}
Armstrong, S.~N., P.~Cardaliaguet, and P.~E. Souganidis (2012, June).
\newblock Error estimates and convergence rates for the stochastic
  homogenization of hamilton-jacobi equations.
\newblock {\em {arXiv:1206.2601} [math]\/}.

\bibitem[\protect\citeauthoryear{Armstrong and Souganidis}{Armstrong and
  Souganidis}{2012}]{armstrong_stochastic_2012}
Armstrong, S.~N. and P.~E. Souganidis (2012, March).
\newblock Stochastic homogenization of level-set convex {Hamilton-Jacobi}
  equations.
\newblock {\em {arXiv:1203.6303} [math]\/}.

\bibitem[\protect\citeauthoryear{Auffinger and Damron}{Auffinger and
  Damron}{2013}]{auffinger_differentiability_2013}
Auffinger, A. and M.~Damron (2013).
\newblock Differentiability at the edge of the percolation cone and related
  results in first-passage percolation.
\newblock {\em Probability Theory and Related Fields\/}~{\em 156\/}(1-2),
  193–227.

\bibitem[\protect\citeauthoryear{Bardi and Capuzzo-Dolcetta}{Bardi and
  Capuzzo-Dolcetta}{1997}]{bardi_optimal_1997}
Bardi, M. and I.~Capuzzo-Dolcetta (1997).
\newblock {\em Optimal control and viscosity solutions of
  {Hamilton-Jacobi-Bellman} equations}.
\newblock Systems \& Control: Foundations \& Applications. Boston, {MA}:
  Birkhäuser Boston Inc.
\newblock With appendices by Maurizio Falcone and Pierpaolo Soravia.

\bibitem[\protect\citeauthoryear{Blair-Stahn}{Blair-Stahn}{2010}]{blair-stahn_first_2010}
Blair-Stahn, N.~D. (2010, May).
\newblock First passage percolation and competition models.
\newblock {\em {arXiv:1005.0649} [math]\/}.

\bibitem[\protect\citeauthoryear{Boivin}{Boivin}{1990}]{boivin_first_1990}
Boivin, D. (1990).
\newblock First passage percolation: the stationary case.
\newblock {\em Probability Theory and Related Fields\/}~{\em 86\/}(4),
  491--499.

\bibitem[\protect\citeauthoryear{Caffarelli, Souganidis, and Wang}{Caffarelli
  et~al.}{2005}]{caffarelli_homogenization_2005}
Caffarelli, L.~A., P.~E. Souganidis, and L.~Wang (2005).
\newblock Homogenization of fully nonlinear, uniformly elliptic and parabolic
  partial differential equations in stationary ergodic media.
\newblock {\em Communications on Pure and Applied Mathematics\/}~{\em 58\/}(3),
  319–361.

\bibitem[\protect\citeauthoryear{Chatterjee and Dey}{Chatterjee and
  Dey}{2013}]{chatterjee_central_2013}
Chatterjee, S. and P.~S. Dey (2013).
\newblock Central limit theorem for first-passage percolation time across thin
  cylinders.
\newblock {\em Probability Theory and Related Fields\/}~{\em 156\/}(3-4),
  613--663.

\bibitem[\protect\citeauthoryear{Cox and Durrett}{Cox and
  Durrett}{1981}]{cox_limit_1981}
Cox, J.~T. and R.~Durrett (1981).
\newblock Some limit theorems for percolation processes with necessary and
  sufficient conditions.
\newblock {\em The Annals of Probability\/}~{\em 9\/}(4), 583--603.

\bibitem[\protect\citeauthoryear{Crandall, Ishii, and Lions}{Crandall
  et~al.}{1992}]{crandall_users_1992}
Crandall, M.~G., H.~Ishii, and P.-L. Lions (1992).
\newblock User's guide to viscosity solutions of second order partial
  differential equations.
\newblock {\em American Mathematical Society. Bulletin. New Series\/}~{\em
  27\/}(1), 1--67.

\bibitem[\protect\citeauthoryear{Dal~Maso and Modica}{Dal~Maso and
  Modica}{1986}]{dal_maso_nonlinear_1986}
Dal~Maso, G. and L.~Modica (1986).
\newblock Nonlinear stochastic homogenization and ergodic theory.
\newblock {\em Journal f{\"u}r die Reine und Angewandte Mathematik\/}~{\em
  368}, 28--42.

\bibitem[\protect\citeauthoryear{Durrett and Liggett}{Durrett and
  Liggett}{1981}]{durrett_shape_1981}
Durrett, R. and T.~M. Liggett (1981).
\newblock The shape of the limit set in {Richardson}'s growth model.
\newblock {\em The Annals of Probability\/}~{\em 9\/}(2), 186–193.

\bibitem[\protect\citeauthoryear{Evans}{Evans}{1998}]{evans_partial_1998}
Evans, L.~C. (1998).
\newblock {\em Partial differential equations}.
\newblock Providence, {RI}: American Mathematical Society.

\bibitem[\protect\citeauthoryear{Fleming and Sheu}{Fleming and
  Sheu}{1985}]{fleming_stochastic_1985}
Fleming, W.~H. and S.~J. Sheu (1985).
\newblock Stochastic variational formula for fundamental solutions of parabolic
  {PDE}.
\newblock {\em Applied Mathematics and Optimization\/}~{\em 13\/}(3), 193--204.

\bibitem[\protect\citeauthoryear{Georgiou, Rassoul-Agha, and
  Sepp\"al\"ainen}{Georgiou et~al.}{2013}]{georgiou_variational_2013}
Georgiou, N., F.~Rassoul-Agha, and T.~Sepp\"al\"ainen (2013, November).
\newblock Variational formulas and cocycle solutions for directed polymer and
  percolation models.
\newblock {\em {arXiv:1311.3016} [math]\/}.

\bibitem[\protect\citeauthoryear{Gomes and Oberman}{Gomes and
  Oberman}{2004}]{gomes_computing_2004}
Gomes, D.~A. and A.~M. Oberman (2004).
\newblock Computing the effective {Hamiltonian} using a variational approach.
\newblock {\em {SIAM} Journal on Control and Optimization\/}~{\em 43\/}(3),
  792–812 (electronic).

\bibitem[\protect\citeauthoryear{Grimmett and Kesten}{Grimmett and
  Kesten}{2012}]{grimmett_percolation_2012}
Grimmett, G.~R. and H.~Kesten (2012, July).
\newblock Percolation since {Saint-Flour}.
\newblock {\em {arXiv:1207.0373} [math-ph]\/}.

\bibitem[\protect\citeauthoryear{Hammersley and Welsh}{Hammersley and
  Welsh}{1965}]{hammersley_first-passage_1965}
Hammersley, J.~M. and D.~J.~A. Welsh (1965).
\newblock First-passage percolation, subadditive processes, stochastic
  networks, and generalized renewal theory.
\newblock In {\em Proc. Internat. Res. Semin., Statist. Lab., Univ. California,
  Berkeley, Calif}, pp.\  61--110. New York: Springer-Verlag.

\bibitem[\protect\citeauthoryear{Johansson}{Johansson}{2000}]{johansson_shape_2000}
Johansson, K. (2000).
\newblock Shape fluctuations and random matrices.
\newblock {\em Communications in Mathematical Physics\/}~{\em 209\/}(2),
  437–476.

\bibitem[\protect\citeauthoryear{Kesten}{Kesten}{1986}]{kesten_aspects_1986}
Kesten, H. (1986).
\newblock Aspects of first passage percolation.
\newblock In {\em École d'été de probabilités de Saint-Flour,
  {XIV—1984}}, pp.\  125--264. Berlin: Springer.

\bibitem[\protect\citeauthoryear{Kingman}{Kingman}{1968}]{kingman_ergodic_1968}
Kingman, J. F.~C. (1968).
\newblock The ergodic theory of subadditive stochastic processes.
\newblock {\em Journal of the Royal Statistical Society. Series B.
  Methodological\/}~{\em 30}, 499--510.

\bibitem[\protect\citeauthoryear{Kosygina, Rezakhanlou, and Varadhan}{Kosygina
  et~al.}{2006}]{kosygina_stochastic_2006}
Kosygina, E., F.~Rezakhanlou, and S.~R.~S. Varadhan (2006).
\newblock Stochastic homogenization of {Hamilton-Jacobi-Bellman} equations.
\newblock {\em Communications on Pure and Applied Mathematics\/}~{\em
  59\/}(10), 1489--1521.

\bibitem[\protect\citeauthoryear{Krishnan}{Krishnan}{2013}]{krishnan_variational_2013-1}
Krishnan, A. (2013, November).
\newblock Variational formula for the time-constant of first-passage
  percolation {I}: {Homogenization}.
\newblock {\em {arXiv:1311.0316} [math]\/}.

\bibitem[\protect\citeauthoryear{Lions, Papanicolaou, and Varadhan}{Lions
  et~al.}{1987}]{lions_homogenization_1987}
Lions, P.-L., G.~Papanicolaou, and S.~R.~S. Varadhan (1987).
\newblock Homogenization of {Hamilton-Jacobi} equations.
\newblock {\em Unpublished preprint\/}.

\bibitem[\protect\citeauthoryear{Lions and Souganidis}{Lions and
  Souganidis}{2003}]{lions_correctors_2003}
Lions, P.-L. and P.~E. Souganidis (2003).
\newblock Correctors for the homogenization of {Hamilton-Jacobi} equations in
  the stationary ergodic setting.
\newblock {\em Communications on Pure and Applied Mathematics\/}~{\em
  56\/}(10), 1501--1524.

\bibitem[\protect\citeauthoryear{Lions and Souganidis}{Lions and
  Souganidis}{2005}]{lions_homogenization_2005}
Lions, P.-L. and P.~E. Souganidis (2005).
\newblock Homogenization of ``viscous'' {Hamilton-Jacobi} equations in
  stationary ergodic media.
\newblock {\em Communications in Partial Differential Equations\/}~{\em
  30\/}(1-3), 335--375.

\bibitem[\protect\citeauthoryear{Lions and Souganidis}{Lions and
  Souganidis}{2010}]{lions_stochastic_2010}
Lions, P.-L. and P.~E. Souganidis (2010).
\newblock Stochastic homogenization of {Hamilton-Jacobi} and “viscous”
  {Hamilton-Jacobi} equations with convex nonlinearities--revisited.
\newblock {\em Communications in Mathematical Sciences\/}~{\em 8\/}(2),
  627--637.

\bibitem[\protect\citeauthoryear{Marchand}{Marchand}{2002}]{marchand_strict_2002}
Marchand, R. (2002).
\newblock Strict inequalities for the time constant in first passage
  percolation.
\newblock {\em The Annals of Applied Probability\/}~{\em 12\/}(3), 1001–1038.

\bibitem[\protect\citeauthoryear{Newman}{Newman}{1995}]{newman_surface_1995}
Newman, C.~M. (1995).
\newblock A surface view of first-passage percolation.
\newblock In {\em Proceedings of the International Congress of Mathematicians,
  Vol.{\textbackslash} 1, 2 (Zürich, 1994)}, pp.\  1017–1023. Birkhäuser,
  Basel.

\bibitem[\protect\citeauthoryear{Oberman, Takei, and Vladimirsky}{Oberman
  et~al.}{2009}]{oberman_homogenization_2009}
Oberman, A.~M., R.~Takei, and A.~Vladimirsky (2009).
\newblock Homogenization of metric {Hamilton-Jacobi} equations.
\newblock {\em Multiscale Modeling \& Simulation\/}~{\em 8\/}(1), 269--295.

\bibitem[\protect\citeauthoryear{Rassoul-Agha and Sepp\"al\"ainen}{Rassoul-Agha
  and Sepp\"al\"ainen}{2014}]{rassoul-agha_quenched_2014}
Rassoul-Agha, F. and T.~Sepp\"al\"ainen (2014).
\newblock Quenched point-to-point free energy for random walks in random
  potentials.
\newblock {\em Probability Theory and Related Fields\/}~{\em 158\/}(3-4),
  711--750.

\bibitem[\protect\citeauthoryear{Rassoul-Agha, Sepp\"al\"ainen, and
  Yilmaz}{Rassoul-Agha et~al.}{2013}]{rassoul-agha_quenched_2013}
Rassoul-Agha, F., T.~Sepp\"al\"ainen, and A.~Yilmaz (2013).
\newblock Quenched free energy and large deviations for random walks in random
  potentials.
\newblock {\em Communications on Pure and Applied Mathematics\/}~{\em 66\/}(2),
  202--244.

\bibitem[\protect\citeauthoryear{Rezakhanlou and Tarver}{Rezakhanlou and
  Tarver}{2000}]{rezakhanlou_homogenization_2000}
Rezakhanlou, F. and J.~E. Tarver (2000).
\newblock Homogenization for stochastic {Hamilton-Jacobi} equations.
\newblock {\em Archive for Rational Mechanics and Analysis\/}~{\em 151\/}(4),
  277–309.

\bibitem[\protect\citeauthoryear{Rosenbluth}{Rosenbluth}{2008}]{rosenbluth_quenched_2008}
Rosenbluth, J.~M. (2008, April).
\newblock Quenched large deviations for multidimensional random walk in random
  environment: a variational formula.
\newblock {\em {arXiv:0804.1444} [math]\/}.

\bibitem[\protect\citeauthoryear{Schwab}{Schwab}{2009}]{schwab_stochastic_2009}
Schwab, R.~W. (2009).
\newblock Stochastic homogenization of {Hamilton-Jacobi} equations in
  stationary ergodic spatio-temporal media.
\newblock {\em Indiana University Mathematics Journal\/}~{\em 58\/}(2),
  537--581.

\bibitem[\protect\citeauthoryear{Sepp{\"a}l{\"a}inen}{Sepp{\"a}l{\"a}inen}{1998}]{seppalainen_exact_1998}
Sepp{\"a}l{\"a}inen, T. (1998).
\newblock Exact limiting shape for a simplified model of first-passage
  percolation on the plane.
\newblock {\em The Annals of Probability\/}~{\em 26\/}(3), 1232--1250.

\bibitem[\protect\citeauthoryear{Souganidis}{Souganidis}{1999}]{souganidis_stochastic_1999}
Souganidis, P.~E. (1999).
\newblock Stochastic homogenization of {Hamilton-Jacobi} equations and some
  applications.
\newblock {\em Asymptotic Analysis\/}~{\em 20\/}(1), 1--11.

\bibitem[\protect\citeauthoryear{van~den Berg and Kesten}{van~den Berg and
  Kesten}{1993}]{van_den_berg_inequalities_1993}
van~den Berg, J. and H.~Kesten (1993).
\newblock Inequalities for the time constant in first-passage percolation.
\newblock {\em The Annals of Applied Probability\/}~{\em 3\/}(1), 56--80.

\bibitem[\protect\citeauthoryear{Williams}{Williams}{1991}]{williams_probability_1991}
Williams, D. (1991).
\newblock {\em Probability with martingales}.
\newblock Cambridge: Cambridge University Press.

\end{thebibliography}
